\documentclass[12pt]{amsart}
\usepackage{amsfonts,amssymb, amscd, latexsym, graphicx, psfrag, color, float}

\usepackage[all,2cell]{xy}
\UseTwocells

\usepackage{mathrsfs}
\usepackage{eucal}

\usepackage[colorlinks,
             linkcolor=black,
             citecolor=blue,
             bookmarksnumbered,hyperindex,hyperfigures,pdfpagelabels]{hyperref}
\usepackage{todonotes}
\usepackage{tikz-cd}
\usepackage{scalerel,stackengine}
\stackMath
\newcommand\reallywidehat[1]{%
\savestack{\tmpbox}{\stretchto{%
  \scaleto{%
    \scalerel*[\widthof{\ensuremath{#1}}]{\kern-.6pt\bigwedge\kern-.6pt}%
    {\rule[-\textheight/2]{1ex}{\textheight}}
  }{\textheight}%
}{0.5ex}}%
\stackon[1pt]{#1}{\tmpbox}%
}
\usepackage{cleveref}

\crefformat{section}{\S#2#1#3} 
\crefformat{subsection}{\S#2#1#3}
\crefformat{subsubsection}{\S#2#1#3}

\newtheorem{dummy}{dummy}[section]
\newtheorem{lemma}[dummy]{Lemma}
\newtheorem{theorem}[dummy]{Theorem}
\newtheorem*{theorem*}{Theorem}

\newtheorem{construction}[dummy]{Construction}
\newtheorem{corollary}[dummy]{Corollary}
\newtheorem{proposition}[dummy]{Proposition}
\theoremstyle{definition}
\newtheorem{definition}[dummy]{Definition}
\newtheorem{example}[dummy]{Example}
\newtheorem{remark}[dummy]{Remark}

\newcommand{\bA}{\mathbb{A}}
\newcommand{\bC}{\mathbb{C}}

\newcommand{\bP}{\mathbb{P}}

\newcommand{\bZ}{\mathbb{Z}}

\newcommand{\Grass}{\mathrm{Grass}}
\newcommand{\KGL}{\mathrm{KGL}}

\newcommand{\B}{\mathsf{B}}
\newcommand{\G}{\mathsf{G}}

\newcommand{\cB}{\mathcal{B}}
\newcommand{\cC}{\mathcal{C}}

\newcommand{\cE}{\mathcal{E}}
\newcommand{\cF}{\mathcal{F}}
\newcommand{\cG}{\mathcal{G}}

\newcommand{\cI}{\mathcal{I}}

\newcommand{\cO}{\mathcal{O}}

\newcommand{\conn}{\mathrm{cn}}
\newcommand{\cS}{\mathcal{S}}

\newcommand{\cU}{\mathcal{U}}
\newcommand{\cV}{\mathcal{V}}
\newcommand{\cX}{\mathcal{X}}
\newcommand{\cY}{\mathcal{Y}}
\newcommand{\cZ}{\mathcal{Z}}

\newcommand{\Mon}{\mathsf{Mon}}
\newcommand{\CMon}{\mathsf{CMon}}
\renewcommand{\S}{\cS\!\operatorname{pc}}

\newcommand{\Shape}{\operatorname{Shape}}
\newcommand{\is}{i_{\Spc_{S,\bA^1}}}
\newcommand{\isa}{i_{S_{\bA^1}}}
\newcommand{\LL}{\mathrm{L}}
\newcommand{\exc}{\mathsf{exc}}
\newcommand{\mot}{\mathsf{mot}}
\newcommand{\Fin}{\mathrm{Fin}}

\newcommand{\sC}{\mathscr{C}}
\newcommand{\sD}{\mathscr{D}}
\newcommand{\sE}{\mathscr{E}}
\newcommand{\sF}{\mathscr{F}}

\newcommand{\sX}{\mathscr{X}}
\newcommand{\sY}{\mathscr{Y}}

\newcommand{\ft}{\mathsf{ft}}

\newcommand{\fin}{\mathsf{fin}}

\newcommand{\Nis}{\mathrm{Nis}}

\newcommand{\Gal}{\operatorname{Gal}}
\newcommand{\Et}{\operatorname{\acute{E}t}}
\newcommand{\DF}{\mathrm{DF}}

\newcommand{\red}{\mathrm{red}}
\newcommand{\Aff}{\mathbf{Aff}}

\newcommand{\Sm}{\mathbf{Sm}}
\newcommand{\Cor}{\mathbf{Cor}}

\newcommand{\Topi}{\mathfrak{Top}_\i}
\newcommand{\Topiy}{\mathfrak{Top}_{\i /\sY}}

\newcommand{\Orb}{\mathsf{Orb}}

\newcommand{\perf}{\mathsf{perf}}

\newcommand{\lex}{\mathsf{lex}}
\newcommand{\stab}{\mathsf{stab}}

\newcommand{\Spec}{\mathrm{Spec}\,}

\newcommand{\Aut}{\operatorname{Aut}}
\newcommand{\st}{\mathrm{st}}
\newcommand{\LKE}{\operatorname{Lan}}

\newcommand{\PShv}{\operatorname{PShv}}
\newcommand{\Shv}{\operatorname{Shv}}

\newcommand{\Cat}{\operatorname{Cat}}
\newcommand{\Mod}{\operatorname{Mod}}
\newcommand{\Spc}{\S}
\newcommand{\Spt}{\cS\!\operatorname{pt}}
\renewcommand{\CMon}{\cC\!\operatorname{Mon}}
\newcommand{\CAlg}{\cC\!\operatorname{Alg}}
\newcommand{\Exc}{\cE\!\operatorname{xc}}
\newcommand{\coExc}{\mathrm{co}\Exc}
\newcommand{\cSpt}{\mathrm{co}\Spt}
\newcommand{\SH}{\cS\!\operatorname{H}}

\newcommand{\Adj}[4]{\xymatrix@1{#2 \ar@<-0.5ex>[r]_-{#4} & #3 \ar@<-0.5ex>[l]_-{#1}}}
\newcommand{\Adjlong}[4]{\xymatrix@C=2cm{#2 \ar@<-0.65ex>[r]_-{#4} & #3 \ar@<-0.65ex>[l]_-{#1}}}

\newcommand{\Hom}{\mathbb{H}\mathrm{om}}
\newcommand{\map}{\mathbb{M}\mathrm{ap}}

\newcommand{\Perf}{\mathcal{P}\mathrm{erf}}

\newcommand{\op}{\mathrm{op}}

\newcommand{\id}{\mathrm{id}}

\newcommand{\et}{\acute{e}t}
\newcommand{\sep}{\mathrm{sep}}
\newcommand{\cofilt}{\mathrm{cofilt}}

\def\Top{\mathbf{Top}}
\def\Sch{\mathbf{Sch}}

\renewcommand{\i}{\infty}

\def\Pro{\operatorname{Pro}}
\def\Ind{\operatorname{Ind}}
\def\Fun{\operatorname{Fun}}

\def\colim{\underrightarrow{\mbox{colim}\vspace{0.5pt}}\mspace{4mu}}

\renewcommand{\lim}{\varprojlim\mspace{3mu}}
\def\blank{\mspace{3mu}\cdot\mspace{3mu}}

\def\Prof{\operatorname{Prof}}

\def\Lan{\operatorname{Lan}}

\def\Top{\mathbf{Top}}
\def\Set{\mathrm{Set}}
\def\St{\mathrm{St}}

\def\Htop{\Topi^{\mathit{Hens.}}}

\def\Aut{\mathbf{Aut}}

\def\cs{\widehat{\left(\blank\right)}_{\Spc_{S,\bA^1}}}

\def\longlongrightarrow{-\!\!\!-\!\!\!-\!\!\!-\!\!\!-\!\!\!-\!\!\!\longrightarrow}
\def\longlonglongrightarrow{-\!\!\!-\!\!\!-\!\!\!-\!\!\!-\!\!\!-\!\!\!\longlongrightarrow}

\newcommand*{\longhookrightarrow}{\ensuremath{\lhook\joinrel\relbar\joinrel\rightarrow}}
\newcommand*{\longlonghookrightarrow}{\ensuremath{\lhook\joinrel\relbar\joinrel\relbar\joinrel\rightarrow}}

\begin{document}

\title[\resizebox{6in}{!}{Relative \'{E}tale Realizations  and Dwyer-Friedlander $K$-Theory of NC Schemes}]{Relative \'{E}tale Realizations of Motivic Spaces and Dwyer-Friedlander $K$-Theory of Noncommutative Schemes}

\author{David Carchedi}
\author{Elden Elmanto}

\maketitle
\begin{abstract} In this paper, we construct a refined, relative version of the \'etale realization functor of motivic spaces, first studied by Isaksen and Schmidt. Denoting by $\Spc\left( S \right)$ the $\infty$-category of motivic spaces over a base scheme $S$, their functor goes from $\Spc\left( S \right)$ to the $\infty$-category of $p$-profinite spaces, where $p$ is a prime which is invertible in all residue fields of $S$. 

In the first part of this paper, we refine the target of this functor to an $\infty$-category where $p$-profinite spaces is a further completion. Roughly speaking, this $\infty$-category is generated under cofiltered limits by those spaces whose associated ``local system" on $S$ is $\bA^1$-invariant. We then construct a new, relative version of their \'etale realization functor which takes into account the geometry and arithmetic of the base scheme $S$. For example, when $S$ is the spectrum of a field $k$, our functor lands in a certain $\Gal\left(k^{\sep}/k\right)$-equivariant $\infty$-category. Our construction relies on a relative version of \'etale homotopy types in the sense of Artin-Mazur-Friedlander, which we also develop in some detail, expanding on previous work of Barnea-Harpaz-Schlank. 

We then stabilize our functor, in the $S^1$-direction, to produce an \'etale realization functor for motivic $S^1$-spectra (in other words, Nisnevich sheaves of spectra which are $\bA^1$-invariant). To this end, we also develop an $\infty$-categorical version of the theory of profinite spectra, first explored by Quick. As an application, we refine the construction of the \'etale $K$-theory of Dwyer and Friedlander, and define its non-commutative extension. This latter invariant should be seen as an $\ell$-adic analog of Blanc's theory of semi-topological $K$-theory of non-commutative schemes. We then formulate and prove an analog of Blanc's conjecture on the torsion part of this theory, generalizing the work of Antieau and Heller.

\end{abstract}

\tableofcontents

\section{Introduction}

\subsubsection{Motivation}

Realization functors are of both practical and theoretical importance in the theory of motives. For example, one often wants to check that a map $f: M \rightarrow N$ of motives is \emph{nontrivial}, i.e., it \emph{does not} factor through the terminal object as $f: M \rightarrow * \rightarrow N$. If the realization functor preserves terminal objects, and it often does, then the realized map being nontrivial implies that the original map must have been nontrivial. This latter check is often easier to perform. 

%

This simple strategy underlies many of the most successful applications of the theory of motives and motivic homotopy theory. For example, in Voevodsky's original plan of attack on the Milnor's conjecture \cite{voevodsky-cobordism}, he proves that a certain map in the motivic stable category is nonzero (see \cite[Proposition 3.24]{voevodsky-cobordism}). In order to do so, he exploits the Betti realization functor together with techniques from classical cobordism theory. The realization strategy also permeates all other approaches to the Milnor/Bloch-Kato conjectures by means of Voevodsky's computation of the motivic Steenrod algebra over fields of characteristic zero \cite{voevodsky2010motivic}, \cite{voevodsky2003reduced}. This computation relies heavily on the Betti realization functor to show that the motivic Steenrod operations span all the operations in motivic cohomology, leveraging the analogous classical statement in topology. An analogous computation in characteristic $p$ (for the mod-$\ell$ motivic Steenrod algebra where $\ell$ is prime to $p$) was carried out in \cite{hoyois2013motivic}, relying on the $\ell$-adic realization functor, going from motivic spaces to the derived category of $\ell$-adic sheaves.

This latter functor is a homological shadow of a homotopical construction due to Isaksen \cite{realization} and Schmidt \cite{schmidt}. Fixing a prime $\ell$ which is coprime to the residue characteristics of a base scheme $S$, Isaksen and Schmidt independently extends the functor which associates to a smooth $S$-scheme the $\ell$-completion of its \'etale homotopy type (in the sense of Artin-Mazur-Friedlander \cite{ArtinMazur}, \cite{Friedlander}) to all motivic spaces over $S$. Over an algebraically closed field, Quick \cite{Qu3} constructed a $\bP^1$-stable version of this \'{e}tale realization functor, going from the category of motivic spectra $\SH(k)$ to a certain completion of the category of profinite spectra. These functors have found further applications in investigating torsion algebraic cycles \cite{Qu4}, sums of squares formulas \cite{duggerisaksen2}, and desuspensions of motivic spaces \cite{Wickelgren}.

\subsubsection{What is done in this paper?} In this paper, we offer a \emph{refinement} and, more interestingly, a \emph{relative version} of the functors of Isaksen and Schmidt. What allows us to do this is, unsurprisingly, the availability of modern technology in the form of $\infty$-categories as invented by Boardman and Vogt \cite{BVogt}, and later developed greatly by Joyal \cite{Joyal} and Lurie \cite{htt}. 

To state our theorem more precisely, we start with a base scheme $S$ and denote by $S_{\bA^1} \subset \Shv\left(S_{\et}\right)$ the full subcategory of the small \'etale $\infty$-topos on $S$ spanned by objects which are $\bA^1$-invariant in the sense of Definition~\ref{defn:a1s}. This inclusion admits a left adjoint, so we may speak of $S_{\bA^1}$-localizations. In Section~\ref{subsec:examples}, we show that this $\infty$-category contains a large class of interesting objects such as Eilenberg-Maclane objects associated to torsion sheaves whose torsion is prime to the residue characteristics of $S$. We prove:

\begin{theorem} \label{thm:main1} Let $S$ be a scheme, then there exists a unique colimit-preserving functor 
\begin{equation} \label{eq:main-functor}
\Et^S_{\bA^1}: \Spc\left( S \right) \rightarrow \Pro\left(S_{\bA^1}\right)
\end{equation} whose value on the motivic localization of a smooth $S$-scheme $X$ is given by the $S_{\bA^1}$-localization of the $S$-relative \'etale homotopy type of $X$: $$\widehat{\Pi^{S,\et}_\i\left(X\right)}_{S_{\bA^1}}.$$ 
\end{theorem}

We note the following specializations of our Theorem~\ref{thm:main1}

\begin{enumerate}
\item If $\ell$ is a prime which is invertible in the residue characteristics of $S$, then there is a functor $\Pro\left(S_{\bA^1}\right) \rightarrow \Prof_{\ell}(\Spc)$, where $\Prof_{\ell}(\Spc)$ is the $\infty$-category of $\ell$-profinite spaces. The functors of Isaksen and Schmidt takes the form $$\Et_{\ell}: \Spc\left( S \right) \rightarrow \Prof_{\ell}(\Spc),$$ and the functor of Theorem~\ref{thm:main1} factors $\Et_{\ell}$ through $\Pro\left(S_{\bA^1}\right)$.
\item If $S$ is the spectrum of a field $k$, then $\Pro\left(S_{\bA^1}\right)$ is a full subcategory of  $\cB \Gal\left(k^{\sep}/k\right)$, the classifying $\i$-topos for the absolute Galois group. Hence our functor lands in an $\infty$-category that takes into account the action of the Galois group of $k$. More precisely, our realization functor produces an inverse system of spaces \[\{ X_{k_i} \}_{k_i/k\,\text{finite Galois extensions}} \] where each $X_{k_i}$ is a space equipped with the action of $\Gal\left(k_i/k\right)$ and the transition maps are appropriately coherently equivariant.
\end{enumerate}

Another feature of the functor~\eqref{eq:main-functor} that we find pleasant is that it admits an explicit formula as a pro-object. As we recall in subsection~\ref{subsec:proobj}, if $\sC$ is accessible and  has finite limits, then $\Pro \left( \sC \right)$ can be described as the (opposite) $\infty$-category of functors $F: \sC \rightarrow \Spc$ which preserves finite limits. Given a motivic space $X$ we will describe $\Et^S_{\bA^1}\left( X \right)$ explicitly as a functor from $S_{\bA^1} \rightarrow \Spc$. This feature is a direct consequence of the origins of the functor~\eqref{eq:main-functor} --- the \emph{relative} \'etale realization functor which we now discuss.

\subsubsection{Relative \'etale realization} The relative \'etale realization functor is of the form $$\Pi^{S,\et}_{\i}:\Shv_{\et}\left(\Sch^{\ft}_S\right) \to \Pro\left(\Shv\left(S_{\et}\right)\right).$$ We discuss its constructions and properties in detail in Section~\ref{sect:relative-types-stacks}. The main technical input to our constructions comes from the first author's paper \cite{dave-etale}, where the Artin-Mazur-Friedlander construction of \'{e}tale homotopy types is promoted to a homotopy coherent construction from higher stacks to pro-spaces; this paper provides a relative version of these constructions, expanding on work of Harpaz-Schlank \cite{harpaztomer} and Barnea-Schlank \cite{barneatomer}. We also provide an explicit description of the relative \'etale homotopy type of an $S$-scheme (more generally of an object in $\Shv_{\et}\left(\Sch^{\ft}_S\right)$, i.e., a higher stack on S) as a left exact functor $\Shv\left(S_{\et}\right) \to \Spc$ in Theorem~\ref{thm:concrete_relative}:
\begin{theorem}\label{thm:main2}
Let $X$ be an $S$-scheme. Viewed as a left exact functor $$\Shv\left(S_{\et}\right) \to \Spc,$$ the $S$-relative \'etale homotopy type $\Pi^{S,\et}_{\i}\left(X\right)$ canonically identifies with the functor $$\map_{\Shv_{\et}\left(\Sch^{\ft}_S\right)}\left(\cX,\rho_!\left(\mspace{3mu}\cdot\mspace{3mu}\right)\right),$$ where $\rho_!$ is the fully faithful embedding of sheaves on the small \'etale site into $\Shv_{\et}(\Sch^{\ft}_S)$.
\end{theorem} 
This explicit description also persists to the level of motivic spaces, and we prove this in Proposition~\ref{prop:explicit-a1}.

\subsubsection{Stabilization and Dwyer-Friedlander $K$-theory} The eventual goal for a realization functor is to construct one from the stable motivic $\i$-category $\SH(S)$ to some stable version of pro-objects in \'etale sheaves. When $k$ is separably closed, Quick uses the K\"unneth formula in order to construct a functor to the category of profinite spectra. Without such a formula, it seems that a direct approach is out of reach; in fact there are real limitations to the existence of such functors as discussed in \cite{Kass:2017aa}. With our eyes towards this end of goal, we revisit the theory of profinite spectra as discussed by Quick in \cite{Qu} and \cite{Qu3} in the setting of $\infty$-categories. We find this ``model-independent" treatment quite convenient. In particular, regarding the $\infty$-category of spectra as reduced excisive functors buys us a lot of mileage in setting up this theory, as well as in developing stable versions of the various localizations of pro-spaces. One of the more amusing features of our treatment is the language of \emph{cospectra/cospectrum objects}; this notion was actually discussed in one of the earliest papers about spectra in \cite{lima}.

We then stabilize our \'etale realization functor, in the $S^1$-direction, obtaining a colimit-preserving stable \'etale realization functor (Theorem~\ref{prop:stab-et-real}) from the stable $\infty$-category of \'etale sheaves of spectra on the big site on $S$ to the stabilization of the $\infty$-category of pro-objects on the small \'etale site.
\begin{equation} \label{functor-stab-main}
\Pi^{S,\et}_{\i,\Spt}: \Shv_{\et,\Spt}\left( \Sch^{\ft}_S \right) \rightarrow \Spt\left(\Pro\left(\Shv\left(S_{\et}\right)\right) \right).
\end{equation}
It is then formal to stabilize the functor of~\eqref{eq:main-functor} (Theorem~\ref{thm:stable}) to obtain one coming from the $S^1$-stable motivic homotopy $\infty$-category
\begin{equation} \label{functor-stab-et}
\Et^S_{\bA^1,\Spt}: \SH^{S^1}\left(S \right) \rightarrow \Spt \left( \Pro\left(S_{\bA^1}\right) \right).
\end{equation}
As an application, we reformulate \'etale $K$-theory, in the sense of Dwyer and Friedlander and extend it to the non-commutative setting. In this reformulation, the theory looks closer to Friedlander-Walker's theory of semi-topological $K$-theory  \cite{fw1}, \cite{fw2} and its non-commutative generalization due to \cite{blanc}. In modern language, semi-topological $K$-theory is calculated by applying the stable Betti realization to a Nisnevich sheaf of spectra on $\Sch_{\bC}^{\ft}$ given by ``$\sC$-twisted $K$-theory" where $\sC$ is a $\bC$-linear dg-category; see Construction~\ref{constrct:twist} for a precise definition of this sheaf in a general context. This Nisnevich sheaf of spectra is denoted by $\underline{K}\left(\sC\right)$. The comparison with the original definitions of Friedlander-Walker is given in \cite[Theorem 2.3]{antieau-heller}. 

Using the functor of~\eqref{functor-stab-main} when $S = \Spec\,R$, we define (Definition~\ref{def:df-k}) the \emph{relative Dwyer-Friedlander $K$-theory} of an $R$-linear stable $\infty$-category $\sC$ to be 
\[
K^{\DF}_R\left(\sC\right):= \Pi^{S,\et}_{\i,\Spt}\underline{K}\left( \sC \right),
\]
and its $\bA^1$-invariant analog to be 
\[
K^{\DF}_{R,\bA^1}\left(\sC\right) := \Et^S_{\bA^1,\Spt}\LL_{\bA^1}\underline{K}\left(\sC\right),
\]
where $\LL_{\bA^1}$ is the $\bA^1$-localization functor. One of the payoffs of the relative nature of our constructions is that both invariants are again sensitive to the arithmetic and geometry of the base scheme; for example if $S$ is the spectrum of a field, then $K^{\DF}_S\left(\sC \right)$ admits a natural action of the Galois group.

The original motivation for semi-topological $K$-theory is to interpolate between topological and algebraic $K$-theory, retaining the finer information about algebraic varieties than algebraic $K$-theory sees (but topological $K$-theory does not) while discarding inaccessible ``transcendental" information that algebraic $K$-groups possess (recall that the positive algebraic $K$-groups of $\Spec\,\bC$ are uncountable). We refer to the introduction of \cite{fw2} for a discussion of this theory and how many deep conjectures about algebraic cycles can be reformulated with semi-topological $K$-theory. As mentioned in page 880 \emph{loc. cit.}, there was no ``reasonable definition" of semi-topological $K$-theory over arbitrary fields at that time of that article; it is our hope that our construction now offers a reasonable candidate. 

The signs seems encouraging; one of the desiderata of a semi-topological theory is that it should agree with the original theory up to torsion which is prime to the base field as discussed in \cite[Section 1.5]{fw2} and made precise in \cite[Theorem 30]{fw2}. For semi-topological $K$-theory of dg-categories, this was conjectured by Blanc in \cite{blanc} and recently proved by Antieau-Heller in \cite{antieau-heller}. We prove this for our invariant (Theorem~\ref{thm:agreement}):

\begin{theorem} Let $k$ be a field, $\sC$ a $k$-linear stable $\infty$-category. Let $e^*K\left(\sC\right) \in \Spt\left(\Pro\left(\Shv\left(k_{\et}\right)\right) \right)$ be the constant spectrum object in $\Pro\left(\Shv\left(k_{\et}\right)\right)$ associated to the spectrum $K\left(\sC \right)$. Then there is canonical map
\[
e^*K\left(\sC\right)  \rightarrow K^{\DF}_k\left( \sC \right),
\]
which is an equivalence mod-$n$ where $n$ is invertible in $k$.
\end{theorem}

In a sequel to this paper, we plan to develop this invariant further by constructing Chern character maps to periodic cyclic homology, its descent properties, and its relationship with recent developments in topological cyclic homology \cite{nikolaus-scholze} and the approach to $\ell$-adic $K$-theory of categories of Clausen \cite{Clausen:2017aa} via topological cyclic homology.

\subsubsection{The landscape of realization functors} The past few years has seen a proliferation of realization functors in motivic homotopy theory. We will give a (non-exhaustive!) list and mention how our work, and its sequel will fit in.

\begin{enumerate}
\item Morel and Voevodsky introduced two kinds of unstable realization functors in \cite[Section 3.3]{morel-voevodsky}. The first is the \emph{Betti realization functor} which is of the form \[\mathrm{Be}: \Spc^{\bA^1}_{\mathbb{C}} \rightarrow \Spc.\] It takes the motivic space of a smooth $\mathbb{C}$-scheme and sends it to its analytification 
\item If $k$ is a field of characteristic zero, we can base change to $\mathbb{C}$ and then compose with the Betti realization functor to get a functor $\mathrm{Be}: \Spc^{\bA^1}_{k} \rightarrow \Spc$. The Betti realization functor is strong symmetric monoidal and takes $\bP^1$ to $S^2$ and hence, by general formalism, stabilizes to a functor $\mathrm{Be}: \SH(k) \rightarrow \Spt$. This stabilized functor has been described and studied by a number of authors including Voevodsky \cite[Section 3.3]{voevodsky2010motivic}, Riou \cite{MR2439197} and Panin, Pimenov and R\"ondigs \cite{ppr}. 
\item A relative Betti realization for a $\mathbb{C}$-scheme $X$ which takes the form $\SH(X) \rightarrow \sD(X^{\mathrm{an}})$ has been introduced and studied extensively by Ayoub; here $\sD(X^{\mathrm{an}})$ is the derived category of complexes of abelian groups on $X^{\mathrm{an}}$ the analytic space associated to $X$. In \emph{loc. cit} Ayoub also proves the compatibility of Betti realization with the six operations on both sides. A further study of a version of this functor that takes into account the Hodge theory of $X$ has been introduced and studied by Drew in \cite{Drew:2018aa}.
\item The second functor defined by Morel and Voevodsky in \cite[Section 3.3]{morel-voevodsky} takes the form $\mathrm{Be}_{C_2}: \Spc^{\bA^1}_{\mathbb{R}} \rightarrow \Spc_{C_2}$, where $\Spc_{C_2}$ indicates the category of $C_2$-spaces. It takes the motivic space of a smooth $\mathbb{R}$-scheme and sends it to its analytification equipped with the action by complex conjugation. On the level of model categories, this functor lands in the model category of $C_2$-spaces where the equivalences are detected on $C_2$-fixed points, i.e., the model category of genuine $C_2$-spaces.
\item If $k$ is a field which admits an embedding to $\mathbb{R}$, we can base change to $\mathbb{R}$ and compose with the real Betti realization to get a functor $\mathrm{Be}_{C_2}: \Spc^{\bA^1}_{k} \rightarrow \Spc_{C_2}$. Heller and Ormsby \cite{heller-ormsby} stabilized the construction above to obtain the \emph{real Betti realization functor} which is of the form $\mathrm{Be}_{C_2}: \SH(k) \rightarrow \Spt_{C_2}$ where $\Spt_{C_2}$ is the $\infty$-category of \emph{genuine} $C_2$-spectra.
\item We have already mentioned that Isaksen in \cite{realization} and Schmidt \cite{schmidt} constructed \'etale realization functors for fields of arbitrary characteristics landing in spaces completed away from the characteristic. 
\item When $k$ is separably closed, work of Quick \cite{Qu3} stabilizes the \'etale realization functor to obtain a functor $\Et_{\ell}: \SH(k) \rightarrow \Prof_{\ell}(\Spt)$ where $\Prof_{\ell}(\Spt)$ is the category of $\ell$-profinite spectra (we review this formalism in the language of this paper in Section~\ref{sect:stable-prof}).
\item When $k$ is an arbitrary field, the best hope for an \'etale realization functor would be one that is strong symmetric monoidal, lands in the $\infty$-category of genuine $\Gal\left( k^\sep/k \right)$-spectra \footnote{At least a symmetric monoidal stable $\infty$-category such that $\pi_0$ of endomorphism of the unit object is the Burnside ring of $\Gal\left( k^\sep/k \right)$}. The work of Kass and Wickelgren \cite{Kass:2017aa} tells us that this hope is \emph{not possible}.
\item If one contends with \emph{homological} versions of the realization functor, i.e., one that lands in an $R$-linear stable $\infty$-category for some discrete ring $R,$ then the results are as good as it gets. Work of Ayoub \cite{MR3205601} and Cisinski-D\'eglise \cite{etalemotives} gives realization functors from $\SH(X)$ landing in the classical derived category of \'etale sheaves (with appropriate coefficients) on $X$ compatible with the full six functors formalism and the vanishing/nearby cycles formalism. The main input to their work is a version of \emph{Suslin rigidity} in the style of \cite{MR1764197}, which is not available in the spectral setting. However, there has been recent significant breakthroughs made by Tom Bachmann in this direction, and the second author is investigating the same results with a different technique; it would be interesting to understand the implication of these results in our setting.
\end{enumerate}

In view of the above discussion, our work in Sections~\ref{sec:relativeetale} and~\ref{sect:relreal} is generalization of (6). In view of the limitations discussed (8), Section~\ref{sect:stable} is an attempt to get as far as possible with a stable version of the \'etale realization functor, namely, we can construct a relative version of the \'etale realization functor if we only stabilize in the $S^1$-direction. An \'etale realization originating from $\SH(k)$ is the subject of future work. In spite of the limitations in (8), we expect our functor to be strong symmetric monoidal and lands in an stable $\infty$-category which takes into account the geometry and arithmetic of the base. Furthermore, this functor should refine the homological realization functors discussed in (9). 

\subsection{Overview} We will now give a linear overview of this paper. In the preliminary Section~\ref{sec:prelim}, written mainly for the reader's convenience, we review pro-objects in $\infty$-categories in subsection~\ref{subsec:proobj}, and the notions of shapes of  $\infty$-topoi and \'etale homotopy types in subsections~\ref{subsec:absshape} and~\ref{subsec:et-type} respectively. We also set up our notation and conventions on motivic homotopy theory in subsection~\ref{subsec:a1}. 

 In Section~\ref{sec:relativeetale}, we revisit the absolute \'etale realization functor. We begin this subsection~\label{subsec:loc} by defining various localizations of the $\infty$-category of pro-spaces. We define a localization which is the ``finest" target for an absolute \'etale realization from motivic spaces and show that previous considered localizations are just further localizations of this $\infty$-category. We define our absolute \'etale realization functor in subsection~\ref{subsec:absolute}. 
 
 We finally come to relative realizations in Section~\ref{sect:relreal}. We first develop the basic theory of relative shapes in subsection~\ref{subsec:relshape} and define the relative \'etale homotopy type of a schemes and, more generally, a higher algebraic stack. In subsection~\ref{subsec:explicit}, we provide an explicit formula for the relative \'etale homotopy type, generalizing the explicit formula for the \'etale homotopy type of a higher stack given in \cite{dave-etale}. In subsection~\ref{subsec:rel-real} we come to the main construction of this paper: the relative \'etale realization functor for motivic spaces.
 
In the last Section~\ref{sect:stable}, we stabilize the story, at least in the $S^1$-direction. Subsection~\ref{sect:stable-prof} offers a ``model-independent" treatment of stable profinite homotopy theory. Using this framework, we stabilize the relative \'etale realization functor to one from $S^1$-motivic spectra. The next two subsections offer two applications of our theory. The first, in subsection~\ref{subsec:revisit}, we offer a refinement of Dwyer and Friedlander's \'etale K-theory. The second, in subsection~\ref{subsec:noncomm}, we define the Dwyer-Friedlander $K$-theory of an $R$-linear stable $\infty$-category. Using representability of algebraic $K$-theory from motivic homotopy theory, we give a computation of the $\bA^1$-local version of this theory. In subsection~\ref{sec:blancdf}, we prove that, with finite coefficients coprime to the base field, the Dwyer-Friedlander $K$-theory of a $k$-linear stable $\infty$-category agrees with usual $K$-theory. This is an analog in our setting of a conjecture of Blanc and a theorem of Antieau-Heller.

\subsection{Conventions} The current work uses heavily the language of $\infty$-categories. By an $\infty$-category, we mean an $\left(\infty, 1\right)$-category and when pressed, we use the model of quasicategories. We borrow heavily terminology and notation from \cite{htt}. For the reader's convenience, here are some of our conventions:

\begin{enumerate}

\item The $\i$-category of spaces/homotopy types/$\i$-groupoids is denoted by $\Spc$
\item Given an $\infty$-category $\sC$ we write $\sC_0$ for its objects ($0$-simplices) and write $\map_{\sC}\left( X, Y \right)$ for the mapping space between two objects and write $\map\left(X,Y\right)$ if the context is clear.
\item The $\i$-category of presheaves of spaces on an $\infty$-category $\sC$ is denoted by $$\PShv\left(\sC\right) := \Fun\left( \sC^{\op}, \Spc\right).$$ Suppose that $\sC$ has a Grothendieck topology \cite[Section 6.2.2]{htt} $\tau$, then the $\infty$-category of sheaves of spaces with respect to $\tau$ \cite[Definition 6.2.2.6]{htt} is denoted by $\Shv_{\tau}\left(\sC\right)$. For example if $S$ is a base scheme, then the $\infty$-category of sheaves (of spaces) on the (big) \'{e}tale site $\Sch_S$ is denoted by $\Shv_{\et}\left(\Sch_S\right)$. For $X$ a scheme, we denote by $\Shv\left(X_{\et}\right)$ the $\i$-topos of sheaves of spaces on its small \'etale site.
\item Given an $\i$-category $\sC$, we have a Yoneda embedding: $y: \sC \hookrightarrow \PShv\left(\sC\right)$. So we write $y\left(c\right)$ for an object in $\sC$ thought of as an object in $\PShv\left(\sC\right)$ via the Yoneda embedding.
\item In the event that we consider presheaves/sheaves of sets/$0$-truncated homotopy types we add decorations: $\PShv^{\delta}\left(\sC\right)$ and $\Shv^{\delta}\left(\sC\right)$.
\item The $\infty$-category of $\infty$-topoi with geometric morphisms is denoted by $\Topi$: the objects are $\infty$-topoi $\cX$ and the arrows are geometric morphisms $f: \cX \rightarrow \cY$. Such a geometric morphism consists of a pair of adjoint functors $f_* \vdash f^*,$ with $$f_*:\cX \to \cY$$ and $$f^*:\cY \to \cX,$$ such that $f^*$ is left exact. As $f_*$ is uniquely determined up to a contractible space of choices by $f^*,$ we have that $$\map_{\Topi}\left(\cX,\cY\right)\simeq \left(\Fun^{LE}\left(\cY,\cX \right)\right)^{\times},$$ where $\left(\Fun^{LE}\left(\cY,\cX\right)\right)^{\times}$ is the maximal sub Kan complex of the $\i$-category $\Fun^{LE}\left(\cY,\cX\right)$ of left exact colimit preserving functors.
\item The $\infty$-category of presentable $\infty$-category with morphisms being left adjoints is denoted by $\mathfrak{Pr}^L$,  while the $\infty$-category with the same objects but right adjoints are morphisms is denoted by $\mathfrak{Pr}^R$. A thorough discussion of these categories is the subject of \cite[Section 5.5.3]{htt}. 
\item We denote $\Sch$ be the category of schemes and let $\Sch^{\ft}$ be the full subcategory of schemes of finite type. We denote by $\Aff$ the category of affine schemes and let $\Aff^{\ft}$ be the full subcategory of affine schemes of finite type. We always write $\Sm$ to be the subcategory of smooth schemes of finite type.
\item Suppose that $j: \sC \rightarrow \sD$ is a functor $\infty$-categories, let $f: \sC \rightarrow \sE$ be a functor and $\sE$ and a cocomplete $\infty$-category. We denote by $\LKE_jf: \sD \rightarrow \sE$ the left Kan extension of $f$ along $j$ \cite[Section 4.3.2-4.3.3]{htt}. In this paper, we will only employ left Kan extensions along Yoneda embeddings.
\end{enumerate}

\subsection{Acknowledgments} It is our pleasure to thank Tom Bachmann, Marc Hoyois, Jacob Lurie, Tomer Schlank, Jay Shah, Vladimir Sosnilo and Kirsten Wickelgren for helpful discussions on various aspects of this work. We would especially like to thank Marc Hoyois for help with Lemma~\ref{lem:coexc-pres}. We also thank George Mason University and the Max-Planck Institute for Mathematics and Peter Teichner for partially funding EE's travels to visit DC.

EE would additionally like to thank his adviser John Francis for continuous support, and assigning him Eric Friedlander's proof of the Adams conjecture via \'{e}tale homotopy theory as an exam topic.  EE acknowledges further support from NSF grant DMS 1508040.
\section{Preliminaries}\label{sec:prelim}

\subsection{Pro-objects}\label{subsec:proobj} In this section, we will give a rapid introduction to pro-objects in the setting of $\infty$-categories. References for this theory, are \cite[Appendix E]{luriespec}, \cite[Section 3]{dagxiii}, and the paper \cite{prohomotopy} which also does comparisons with approaches using model categories.

The idea of pro-objects in $\i$-categories, just like its classical counterpart, is to adjoin formal cofiltered limits. 

\begin{definition} \label{def:pro} Let $\sC$ be an $\infty$-category, then the $\infty$-category $\Pro\left(\sC\right)$, called the $\infty$-category of \emph{pro-objects} of $\sC$, is an $\infty$-category equipped with a fully faithful functor 
\begin{equation} \label{eq:pro-yoneda}
j: \sC \rightarrow \Pro\left(\sC\right)
\end{equation} satisfying the following universal property: 

$\Pro\left(\sC\right)$ has small cofiltered limits and if $\sD$ is an $\infty$-category admitting small cofiltered limits, then the precomposition with $j$ induces an equivalence of $\infty$-categories:

$$\Fun_{\cofilt}\left(\Pro\left(\sC\right), \sD\right) \rightarrow \Fun\left(\sC, \sD\right)$$
where $\Fun_{\cofilt}\left(\Pro\left(\sC\right), \sD\right)$ is the full subcategory of $\Fun\left(\Pro\left(\sC\right), \sD\right)$ spanned by functors that preserve small cofiltered limits.
\end{definition}

The expected formula for mapping spaces in pro-categories holds: if we write $ X= \underset{i}{\lim} X_i$ and $ Y = \underset{j}{\lim} Y_j$ are objects of $\Pro\left( \sC \right)$ then we have a canonical equivalence

\begin{equation} \label{eq:maps}
\map_{\Pro\left( \sC \right)}(X, Y) \simeq \underset{j}\lim \underset{i}{\colim} \map_{\sC}(X_i, Y_j).
\end{equation}
For a reference see, for example, \cite[Section 5]{prohomotopy}. 

Definition~\ref{def:pro} may seem abstract and difficult to work with, but in some cases it has more explicit description:

\begin{proposition} \label{prop:lex} \cite[Proposition 3.1.6]{dagxiii} Suppose that $\sC$ is accessible and has finite limits, then $\Pro\left(\sC\right)$ is the full subcategory of $\Fun\left(\sC, \Spc\right)^{\op}$ on those functors which are left exact and accessible.
\end{proposition}

In particular, functor~\eqref{eq:pro-yoneda} takes an object $C \in \sC$ to the left-exact functor $D \mapsto \map_{\sC}(C, D)$ and, using the formula for mapping spaces~\eqref{eq:maps}, the image of $j$ is cocompact.

One technical issue with the $\infty$-category $\Pro\left(\sC\right)$ is that it is \emph{not necessarily} a presentable $\infty$-category, and we are thus deprived of the machinery of adjoint functor theorems, which are useful for constructing localizations. This turns out to be the source for a lot of difficulties with localizing pro-categories, as was encountered in the predecessors of this work \cite{realization,Qu,Qu2}. The main tool we use to build localization functors is the explicit description of pro-objects as left exact functors, which results in:

\begin{proposition} \label{prop:ladj}  Let $f: \sD \rightarrow \sC$ be a functor of $\infty$-categories which are accessible and have all finite limits, and suppose that $f$ preserves finite limits, then the functor $$\Pro\left(f\right): \Pro\left(\sD\right) \rightarrow \Pro\left(\sC\right)$$ has a left adjoint $$f^*: \Pro\left(\sC\right) \rightarrow \Pro\left(\sD\right).$$
\end{proposition}

\begin{proof} The proof can be found in \cite[Remark 3.1.7]{dagxiii} (but note that there is a typo --- $f^*$ should be a left adjoint). The functor $f^*$ is informally given by $$F: \sD \rightarrow \Spc \mapsto F \circ f: \sC \rightarrow \sD \rightarrow \Spc$$
\end{proof}

The following establishes the (co)completeness properties of pro-categories which will be useful later.

\begin{proposition} \label{prop:limcolim}  Suppose that $\sC$ is accessible and has finite limits, then the $\infty$-category $\Pro\left(\sC \right)$ has has small colimits.
\end{proposition}

\begin{proof} Let us first assume that $\sC$ is small. By Proposition~\ref{prop:lex}, the $\infty$-category $\Pro(\sC)^{\op}$ is equivalent to $\Ind(\sC^{\op})$. Therefore, by \cite[Theorem 5.5.1.1.4]{htt} which characterizes presentable $\infty$-categories are those equivalent to $\Ind$ of an $\infty$-category with finite colimits, we see that $\Pro(\sC)^{\op}$ is presentable. Therefore $\Pro(\sC)$ has all small limits and colimits since presentable $\infty$-categories admit them. When $\sC$ is not small, let $\cU$ be the Grothendieck universe of small sets, and $\cV \ni \cU$ be the Grothendieck universe of large sets, such that $\sC$ is $\cV$-small. Denote by $\Pro_{\cV}\left(\sC\right)$ the $\infty$-category of pro-objects in the universe $\cV,$ which can be described as the opposite category of accessible finite limit preserving functors $F:\sC \to \widehat{\Spc},$ where $\widehat{\Spc}$ is the $\infty$-category of $\cV$-small spaces. Then $\Pro\left(\sC\right)$ can be identified with the full subcategory on those functors which take values in $\cU$-small spaces. These are clearly closed under $\cU$-small colimits in $\Fun\left(\sC,\widehat{\Spc}\right)^{op}.$ Hence $\Pro\left(\sC\right)$ is cocomplete.
\end{proof}

\begin{proposition} \label{prop:finlim} Suppose that $\sC$ is accessible and has finite limits, then the $\infty$-category $\Pro\left(\sC \right)$ has all limits and the functor $\sC \hookrightarrow \Pro(\sC)$ preserves finite limits.

\end{proposition} 

\begin{proof} A proof can be found in, for example, \cite[Lemma 2.3]{kerz-saito-tamme} 

\end{proof}

\begin{definition}
The $\i$-category of \textbf{pro-spaces} is the $\i$-category $\Pro\left(\Spc \right).$ Let $\Spc^\pi$ be the full subcategory of spaces on those which have finitely many connected components, and finitely many non-trivial homotopy groups, all of which are finite. Such spaces are called \textbf{$\pi$-finite spaces}. The $\i$-category $\Pro\left(\Spc^\pi \right)$ is the $\i$-category of \textbf{profinite spaces}.
\end{definition}

\subsection{Shapes of $\infty$-topoi}\label{subsec:absshape}
We now review the basic theory of shapes of $\infty$-topoi. The original reference is \cite{ToVe}, but see also \cite[Section 7.1.6]{htt}, \cite[Appendix E.2]{luriespec}, \cite{Hoyois}, and \cite[Section 2]{dave-etale}.

Given a homotopy type $X \in \Spc_0,$ one has its associated slice $\i$-topos $\Spc/X,$ which, when regarding $X$ as an $\i$-groupoid, can be identified with the presheaf $\i$-category $\PShv\left(X\right)$ by \cite[Corollary 5.3.5.4]{htt}. By \cite[Remark 6.3.5.10, Theorem 6.3.5.13, and Proposition 6.3.4.1]{htt}, this construction extends to a fully faithful embedding
$$\Spc/\left(\blank\right):\Spc \hookrightarrow \Topi$$ (which moreover is colimit preserving).
This functor induces a well-defined functor $$\Spc^{pro}/\left(\blank\right):\Pro\left(\cS\right) \to \mathfrak{Top}_\i$$ which sends a representable pro-space $j\left(X\right)$ to $\Spc/X,$ and sends a pro-space of the form $\underset{ i \in \cI} \lim X_i$ to the cofiltered limit of $\i$-topoi $\underset{ i \in \cI} \lim \Spc/X_i.$ By \cite[Remark 7.1.6.15]{htt}, this functor has a left adjoint $$\Shape:\mathfrak{Top}_\i \to \Pro\left( \cS \right).$$

\begin{remark}
The above functor $\Spc^{pro}/\left(\blank\right)$ is \emph{not} fully faithful, however, it is fully faithful when restricted to the full subcategory of pro-spaces spanned by profinite spaces, by \cite[Appendix E.2]{luriespec}.
\end{remark}

More precisely, let $\sX$ be an $\infty$-topos and let $e_{\sX}: \sX \rightarrow \Spc$ be the terminal geometric morphism \cite[Proposition 6.3.4.1]{htt}. 
 In this special case, we typically denote $e_{\sX}^*$ by $\Delta_{\sX},$ as it sends a space $\cG$ to the constant stack with value $\cG$, and we typically denote $e_{\sX*}$ as $\Gamma_{\sX},$ since it sends a stack to its global sections. If the context is clear, we drop the $\sX$ decorations.

\begin{definition} Let $\sX$ be an $\infty$-topos, then the pro-space $$\Shape\left(\sX\right)$$ is the left-exact functor $\Gamma \circ \Delta: \Spc \rightarrow \Spc$. This pro-space is called the \emph{shape} of the $\infty$-topos $\sX$
\end{definition}

By naturality of the terminal geometric morphism, $\Shape\left(-\right)$ assembles into a functor $$\Shape: \Topi \rightarrow \Pro\left(\Spc\right).$$ Even better, this functor  fits into an adjunction:

\begin{proposition} \label{prop:ladjfun} The functor: $\Spc_{/\left(-\right)}: \Spc \rightarrow \Topi$ extends to the functor $$\Spc^{pro}_{/\left(-\right)}: \Pro\left(\Spc\right) \rightarrow \Topi$$ which participates in an adjunction: $$\Shape: \Topi \rightleftarrows \Pro\left(\Spc\right): \Spc^{pro}{/\left(-\right)}.$$
\end{proposition}

\begin{proof} This is the content of \cite[Remark 7.1.6.15]{htt}. We will also recover this proposition in a more general setting later.
\end{proof}

\subsubsection{Some geometric intuition}

We will now explain some geometric intuition behind this construction. Let $X$ be a topological space, and $S$ a set. Then, as long as $X$ is locally path connected, the set $\pi_0\left(X\right)$ can be described by the universal property $$\Hom_{\mathrm{Top}}\left(X,S\right)\cong \Hom\left(\pi_0\left(X\right),\delta\left(S\right)\right),$$ where $$\delta:\Set \hookrightarrow \mathrm{Top}$$ is the canonical fully faithful inclusion. If we restrict to the full subcategory on locally path connected spaces, then we get that $\pi_0 \dashv \delta.$ Denoting by $\mathfrak{Top}$ the $\left(2,1\right)$-category of topoi, provided we restrict to sober topological space (e.g. Hausdorff), we have fully faithful inclusions
$$\Set \stackrel{\delta}{\longhookrightarrow} \mathrm{Top} \stackrel{\Shv^\delta}{\longhookrightarrow} \mathfrak{Top},$$ and since for a set $S,$ $$\Shv^\delta\left(\delta\left(S\right)\right) \cong \Set/S,$$ we deduce for (sober) locally path connected spaces, $$\Hom_{\mathfrak{Top}}\left(\Shv^\delta\left(X\right),\Set/S\right)\cong \Hom_{\Set}\left(\pi_0\left(X\right),S\right).$$ For a general Grothendieck topos $\cE,$ one can \emph{define} $\pi_0\left(\cE\right)$ to be the set such that for all sets $S,$
$$\Hom_{\Set}\left(\pi_0\left(\cE\right),S\right) \cong \Hom_{\mathfrak{Top}}\left(\cE,\Set/S\right).$$ Clearly, if such a $\pi_0\left(\cE\right)$ exists, it is unique, but it need not exist. However, the functor $\Hom_{\mathfrak{Top}}\left(\cE,\Set/\left(\blank\right)\right)$ is always pro-representable; as a left-exact functor, this pro-set is imply $\Gamma_{\cE}\Delta_{\cE}:\Set \to \Set.$ It therefore makes sense to define the pro-set $$\pi_0\left(\cE\right):=\Gamma_{\cE}\Delta_{\cE}.$$ If $\cE$ is a locally connected topos, $\Delta_{\cE}$ has a left adjoint  $\Pi_{\cE}$, and therefore the above functor preserves all limits, and is thus representable by $\Pi_\cE\left(1\right)=\pi_0\left(\cE\right).$

Let us go one step further. Suppose that $X$ is a topological space. We would like to describe the fundamental groupoid $\Pi_1\left(X\right)$ as the groupoid such that for all groupoids $\cG,$ we have an equivalence of groupoids $$\Hom\left(X,\cG\right) \simeq \Hom_{\mathrm{Gpd}}\left(\Pi_1\left(X\right),\cG\right).$$ For the analogous description of $\pi_0\left(X\right),$ we used that sets (i.e. $0$-groupoids) embed fully faithfully into topological spaces, however there is no such embedding of $1$-groupoids. However, both categories embed into $2$-topoi, that is, when restricting to sober topological spaces, we have the embedding $$\Top \hookrightarrow \mathfrak{Top} \hookrightarrow \mathfrak{Top}_2$$ which sends a topological space $X$ to its $2$-topos of stacks of groupoids $\St\left(X\right),$ and we also have the fully faithful embedding $$\mathrm{Gpd} \hookrightarrow \mathfrak{Top}_2$$ sending a groupoid $\cG$ to $\mathrm{Gpd}/\cG.$ If $X$ is not only locally path connected, but also (semi-)locally simply connected, then it is true that for all groupoids $\cG$, there is a natural equivalence of groupoids
$$\Hom_{\mathfrak{Top}_2}\left(St\left(X\right),\mathrm{Gpd}/\cG\right) \simeq \Hom_{\mathrm{Gpd}}\left(\Pi_1\left(X\right),\cG\right).$$ For a general $2$-topos $\cE,$ one can \emph{define} $\Pi_1\left(\cE\right)$ to be the groupoid such that for all groupoids $\cG,$
$$\Hom_{\mathrm{Gpd}}\left(\Pi_1\left(\cE\right),\cG\right) \simeq \Hom_{\mathfrak{Top}_2}\left(\cE,\mathrm{Gpd}/\cG\right),$$ and the functor $\Hom_{\mathfrak{Top}_2}\left(\cE,\mathrm{Gpd}/\left(\blank\right)\right)$ is always pro-representable. Similarly to as before, as a left-exact functor, this pro-groupoid is simply $\Gamma_{\cE}\Delta_{\cE}:\mathrm{Gpd} \to \mathrm{Gpd}.$ This follows since the following is a pullback diagram in $\mathfrak{Top}_2$
$$\xymatrix{\cE/\Delta_{\cE}\left(\cG\right) \ar[d] \ar[r] & \mathrm{Gpd}/\cG \ar[d]\\ \cE \ar[r] & \mathrm{Gpd},}$$ together with the fact that $\mathrm{Gpd}$ is the terminal $2$-topos. It therefore make sense to define the pro-groupoid $$\Pi_1\left(\cE\right):=\Gamma_{\cE}\Delta_{\cE}.$$ 
Given a scheme $S,$ one has associated to $S$ its small \'etale $2$-topos $\St\left(S_{\et}\right)$, which is the $2$-topos of stacks over the small \'etale site $S_{\et}$ of $S.$ The objects of this site consist of \'etale morphisms $U \to S,$ with $U$ another scheme, and the Grothendieck topology is generated by \'etale covering families. If $S$ is moreover connected, the pro-groupoid $\Pi_1\left(\St\left(S_{\et}\right)\right)$ is equivalent to the \'etale fundamental group $\pi_1^{\et}\left(S\right)$, regarded as a pro-groupoid with one object. In particular, it follows that if $A$ is an abelian group, that the groupoid $$\Hom_{\Pro\left(\mathrm{Gpd}\right)}\left(\Pi_1\left(\St\left(S_{\et}\right)\right),\B\Aut\left(A\right)\right) \simeq \Hom_{\Pro\left(\mathrm{Grp}\right)}\left(\pi_1^{\et}\left(S\right),\Aut\left(A\right)\right)$$ is equivalent to the groupoid of $\Aut\left(A\right)$-torsors on $S.$ When $S$ is locally Noetherian, it is well known that the latter is equivalent to the groupoid of local systems on $S$ with coefficients in $A.$

The functor $\Shape$ generalizes both of these constructions, but for $\i$-groupoids.

\begin{remark}
Given the above, one may hope that for a locally contractible topological space $X,$ that $\Shape\left(\Shv\left(X\right)\right)$ is given by its underlying $\i$-groupoid $\Pi_\i\left(X\right),$ but this is not quite true, but this is a subtle issue. This is true if instead of using $\Shv\left(X\right),$ we use its hypercompletion (i.e. we use descent with respect to hypercovers rather than \v{C}ech covers) \cite[Proposition 2.12]{dave-etale}, or if we add additional mild hypothesis to $X,$ such that $X$ has the homotopy type of a CW-complex \cite[Remark A.1.4]{higheralgebra}.
\end{remark}

\subsection{The \'etale homotopy type} \label{subsec:et-type} Given a scheme $S,$ one has associated to $S$ its small \'etale $\i$-topos $\Shv\left(S_{\et}\right).$
This construction extends to a functor $$\Shv\left(\left(\mspace{3mu}\cdot\mspace{3mu}\right)_{\et}\right):\Sch \to \Topi$$ \cite[Theorem 2.2.12]{dagv}. Moreover, this functor extends to the $\i$-topos $\Shv_{\et}\left(\Sch^{\ft}\right)$:

\begin{lemma}\cite[Lemma 2.2.6]{dave-etale}
The above construction extends to a colimit preserving functor
$$\Shv\left(\left(\mspace{3mu}\cdot\mspace{3mu}\right)_{\et}\right):\Shv_{\et}\left(\Sch^{\ft}\right)\to \Topi.$$
\end{lemma}
For a fixed scheme $S,$ the pro-space arising as the shape of $\i$-topos $\Shv\left(S_{\et}\right)$ is what is defined in \cite{dave-etale} as the \'etale homotopy type of $S,$ and is denoted by $\Pi^{\et}_\i\left(S\right)$. It is closely related to the \'etale homotopy type of Artin-Mazur, and also that of Friedlander; see \cite[Remark 3.25]{dave-etale} and \cite[Corollary 3.4]{Hoyois} respectively.

\begin{corollary}\label{cor:shapedesc}
The construction of \'etale homotopy type extends to a colimit preserving functor
$$\Pi^{\et}_\i\left(\mspace{3mu}\cdot\mspace{3mu}\right):\Shv_{\et}\left(\Sch^{\ft}\right) \rightarrow \Pro\left(\Spc\right).$$
\end{corollary}
\begin{proof}
The above functor is simply the composite $\Shape \circ \Shv\left(\left(\mspace{3mu}\cdot\mspace{3mu}\right)_{\et}\right),$ which is colimit preserving since $\Shv\left(\left(\mspace{3mu}\cdot\mspace{3mu}\right)_{\et}\right)$ and $\Shape$ are left adjoints.
\end{proof}

\begin{theorem}\cite[Theorem 2.40]{dave-etale}\label{thm:2.40}
Let $\cX$ be an object of $\Shv_{\et}\left(\Sch^{\ft}\right).$ Then its \'etale homotopy type $$\Pi^{\et}_\i \left(\cX \right):\Spc \to \Spc$$ may be identified with the functor $$V \mapsto \map_{\Shv_{\et}\left(\Sch^{\ft}\right)}\left(\cX,\Delta\left(V\right)\right),$$ where $\Delta\left(V\right)$ is the constant stack on the site $\left(\Sch^{\ft},\et\right)$ with value $V$.
\end{theorem}

\subsubsection{\'Etale geometric morphisms}
Let $E$ and object of an $\i$-topos $\cE.$ There is a canonically induced geometric morphism $\pi:\cE/E \to \cE$,
\begin{equation*}\label{eq:ettt}
\xymatrix{\cE/E \ar@<-0.5ex>[r]_-{\pi_*} & \cE \ar@<-0.5ex>[l]_-{\pi^*}},
\end{equation*}
such that $\pi^*$ has a left adjoint
\begin{eqnarray*}
\pi_!:\cE/E &\to& \cE\\
f:F \to E &\mapsto& F
\end{eqnarray*}
\begin{definition}\label{dfn:etale}
A geometric morphism $\cF \to \cE$ between $\i$-topoi is \textbf{\'etale}, if it is equivalent to one of the form $$\cE/E \to \cE$$ for some object $E \in \cE.$
\end{definition}

Important properties of \'etale geometric morphisms are as follows:\\
\begin{proposition}(\cite[Corollary 6.3.5.9]{htt})\label{prop:6.3.5.9}
Suppose that a composite of geometric morphisms $$\cE \stackrel{f}{\to} \cF \stackrel{g}{\to} \cG$$ is \'etale, and that $g$ \'etale. Then $f$ is \'etale.
\end{proposition}

Let $\Topi^{\et}/\cE$ be the full subcategory of the slice $\i$-category $\Topi/\cE$ spanned by the \'etale geometric morphisms. By the above proposition, all the morphisms in this $\i$-category are \'etale. The assignment $E \mapsto \left(\cE/E \to E\right)$ extends to a canonical functor
$$\chi_{\cE}:\cE \to \Topi^{\et}/\cE.$$
\begin{proposition}\label{lem:htt6.3.5.10} (\cite[Remark 6.3.5.10]{htt})
For every $\i$-topos $\cE,$ the functor $\chi_\cE$ is an equivalence of $\i$-categories.
\end{proposition}

\subsection{Motivic homotopy theory}\label{subsec:a1}

We now review Morel-Voevodsky's motivic homotopy theory \cite{morel-voevodsky} mainly to set up some notation. We will be brief; in the language of $\infty$-categories, a thorough treatment is given in \cite{robalo}. Let $S$ be a quasicompact, quasiseparated scheme. We may consider the $\infty$-category of presheaves on smooth schemes of finite type over $S$, $\PShv\left(\Sm_S\right)$. The (reflective) localization that constructs the \emph{unstable motivic homotopy $\infty$-category} is done in two steps.

The first is given by Nisnevich localization. Recall that a pullback diagram of finite type $S$-schemes
    \begin{equation} \label{cd}
        \xymatrix{
            U\times_X V\ar[r]\ar[d] &   V\ar[d]^p\\
            U\ar[r]^i & X
        }
    \end{equation}
is an \emph{elementary distinguished (Nisnevich) square} if $i$ is a Zariski open immersion, $p$ is \'etale, and $p^{-1}\left(X-U\right)\rightarrow\left(X-U\right)$ is an isomorphism of schemes where $X-U$ is equipped with the reduced induced scheme structure.

We let $\Shv_{\Nis,\exc}\left(\Sm_S\right) \subset \PShv\left(\Sm_S\right)$ be the full subcategory of \emph{Nisenvich excisive} presheaves --- those that take any elementary distinguished square to a pullback square and $\emptyset$ to the terminal object. In this generality, the $\infty$-category of Nisnevich excisive presheaves agree with the $\infty$-category of Nisnevich sheaves (with respect to all existing definitions of the Nisnevich topology); see \cite[Proposition A.2]{Bachmann:2017aa}.

Next, we invert the projection morphisms $X \times_S \bA^1 \rightarrow X$. To do so we consider the full subcategory $$\Spc\left( S \right) \hookrightarrow \Shv_{\Nis}\left(\Sm_S\right)$$ of objects $Y$ such that the induced map $\map\left(X, Y\right) \rightarrow \map\left(X \times_S \bA^1, Y\right)$ is an equivalence; objects satisfying the above property will be called \emph{$\bA^1$-invariant}. We call this $\infty$-category the \emph{unstable category of motivic spaces} over $S$; objects in this category will be called \emph{motivic spaces}.

In total, the unstable category of motivic spaces is a two step localization: 
\begin{equation} \label{eq:lmot}
\PShv\left(\Sm_S\right) \stackrel{\LL_{\Nis}}{\longlongrightarrow} \Shv_{\Nis}\left(\Sm_S\right) \stackrel{\LL_{\bA^1}}{\longlongrightarrow}\Spc\left( S \right)
\end{equation}
 and the composite will be denoted by $\LL_{\mot}.$ The universal properties for localization in the $\infty$-categorical setting yields the following universal property; for details the reader can consult \cite[Theorem 2.30]{robalo}:

\begin{theorem} The $\infty$-category $\Spc\left( S \right)$ satisfies the following universal property: for any $\infty$-category $\sD$ with all small colimits, composition with $\LL_{\mot}$ induces a fully faithful functor $$\Fun^{\LL}\left(\Spc\left( S \right), \sD\right) \rightarrow \Fun\left(\Sm_S, \sD\right)$$ and its essential image is spanned by functors which satisfy Nisnevich excision and $\bA^1$-invariance. Here $\Fun^{\LL}\left(-,-\right)$ is the full subcategory of functors which preserves colimits.
\end{theorem}

For use in Section~\ref{sect:stable}, the stable $\infty$-category of \emph{$S^1$-motivic spectra} is the stabilization of $\Spc\left(S \right)$:
\[
\SH^{S^1}(S):= \Spt\left( \Spc\left(S \right) \right);
\]
we give a rapid review of stabilization in subsection~\ref{sect:stable-prof}. Since $\Spc\left(S \right), \Shv_{\Nis}\left(\Sm_S\right), \PShv\left(\Sm_S\right)$ are presentable $\infty$-category, their stabilizations are computed as tensoring with $\Spt$ in $\Pr^L$ \cite[Example 4.8.1.23]{higheralgebra}. This is one quick way to see that the localization functors of~\eqref{eq:lmot} stabilize to localizations fitting in the top row of the commutative diagram of left adjoints below

\begin{equation} \label{eq:stab-mot}
\xymatrix{
\Spt\left(\PShv\left(\Sm_S\right) \right) \ar[r]^{\LL_{\Nis}} & \Spt\left(\Shv_{\Nis}\left(\Sm_S\right) \right) \ar[rr]^-{\LL_{\bA^1}} & & \SH^{S^1}(S)\\
\PShv\left(\Sm_S\right)  \ar[r]^{\LL_{\Nis}} \ar[u]_{\Sigma^{\infty}_+}& \Shv_{\Nis}\left(\Sm_S\right)  \ar[rr]^-{\LL_{\bA^1}} \ar[u]_{\Sigma^{\infty}_+} & &\Spc\left( S \right) \ar[u]_{\Sigma^{\infty}_+}.
}
\end{equation}
We also note that $\Spt\left(\PShv\left(\Sm_S\right) \right)$ canonically identifies with the stable $\infty$-category of presheaves of spectra on $\Sm_S$ and $\SH^{S^1}\left(S \right)$ (resp. $\Spt\left(\Shv_{\Nis}\left(\Sm_S\right) \right)$) identifies as the full subcategory of $\bA^1$-invariant Nisnevich sheaves of spectra (resp. Nisnevich sheaves of spectra). This is true for any Grothendieck topology $\tau$ and true if we replace $\Sm_S$ by a small subcategory of $\Sch^{\ft}_S$. Henceforth, if $\tau$ is a topology and $\sC \subset \Sch^{\ft}_S$, we denote by $\Shv_{\tau,\Spt}\left(\sC \right)$ the stable $\infty$-category of $\tau$-sheaves of spectra, whence $\Spt\left(\Shv_{\tau}\left(\Sm_S\right) \right) \simeq \Shv_{\tau,\Spt}\left(\sC \right)$.

For later use, we need to know the Suslin-Voevodsky formula for $\LL_{\bA^1}$-localization. We have the standard cosimplicial scheme $\Delta^{\bullet}_S$ where $\Delta^{n}_S := \Delta^n_{\bZ} \times_{\bZ} S$ and $\Delta^n_{\bZ} := \Spec\,\bZ[T_0, \cdots T_n]/ (T_0 + \cdots +T_n = 1)$. Then, if $F$ is a presheaf of spectra the $\bA^1$-localization functor is computed as 
\begin{equation} \label{eq:sv-loc}
\LL_{\bA^1}F \simeq \colim F^{\Delta^{\bullet}},
\end{equation} as proved in \cite[Section 2.3]{morel-voevodsky}. The following fact is standard

\begin{proposition} \label{prop:la1-lnis} Let $F$ be a Nisnevich sheaf of spectra on $\Sch^{\ft}_S$, then $\LL_{\bA^1}F$ is a Nisnevich sheaf.
\end{proposition}

\begin{proof} Suppose that $Q \in \Sch_S^{\ft,\Delta^1 \times \Delta^1}$ is an elementary distinguished square, then we need to show that $\LL_{\bA^1}F(Q)$ is Cartesian. Since we are in a stable setting it suffices to prove that $\LL_{\bA^1}F(Q)$ is coCartesian, but this follows from the formula in~\eqref{eq:sv-loc} and the fact that colimits commute with one another.
\end{proof}

The $\bP^1$-stable $\infty$-category of motivic spectra will be discussed and used only in the last section of this paper; this is the stable presentably symmetric monoidal $\infty$-category obtained from $\SH^{S^1}(S)$ by inverting $\Sigma^{\infty}(\bP^1, \infty)$:
\[
\SH(S):= \SH^{S^1}(S)[\Sigma^{\infty}(\bP^1, \infty)^{\otimes -1}]. 
\]
For its universal property, and more details on this construction, we refer to \cite[Corollary 1.2]{robalo}; see also \cite[Lemma 4.1]{Bachmann:2017aa}. One of the key theorems of the $\bP^1$-stable category is that invariants representable in this stable $\infty$-category satisfy \emph{cdh} descent  as proved in \cite{cisinski-desc}. Just like the Nisnevich topology, having cdh descent is equivalent to being excisive for certain squares --- the Nisnevich distinguished squares and the \emph{abstract blow up squares} (see \cite{vv-jpaa}). The latter is defined via the diagram~\eqref{cd}: one requires that $p$ is proper, $i$ is a closed immersion and $p^{-1}\left(X-U\right)\rightarrow\left(X-U\right)$ is an isomorphism. We will use this topology in subsection~\ref{sec:blancdf}.

\section{Absolute \'{E}tale Realization of Motivic Spaces Revisited}\label{sec:relativeetale}

Let $S$ be a fixed scheme. In this section, we begin by constructing an $\infty$-categorical version of the \'{e}tale realization functor of motivic spaces of Isaksen's \cite{realization}. However, there is already a slight improvement in our construction: our functor, constructed in Theorem~\ref{thm:absolute_realization}, is of the form: $$\Et_{\bA^1}: \Spc\left( S \right) \rightarrow \Pro\left(\Spc_{S,\bA^1}\right)$$ where $\Pro\left(\Spc_{S,\bA^1}\right)$ is a certain localization of the $\infty$-category of pro-spaces (see Theorem~\ref{thm:spc-s-contains} for the class of pro-spaces that it contains). This category refines the target of Isaksen's functor which lands in a completion away from the residue characteristics that appear in $S$ \cite[Corollary 2.7]{realization}, in the sense that the latter is a further localization of the former (see Remark~\ref{rmk:isaksen}).



\subsection{Localizations of pro-spaces}\label{subsec:loc} We begin with generalities on localizations of pro-spaces.

\subsubsection{$p$-profinite spaces and Isaken's model structure on pro-spaces}

\begin{definition}
Let $p$ be a prime. A space $X$ is \textbf{$p$-finite} if $\pi_0\left(X\right)$ is finite and each $\pi_i\left(X\right)$ is a finite $p$-group. Denote the full subcategory of $\Spc$ on $p$-finite spaces by $\Spc^{p\mbox{-}\pi}.$
\end{definition}

\begin{proposition}
$\Spc^{p\mbox{-}\pi}$ has finite limits and is idempotent complete.
\end{proposition}
\begin{proof}
The fact that $\Spc^{p\mbox{-}\pi}$ has finite limits follows from the long exact sequence in homotopy groups arising from a homotopy pullback. To see that $\Spc^{p\mbox{-}\pi}$ is idempotent complete, it suffices to observe that retracts are preserved by all functors, and a retract of a finite $p$-group is again a finite $p$-group.
\end{proof}

\begin{definition} \label{def:prof-compl}
The $\i$-category of \textbf{$p$-profinite} spaces is the $\i$-category of pro-objects $\Prof_p\left(\Spc\right):=\Pro\left(\Spc^{p\mbox{-}\pi}\right).$ 
\end{definition}

Denote by $$i_p:\Spc^{p\mbox{-}\pi} \hookrightarrow \Spc$$ the fully faithful inclusion. Then $i_p$ is accessible and preserves finite limits, hence the canonical fully faithful functor $$\Pro\left(i_p\right):\Prof_p\left(\Spc\right) \hookrightarrow \Pro\left(\Spc\right)$$ has a left adjoint $\left(i_p\right)^*,$ by Proposition \ref{prop:ladj}.

\begin{definition} \label{def:p-prof}
We denote $\left(i_p\right)^*$ by $\widehat{\left(\blank\right)}_p,$ and refer to $\widehat{X}_p,$ for $X$ a pro-space, as its \textbf{$p$-profinite completion}.
\end{definition}

\begin{theorem}\cite[Theorem 2.5]{realization}
There is a model structure on the category $\Pro\left(\Set^{\Delta^{op}}\right)$ of pro-simplicial sets, for which the weak equivalences are precisely those maps which induce an isomorphism in all cohomology groups with $\mathbb{Z}/p\mathbb{Z}$-coefficients.
\end{theorem}

The following result is due to Barnea-Harpaz-Horel:
\begin{theorem}\cite[Corollary 7.3.7]{prohomotopy}\label{thm:compareisak}
The $\i$-category associated to the above model category is equivalent to the $\i$-category $\Prof_p\left(\Spc\right)$ of $p$-profinite spaces.
\end{theorem}

\subsubsection{$\bA^1$-invariant spaces with respect to a base scheme.}
Throughout this subsection, we fix a base scheme $S$. Denote by $\Sch^{\ft}_S$ the category of $S$-schemes of finite type, equip it with the \'etale topology and denote by $\Shv_{\et}\left(\Sch^{\ft}_S\right)$ the $\i$-topos of $\i$-sheaves on the resulting site.

\begin{definition} \label{defn:sa1} 
An object $\cY$ of $\Shv_{\et}\left(\Sch^{\ft}_S\right)$ is \textbf{$\bA^1$-invariant} if for all objects $\cX$ of $\Shv_{\et}\left(\Sch^{\ft}_S\right),$ the canonical map
$$\map\left(\cX,\cY\right) \to \map\left(\cX \times \bA^1_S,\cY\right)$$ is an equivalence.
Consider the unique geometric morphism 
\begin{equation}\label{eq:geomterm}
\xymatrix@C=2cm{\Shv_{\et}\left(\Sch^{\ft}_S\right) \ar@<-0.65ex>[r]_-{\Gamma} & \Spc. \ar@<-0.65ex>[l]_-{\Delta}}
\end{equation}
A space $V$ will be called \textbf{$\mathbb{A}^1$-invariant relative to $S$}  if $\Delta\left(V\right)$ is.
Let $\Spc_{S,\bA^1}$ be the full subcategory of $\Spc$ spanned by the $\bA^1$-invariant spaces.
\end{definition}


\begin{proposition} \label{prop:sa1-ft}
An object $\cY \in \Shv_{\et}\left(\Sch^{\ft}_S\right)$ is $\bA^1$-invariant, if and only if for all $S$-schemes $Z$ of finite type, the canonical map $$\cY\left(Z\right) \to \cY\left(Z \times \bA^1_S\right)$$ is an equivalence.
\end{proposition}

\begin{proof}
The only if direction is trivial. Conversely, suppose that the above condition holds for all schemes of finite type. Let $\cX$ be an arbitrary stack. Then it can be written canonically as a colimit of representables:
$$\cX = \underset{Z \to \cX} \colim Z.$$ Then:
\begin{eqnarray*}
\map\left(\cX,\cY\right) &\simeq& \underset{Z \to \cX} \lim \map\left(Z,\cY\right)\\
&\simeq& \underset{Z \to \cX} \lim \map\left(Z \times \bA^1_S,\cY\right)\\
&\simeq&  \map\left(\underset{Z \to \cX} \colim\left(Z \times \bA^1_S\right),\cY\right)\\
&\simeq&  \map\left(\left(\underset{Z \to \cX} \colim Z\right) \times \bA^1_S,\cY\right)\\
&\simeq& \map\left(\cX \times \bA^1_S,\cY\right),
\end{eqnarray*}
where the second equivalence follows from the Yoneda Lemma, and the fourth since colimits are universal in any $\i$-topos.
\end{proof}

\begin{corollary}\label{cor:repsgood}
A space $V$ is $\bA^1$-invariant relative to $S$ if and only if for all $S$-schemes of finite type $Z,$ the canonical map $$\Delta\left(V\right)\left(Z\right) \to \Delta\left(V\right)\left(Z \times \bA^1_S\right)$$ is an equivalence.
\end{corollary}

 \begin{example} If $A$ is locally constant sheaf of torsion abelian group, whose torsion is prime to the residue characteristics of $S$ then $A$ is \emph{strictly $\mathbb{A}^1$-invariant} in the sense that the presheaf of abelian groups on $\Sch^{\ft}_S$ given by $U \mapsto H^i_{\et}(U, A)$ is $\bA^1$-invariant for all $i \geq 0$. relative to $S$ \cite[XV Corollarie 2.2]{SGA4}. Hence, for such groups if $A$ is a group whose torsion is prime to the residue characteristics, the $n$-fold deloopings $K( A, n)$ are spaces which are $\bA^1$-invariant relative to $S$. We will see generalizations of this example later in Theorem~\ref{thm:spc-s-contains}.
 \end{example}

We denote the canonical fully faithful embedding by 
$$\is: \Spc_{S,\bA^1} \hookrightarrow \Spc.$$

\begin{proposition}  \label{prop:sa1good} The $\infty$-category $\Spc_{S,\bA^1}$ has finite limits and is accessible.
\end{proposition}

\begin{proof} It obviously has a terminal object, so, to prove it has finite limits, it suffices to prove that it has pullbacks. Suppose that 

$$\xymatrix{X \times_Y Z \ar[r] \ar[d] & Z \ar[d]\\ X \ar[r] & Y}$$
is a Cartesian square where $X, Y, Z \in \Spc_{S,\bA^1}$, we claim that $X \times_Y Z$ is in $\Spc_{S,\bA^1}$. The functor $\Delta$, being part of a geometric morphism, preserves finite limits and thus preserves the above pullback square. It then follows that $\Delta\left(X \times_Y Z\right)$ is $\bA^1$-invariant since $\map\left(T, -\right)$ preserves limits for all $T$.

To prove that the category is accessible it suffices to prove that it is idempotent complete \cite[Corollary 5.4.3.6]{htt}. To do so, we need only prove that $\Spc_{S,\bA^1}$ is closed under retracts (in $\Spc$). Let $r: X \rightarrow Y$ be a retract such that $X \in \Spc_{S,\bA^1}$. Consider the diagram in $\Shv_{\et}\left(\Sch^{\ft}_S\right)$:

$$\xymatrix{
\map\left(T, \Delta\left(X\right)\right) \ar[r] \ar[d] & \map\left(T \times \bA^1, \Delta\left(X\right)\right) \ar[d]\\
\map\left(T, \Delta\left(Y\right)\right) \ar[r]  & \map\left(T, \times \bA^1, \Delta\left(Y\right)\right).}$$
The diagram witnesses the bottom arrow as a retract of the top arrow in the arrow category. The top arrow is an equivalence by hypothesis and the claim follows from the fact that equivalences are closed under retracts in any $\infty$-category.
\end{proof}

\begin{corollary} \label{prop:sa1isloc} The fully faithful embedding: $\is: \Spc_{S,\bA^1} \hookrightarrow \Spc$ induces a fully faithful functor $$\Pro\left(i\right): \Pro\left(\Spc_{S,\bA^1}\right) \hookrightarrow \Pro\left(\Spc\right)$$ which admits a left adjoint $$\is^*: \Pro\left(\Spc\right) \rightarrow \Pro\left(\Spc_{S,\bA^1}\right).$$
Hence $\is^*$ is a localization.
\end{corollary}
\begin{proof} The induced functor exists and is fully faithful due to the universal properties of $\Pro$. The fact that it admits a left adjoint is due to Proposition \ref{prop:sa1good} and Proposition \ref{prop:ladj}.
\end{proof}

We now introduce some notation:

\begin{definition}
We denote the localization functor $\is^*$ by $\cs$.
\end{definition}

\subsubsection{Objects of $\Spc_{S,\bA^1}$}\label{subsec:objects}

Our realization functor will land in the $\infty$-category $\Spc_{S,\bA^1}$, hence it will be paramount to know that this $\infty$-category contains interesting objects. We will name some objects in this category and show that it is a finer category than just the completion of pro-spaces away from primes occurring in the local rings of $S$ as considered by \cite{morel1} and \cite{christensen-isaksen}. 

\begin{definition}\label{defn:spcp} Let $\mathcal{P}$ denote a collection of prime numbers. Denote by $\Spc_\mathcal{P}$ the full subcategory of $\Spc$ on those spaces $X$ such that for all $p$ in $\mathcal{P}$:

\begin{enumerate}
\item For each connected component $X_{\alpha}$ of $X,$ the space $X_{\alpha}$ has only finitely many nonzero homotopy groups, each of which is finite.
\item For $i\geq1,$ each $|\pi_i\left(X_{\alpha}\right)|$ is coprime to $p$.
\item For $i >2,$ each $|\Aut\left(\pi_i\left(X_{\alpha}\right)\right)|$ is also coprime to $p.$
\end{enumerate}
\end{definition}



The motivation for condition (3) in Definition \ref{defn:spcp} comes from the following \cite{dave-etale}:

\begin{proposition} \cite[Proposition 3.7]{dave-etale}\label{prop:univfib} Let $\Aut\left(K\left(A,n\right)\right)$ denote the space of self-homotopy equivalences of $K\left(A,n\right)$. The space $\B\Aut\left(K\left(A,n\right)\right)$ classifies $K\left(A,n\right)$-fibrations and the universal $K\left(A, n\right)$ fibration is given by $K\left(\Aut\left(A\right), 1\right) \rightarrow \B\Aut \left(K\left(A, n\right)\right).$ 
\end{proposition}

\begin{remark}
The $\i$-categories $\Spc^{\ell\mbox{-}\pi}$ and $\Spc_\mathcal{P}$ are related when $\ell \notin \mathcal{P}.$ However, $\Aut\left(\bZ/\ell\right)\cong \bZ/\!\!\left(\ell-1\right),$ so $K\left(\bZ/\ell,n\right)$ does not satisfying condition (3) of Definition \ref{defn:spcp} when $\ell \equiv 1 \mod p.$
\end{remark}

\begin{definition}\label{defn:spcpnil} Let $\mathcal{P}$ denote a collection of prime numbers. Denote by $\Spc_\mathcal{P}^{nil}$ the full subcategory of $\Spc$ on those spaces $X$ such that for all $p$ in $\mathcal{P}$:
\begin{enumerate}
\item For each connected component $X_{\alpha}$ of $X,$ the space $X_{\alpha}$ has only finitely many nonzero homotopy groups, each of which is finite.
\item For $i\geq1,$ each $|\pi_i\left(X_{\alpha}\right)|$ is coprime to $p$.
\item Each $\pi_1\left(X_\alpha\right)$ is nilpotent, and each $\pi_i\left(X_\alpha\right)$ is a nilpotent $\pi_1\left(X_\alpha\right)$-module.
\end{enumerate}
\end{definition}

\begin{remark}
$\Spc_\mathcal{P}^{nil}$ contains all simply connected spaces with finitely many connected components, all of whose homotopy groups are finite of order coprime to $p$, for all $p \in \mathcal{P}.$
\end{remark}

The following proposition follows immediately from \cite[Proposition 7.3.4]{prohomotopy}:
\begin{proposition}\label{prop:nilbig}
$\Spc_\mathcal{P}^{nil}$ contains $\Spc^{\ell\mbox{-}\pi}$ for all $\ell \notin \mathcal{P}.$
\end{proposition}

Fix a scheme $Z.$ The morphism $\pi: Z \times \bA^1 \rightarrow Z$ induces a morphism in $\Topi$: $$\Shv_{\et}\left(\pi\right): \Shv_{\et}\left(Z \times \bA^1\right) \rightarrow \Shv_{\et}\left(Z\right)$$ whence we get a morphism:  $$\Pi^{\et}_\i\left(Z\right) \rightarrow \Pi^{\et}_\i\left(Z\times \bA^1\right)$$ in $\Pro\left(\Spc\right)$. Regarding both of these pro-spaces as functor from spaces to spaces, we have the following lemma:

\begin{lemma}\label{lem:smalliscooltoo}
A space $V$ is in $\Spc_{S,\bA^1}$ if and only if for all $S$-schemes $Z$ of finite type, the canonical map 
\begin{equation}\label{eq:z}
\Pi^{\et}_\i\left(Z\right)\left(V\right) \to \Pi^{\et}_\i\left(Z \times \bA^1\right)\left(V\right)
\end{equation}
is an equivalence.
\end{lemma}

\begin{proof}
This follows immediately from \cite[Theorem 2.40]{dave-etale} and Corollary \ref{cor:repsgood}.
\end{proof}

\begin{definition}
Fix a scheme $Z$. Denote by $\mathcal{P}_Z$ the collection of primes that occur as the residue characteristic of the local ring of a stalk of the structure sheaf of $Z.$
\end{definition}

\begin{remark}\label{rmk:ZS}
Let $f:Y \to Z$ be any morphism of schemes. Then $$\mathcal{P}_Y \subseteq \mathcal{P}_Z.$$ This follows since if $y$ is a point of $Y,$ then the induced map between stalks $\left(\mathcal{O}_Y\right)_y \to \left(\mathcal{O}_Z\right)_{f\left(y\right)}$ is a map of local rings, hence the characteristics of the respective residue fields must agree.
\end{remark}

\begin{lemma}\label{lem:1} Let $Z$ be a connected locally Noetherian scheme. Denote by $\Spc_Z$ the collection of spaces for which 
the map (\ref{eq:z}) is an equivalence. Then $\Spc_Z$ contains
\begin{itemize}
\item[a)] $\Spc_{\mathcal{P}_Z}$
\item[b)] $\Spc_{\mathcal{P}_Z}^{nil}$
\item[c)] $\Spc^{\ell\mbox{-}\pi}$ for all $\ell \notin \mathcal{P}_Z.$
\end{itemize}
\end{lemma}

\begin{proof}
Firstly, by Proposition \ref{prop:nilbig}, $b) \Rightarrow c),$ so we will start by proving $b).$ Note that, by an analogous argument to that of Proposition \ref{prop:sa1good}, $\Spc_Z$ has finite limits and is idempotent complete. We will first show that it contains all Eilenberg-Maclane spaces $K\left(\bZ/\ell,n\right),$ for $n >0,$ and $\ell \notin \mathcal{P}_Z$ a prime. In fact, let us prove something stronger, namely that for $n \geq 1$ that $K\left(G, n\right)$ is in $\Spc_Z$ where $G$ is a finite (not necessarily abelian when $n=1$) group of cardinality coprime to $p,$ for all $p \in \mathcal{P}_Z.$ From \cite[Lemma 4.10]{dave-etale}, this follows from the $\bA^1$-invariance cohomology groups: $$H^i\left(Z,G\right) \rightarrow H^i\left(Z \times \bA^1,G\right)$$ for the above $G$, c.f. \cite[Corollary VI.4.20]{milne-book} for $G$ abelian and $i \geq 1$ and \cite[Proposition 4.3.1]{morel-voevodsky} for $i =1$ and $G$ not necessarily abelian. Secondly, since $Z$ is connected and locally Noetherian, the $\i$-topos $\Shv\left(Z_{\et}\right)$ is connected and locally connected, and similarly for $\Shv\left(\left(Z\times\bA^1\right)_{\et}\right).$ Therefore, by \cite[Proposition 3.26]{dave-etale}, $\Gamma_Z$ and $\Gamma_{Z\times \bA^1}$ both preserve coproducts, and hence $\Spc_Z$ is closed under coproducts. Hence, to prove that $\Spc_Z$ contains all spaces in $b),$ it suffices to prove so for connected such spaces. (Note also it clearly contains the terminal object and hence contains all $0$-truncated spaces). We will proceed by induction on homotopy dimension. Suppose that the claim is established for all connected spaces in $b)$ which are $n$-truncated. Let $X$ be an $\left(n+1\right)$-truncated space in $b).$ Such a space is nilpotent, so by \cite[Proposition 6.2]{goerss-jardine}, we can refine the map $$\tau_{\leq n-1}:X \to X_n$$ in the Postnikov tower of $X$ to a finite composition
$$X=Y_k \to Y_{k-1} \to \cdots Y_1 \to Y_0=X_{n},$$
such that for all  $i,$ we have a pullback square
$$\xymatrix{Y_i \ar[r] \ar[d] & \ast \ar[d]\\
Y_{i-1} \ar[r] & K\left(A_i,n+2\right),}$$
where $A_i$ is an abelian subquotient of a homotopy group of $X.$ Observe that the order of $A_i$ is coprime to $p$ for all $p \in \mathcal{P},$ so $K\left(A_i,n+2\right)$ is in $\Spc_Z.$ Notice that when $i=1,$ we have a pullback square
$$\xymatrix{Y_1 \ar[r] \ar[d] & \ast \ar[d]\\
X_{n} \ar[r] & K\left(A_1,n+2\right),}$$
and since $X_n,$ $\ast$ and $K\left(A_1,n+2\right)$ are all in $\Spc_Z,$ so is $Y_1.$ Continuing by induction, we conclude that $Y_k=X$ is also in $\Spc_Z.$ This establishes that $\Spc_Z$ contains all the spaces in $b).$

Next, we will prove that $\Spc_Z$ contains $\Spc_{\mathcal{P}_Z}.$ We will proceed in several steps. Firstly, we note that if $A$ is any finite group such that $\Aut\left(A\right)$ is of cardinality coprime to $p$ for all $p \in \mathcal{P}_Z$, then $K\left(\Aut\left(A\right), 1\right)$ is in $\Spc_Z$, by \cite[Proposition 4.3.1]{morel-voevodsky}. Secondly, we claim that for $n \geq 1$ and $A$ as above, $\B\Aut K\left(A, n\right)$ is also in $\Spc_Z$. Indeed, by an analogous argument to \cite[Proposition 4.11]{dave-etale}, in order to show that each $\B\Aut\left(K\left(A, n\right)\right)$ is in $\Spc_Z,$ where $A$ is a finite abelian group of cardinality coprime to $p$ for all $p \in \mathcal{P}_Z,$ it suffices to note that cohomology with coefficients in any $A$-local system is $\bA^1$-invariant, which follows by the full strength of \cite[Corollary VI.4.20]{milne-book}. Thirdly, we claim that $\Spc_Z$ contains all connected spaces that are in $\Spc_{\mathcal{P}_Z}.$ We proceed by induction on the homotopy dimension:

For the base case, we note that a connected, $1$-truncated space satisfying Conditions 1,2 and 3 of~\ref{defn:spcp} can only be a $K\left(G, 1\right)$ where $G$ is of cardinality prime to $p$ for all $p \in \mathcal{P}_Z$; this space is already in $\Spc_Z$ as we saw above. Suppose that that $\Spc_Z$ contains all connected $\left(n-1\right)$-truncated spaces satisfying conditions (1)-(3) of Definition~\ref{defn:spcp}, and let $V$ be another such connected space which is $n$-truncated. We will show that $V$ is also in $\Spc_Z.$ Let $$\tau_{\leq n-1}: V \rightarrow V_{n-1}$$ be the $\left(n-1\right)^{\mbox{\tiny st}}$-truncation, which has homotopy fiber equivalent to $K\left(\pi_n\left(V\right), n\right)$ where $\pi_nV$ is abelian. Then we get the following pullback square by Proposition~\ref{prop:univfib}:

$$\xymatrix{
V \ar[d]_{\tau_{\leq n-1}} \ar[r] & K\left(\Aut\left(\pi_nV\right), 1\right) \ar[d]\\
V_{n-1} \ar[r] & \B\Aut\left(K\left(\pi_nV, n\right)\right)}$$
Hence, since $\B\Aut\left(K\left(\pi_nV, n\right)\right)$ is indeed in $\Spc_Z,$ we are done since $V_{n-1}$ is also by induction hypothesis, and $K\left(\Aut\left(\pi_nV\right),1\right)$ is too by the first part of the proof, and $\Spc_Z$ is closed under finite limits. Finally, since $\Spc_Z$ is closed under coproducts, it contains a space if and only if it contains each of its connected components, hence it contains all of $\Spc_{\mathcal{P}_Z}.$
\end{proof}

\begin{corollary}\label{cor:1}
Let $Z$ be a locally Noetherian scheme, which is not necessarily connected. Then the conclusion of Lemma \ref{lem:1} still holds.
\end{corollary}

\begin{proof}
The proof is similar to \cite[Proposition 4.12]{dave-etale}: Firstly, write $$Z=\underset{\alpha} \coprod Z_{\alpha},$$ with each $Z_{\alpha}$ a connected scheme, which is possible since $Z$ being locally Noetherian implies locally connected. Let $V$ be a space satisfying the hypothesis of Lemma \ref{lem:1}. Then, since each scheme $Z_\alpha$ is also locally Noetherian and connected, we have that the canonical map $$\Pi^{\et}_\i\left(Z_\alpha\right)\left(V\right) \to  \Pi^{\et}_\i\left(Z_\alpha \times \bA^1\right)\left(V\right)$$ is an equivalence. Unwinding the definitions, this means that the canonical map
$$\map_{\Shv\left(Z_{\et}\right)}\left(Z_\alpha,\Delta_{Z_\alpha}\left(V\right)\right) \to \map_{\Shv\left(\left(Z_\alpha \times \bA^1\right)_{\et}\right)}\left(Z_\alpha \times \bA^1,\Delta_{Z_\alpha \times \bA^1}\left(V\right)\right)$$ 
is an equivalence. Now consider the following natural string of equivalence:
\begin{eqnarray*}
\Gamma_{Z}\Delta_Z\left(V\right)&=& \map_{\Shv\left(Z_{\et}\right)}\left( \underset{\alpha} \coprod Z_{\alpha},\Delta_{Z}\left(V\right)\right)\\
&\simeq& \underset{\alpha} \prod \map_{\Shv\left(Z_{\et}\right)}\left(Z_\alpha,\Delta_{Z}\left(V\right)\right)\\
&\simeq& \underset{\alpha} \prod \map_{\Shv\left(\left(Z_\alpha\right)_{\et}\right)}\left(Z_\alpha,\Delta_{Z_\alpha}\left(V\right)\right)\\
&\simeq& \underset{\alpha} \prod \map_{\Shv\left(\left(Z_\alpha \times \bA^1\right)_{\et}\right)}\left(Z_\alpha \times \bA^1,\Delta_{Z_\alpha\times \bA^1}\left(V\right)\right)\\
&\simeq& \map_{\Shv\left(\left(Z_\alpha \times \bA^1\right)_{\et}\right)}\left(\underset{\alpha} \coprod Z_\alpha \times \bA^1,\Delta_{Z \times \bA^1}\left(V\right)\right)\\
&=& \Gamma_{Z \times \bA^1}\Delta_{Z \times \bA^1}\left(V\right).
\end{eqnarray*}
\end{proof}

Finally we see that the $\infty$-category $\Spc_{S,\bA^1}$ contains a lot of interesting objects. 

\begin{theorem} \label{thm:spc-s-contains}
Let $S$ be a locally Noetherian scheme. Then $\Spc_{S,\bA^1}$ is closed under finite limits and retracts in $\Spc$ and contains
\begin{itemize}
\item[a)] $\Spc_{\mathcal{P}_S}$
\item[b)] $\Spc_{\mathcal{P}_S}^{nil}$
\item[c)] $\Spc^{\ell\mbox{-}\pi}$ for all $\ell \notin \mathcal{P}_S.$
\end{itemize}
\end{theorem}

\begin{proof}
The first statement is a restatement of Proposition \ref{prop:sa1good}. To prove the rest, by Lemma \ref{lem:smalliscooltoo} and Remark\ref{rmk:ZS}, it suffices to prove that $\Spc_{S,\bA^1}$ contains $\Spc_Z$ and $\Spc_{\mathcal{P}_Z}^{nil}$ for all $S$-schemes $Z$ of finite type. Note that such a $Z$ is automatically also locally Noetherian by Lemma 28.14.6 of \cite[Tag 01T6]{stacks-project}. The result now follows from Corollary \ref{cor:1}. 
\end{proof}




\subsubsection{Absolute \'{e}tale realization}\label{subsec:absolute}






We are now ready to define the \'{e}tale realization functor in the absolute situation:

\begin{theorem} Let $S$ be a scheme, then there exists a colimit preserving functor $$\Et_{\bA^1}: \Spc\left( S \right) \rightarrow \Pro\left(\Spc_{S,\bA^1} \right)$$ whose value on $\LL_{\mot}\left(X\right)$ where $X$ is a smooth $S$-scheme is given by the $\Spc_{S,\bA^1}$-localization of the \'etale homotopy type of $X$: $$\widehat{\Pi^{\et}_\i\left(X\right)}_{\Spc_{S,\bA^1}}.$$
\end{theorem}\label{thm:absolute_realization}


\begin{proof}
Consider the fully faithful inclusion $$q:\Sm_S \hookrightarrow \Sch^{\ft}_S$$ of smooth finite type $S$-schemes into all finite type $S$-schemes, and consider furthermore its composite with the Yoneda embedding $$y \circ q: \Sm_S \hookrightarrow \Shv_{\et}\left(\Sch^{\ft}_S\right).$$

The left Kan extension of this composite along the Yoneda embedding $$y_{\Sm}:\Sm_S\hookrightarrow \PShv\left(\Sm_S\right),$$ namely $$\LKE_{y_{\Sm}}\left(y\circ q\right):\PShv\left(\Sm_S\right) \to \Shv_{\et}\left(\Sch^{\ft}_S\right)$$ has a right adjoint $R_q$ given by restriction along $q,$ i.e. for an object $\cY$ of $\Shv_{\et}\left(\Sch^{\ft}_S\right)$ and a smooth $S$-scheme $X,$ $R_q\left(\cY\right)\left(X\right)=\cY\left(X\right).$ Any such $R_q\left(\cY\right)$ satisfies \'etale descent, and therefore also Nisnevich descent, and hence there is an induced adjunction $\LL_{\et} \dashv R_{\et},$ with $$\LL_{\et}:\Shv_{\Nis}\left(\Sm_S\right) \to \Shv_{\et}\left(\Sch^{\ft}_S\right).$$ The functor $\LL_{\et},$ by construction, sends any smooth $S$-scheme to itself; in fact, it is the composite $$\Shv_{\Nis}\left(\Sm_S\right) \hookrightarrow \PShv\left(\Sm_S\right) \stackrel{\LKE_{y_{\Sm}}\left(y\circ q\right)}{\longlonglongrightarrow} \Shv_{\et}\left(\Sch^{\ft}_S\right).$$ We claim that $\LL_{\et}$ is furthermore left-exact. For this, it suffices to prove that the left Kan extension $\LKE_{y_{\Sm}}\left(y\circ q\right)$ is. Since the category $\Sm_S$ has finite limits, appealing to \cite[Proposition 6.1.5.2]{htt}, it suffices to observe that these finite limits are preserved by the functor $y\circ q.$ In particular, we conclude that for any Nisnivech sheaf of spaces $\cX,$ 
\begin{equation}\label{eq:oknis}
\LL_{\et}\left(\cX \times \mathbb{A}^1\right) \simeq \LL_{\et}\left(\cX\right) \times \mathbb{A}^1.
\end{equation}
Denote by $\widetilde{\Et_{\bA^1}}$ the composite of colimit preserving functors
$$\xymatrix@C=2cm{\Shv_{\Nis}\left(\Sm_S\right) \ar[r]^-{\LL_{\et}} & \Shv_{\et}\left(\Sch^{\ft}_S\right) \ar[r]^-{\Pi^{\et}_\i} & \Pro\left(\Spc\right) \ar[r]^-{\cs} & \Pro\left(\Spc_{S,\bA^1}\right).}$$
Recall that the $\i$-category $\Spc\left(S\right)$ of motivic spaces over $S$ is the localization $$\xymatrix@1{\Spc\left(S\right)\mspace{4mu} \ar@{^{(}->}[r]<-0.9ex> & \Shv_{\Nis}\left(\Sm_S\right) \ar@<-0.5ex>[l]_-{\LL_{\mathbb{A}^1}}}$$
of the presentable $\i$-category $\Shv_{\Nis}\left(\Sm_S\right)$ at the $\mathbb{A}^1$-local equivalences. Hence, by \cite[Proposition 5.5.4.20]{htt}, composition with $\LL_{\mathbb{A}^1}$ induces a fully faithful functor $$\Fun^{\LL}\left(\Spc\left(S\right),\Pro\left(\Spc_{S,\bA^1}\right)\right) \hookrightarrow \Fun^{\LL}\left(\Shv_{\Nis}\left(\Sm_S\right),\Pro\left(\Spc_{S,\bA^1}\right)\right)$$ whose essential image is precisely those colimit preserving functors $\Shv_{\Nis}\left(\Sm_S\right) \to \Pro\left(\Spc_{S,\bA^1}\right)$ which send $\mathbb{A}^1$-local equivalences to equivalences. We claim that there is an essentially unique colimit preserving functor $\Et_{\bA^1}$ for which the following diagram commutes up to equivalence:
$$\xymatrix{\Spc\left(S\right) \ar[r]^-{\Et_{\bA^1}} & \Pro\left(\Spc_{S,\bA^1}\right).\\
\Shv_{\Nis}\left(\Sm_S\right) \ar[u]^-{\LL_{\mathbb{A}^1}} \ar[ru]_-{\widetilde{\Et_{\bA^1}}} &}$$
By the above, it suffices to show that $\widetilde{\Et_{\bA^1}}$ sends $\mathbb{A}^1$-local equivalences to equivalences, i.e. we need to show that for all $\cX$ in $\Shv_{\Nis}\left(\Sm_S\right),$ the induced map $$\widetilde{\Et_{\bA^1}}\left(\cX\times \mathbb{A}^1\right) \to \widetilde{\Et_{\bA^1}}\left(\cX\right)$$ is an equivalence. In light of (\ref{eq:oknis}), this means we need to show that the induced map $$\Pi^{\et}_\i\left(\LL_{\et}\left(\cX\right) \times \mathbb{A}^1\right) \to \Pi^{\et}_\i\left(\LL_{\et}\left(\cX\right) \right)$$ is an equivalence after $\Spc_{S,\mathbb{A}^1}$-localization. Using Theorem \ref{thm:2.40} and unwinding definitions, this is equivalent to checking that for all $V$ in $\Spc_{S,\mathbb{A}^1},$ the induced map
$$\map_{\Shv_{\et}\left(\Sch^{\ft}_S\right)}\left(\LL_{\et}\left(\cX\right),\Delta\left(V\right)\right) \to \map_{\Shv_{\et}\left(\Sch^{\ft}_S\right)}\left(\LL_{\et}\left(\cX\right) \times \mathbb{A}^1,\Delta\left(V\right)\right)$$ is an equivalence of spaces, which is true since, by definition of $\Spc_{S,\mathbb{A}^1},$ $\Delta\left(V\right)$ is $\mathbb{A}^1$-invariant.
\end{proof}

\begin{remark} \label{rmk:isaksen}

We remark that this approach lets us substantially enlarge the target of the realization functor; Isaksen's realization functor in \cite{realization} lands in the $p$-complete category for a fixed prime $p$ which is invertible in the residue characteristics of $S$. Indeed, if $p$ is a fixed prime which is not in $\mathcal{P}_S$ then point (c) of Theorem~\ref{thm:spc-s-contains} gives us an inclusion
\[
\Spc^{p\mbox{-}\pi} \hookrightarrow \Spc_{S,\bA^1}.
\]
By Proposition \ref{prop:ladj} we have a functor 
\begin{equation} \label{eq:at-p}
\Pro\left(\Spc_{S,\bA^1}\right) \rightarrow \Pro\left( \Spc^{p\mbox{-}\pi} \right)
\end{equation} with a fully faithful right adjoint. The functor constructed in~\cite[Corollary 2.7]{realization} is then obtained by postcomposing the functor in Theorem~\ref{thm:absolute_realization} with~\eqref{eq:at-p}.
\end{remark}

\begin{remark}\label{rmk:Li1}
Unwinding the definitions, we see that for $\cX$ a motivic space, $$\Et_{\bA^1}\left(\cX\right)=\reallywidehat{\Pi_{\i}^{\et}\left(\LL^{\et}i\left(\cX\right)\right)}_{\Spc_{S,\bA^1}},$$ where $$i:\Spc\left( S \right) \hookrightarrow \Shv_{\Nis}\left(\Sm_S\right)$$ is the canonical inclusion.
\end{remark}






\section{Relative \'Etale Realization of Motivic Spaces} \label{sect:relreal}

While the above section refines Isaksen's result, we will now produce a new functor whose existence was anticipated in \cite{realization}. In the case that the base scheme $S = \Spec\,k$, the functor goes from the $\i$-category of motivic spaces to (a localization of) the $\i$-category of pro-objects in $\Gal\left(k^{\sep}/k\right)$-spaces, where $\Gal\left(k^{\sep}/k\right)$ is the profinite absolute Galois group. 

We first discuss a convenient approach to relative \'{e}tale realization via $\infty$-categories. The original approach \cite{harpaztomer}, which was developed for varieties over a field in order to study obstructions to the existence of rational points, was close in spirit to the original (non-relative) construction of Artin-Mazur \cite{ArtinMazur}. It was later explained more succinctly as a derived functor in \cite{barneatomer}, and in the greater generality of a map of topoi; it was explained how a map of topoi $f:\cE \to \cF$ naturally gives rise to a pro-object in the hypercomplete $\infty$-topos associated  to $\cF,$ using the language of model categories. In subsection~\ref{subsec:relshape}, we give a uniform approach by developing the general theory of relative shapes of $\i$-topoi; here it becomes quite convenient using the language of $\i$-categories.

In subsection~\ref{subsec:explicit} we offer an explicit formula for the relative \'etale homotopy type as a pro-object. Finally, in subsection~\ref{subsec:rel-real}, we discuss the main construction of this paper --- the relative \'etale realization of motivic spaces.

\subsection{Relative shapes}\label{subsec:relshape}

\subsubsection{The relative $\Shape$ functor}\label{subsec:relshapefunctor}

We will now extend the construction of the shape of an $\i$-topos (over the terminal $\i$-topos $\Spc$) to the setting of an arbitrary base. This construction is probably well known, but we could not find it it in its general form in the literature. Explicitly, for $\cE$ an $\i$-topos, we will construct an adjunction (Proposition~\ref{prop:shape-adj})
$$\Shape_{\cE}:\Topi/\cE \rightleftarrows \Pro\left(\cE\right):\cE/\left(\mspace{3mu}\cdot\mspace{3mu}\right)$$ which reduces to that of Proposition \ref{prop:ladjfun} when $\cE=\Spc.$

Using the universal property of the $\i$-category $\Pro\left(\cE\right),$ the composite
$$\cE \stackrel{\chi_\cE}{\longlongrightarrow} \Topi^{\et}/\cE \to \Topi/\cE$$
extends to an essentially unique cofiltered limit preserving functor $$\cE/\left(\mspace{3mu}\cdot\mspace{3mu}\right):\Pro\left(\cE\right) \to \Topi/\cE.$$

\begin{proposition} \label{prop:shape-adj}
The above functor $\cE/\left(\mspace{3mu}\cdot\mspace{3mu}\right)$ has a left adjoint $\Shape_{\cE}$ described as follows:\\
Let $$f: \sE \rightarrow \sF$$ be a geometric morphism corresponding to the adjunction: 
$$\Adj{f^*}{\cE}{\cF}{f_*},$$ with $f^* \dashv f_*.$
Then $\Shape_{\cE}\left(f\right)$ can be identified with the left exact functor 
\begin{equation} \label{eq:shape-def}
\cE \stackrel{f^*}{\longrightarrow} \cF \stackrel{\Gamma_{\cF}}{\longlongrightarrow} \Spc.
\end{equation}
\end{proposition}
\begin{proof}
Firstly, remark that both functors participating in any geometric morphism must be accessible, hence $\Gamma_{\cF} \circ f^*$ is left exact and accessible, therefore it can be identified with a pro-object. We proceed by defining $\Shape_{\cE}$ via the formula in~\eqref{eq:shape-def} and verifying that it participates in the claimed adjunction.

Indeed, suppose that $f:\cF \to \cE$ is a geometric morphism. Let $T = \underset{j} \lim E_j$ be a pro-object in $\cE.$ Then we have the following string of natural equivalences:
\begin{eqnarray*}
\map_{\Topi/\cE}\left(f,\cE/T\right) &\simeq& \underset{i} \lim  \map_{\Topi/\cE}\left(f,\cE/E_i\right)\\
&\simeq& \underset{i} \lim \map_{\Topi/\cF}\left(\cF,\cF/f^*E_i\right)\\
&=& \underset{i} \lim \map_{\Topi^{\et}/\cF}\left(\cF,\cF/f^*E_i\right)\\
&\simeq& \underset{i} \lim \map_{\cF}\left(1_\cF,f^*E_i\right)\\
&\simeq& \underset{i} \lim \Gamma_{\cF} \left(f^*E_i\right)\\
&\simeq&  \underset{i} \lim \map_{\Pro\left(\cE\right)}\left(\Shape_{\cE}\left(f\right),j\left(E_i\right)\right)\\
&\simeq&  \map_{\Pro\left(\cE\right)}\left(\Shape_{\cE}\left(f\right),T\right).
\end{eqnarray*}
Indeed, the second equivalence follows since we have a pullback diagram of $\i$-topoi \cite[Remark 6.3.5.8]{htt}:
$$\xymatrix{ \cF/f^*E_i \ar[d] \ar[r] & \cE/E_i \ar[d]\\ \cF \ar[r]^-{f} & \cE},$$ the third equivalence follows since any morphism in $\Topi/\cE$ between \'etale geometric morphisms must be \'etale \cite[Corollary 6.3.5.9]{htt}, and finally, the second to last equivalence follows from the Yoneda lemma.

Consequently, we conclude that the functor
$$\map_{\Topi/\cE}\left(f,\cE/\left(\mspace{3mu}\cdot\mspace{3mu}\right)\right):\Pro\left(\cE\right) \to \widehat{\Spc},$$ is corepresentable by the pro-object $\Gamma_{\cF} \circ f^*,$ where $\widehat{\Spc}$ is the $\i$-category of large spaces. It follows that the dotted functor below exists
$$\xymatrix@C=2.5cm@R=2.5cm{\Topi/\cE \ar@{^{(}->}[r]^-{y} \ar@{-->}[rrd]^{\Shape_{\cE}} & \left(\Fun\left(\Topi/\cE,\widehat{\widehat{\Spc}}\right)\right)^{op}  \ar[r]^-{\left(\left(\cE/\left(\mspace{3mu}\cdot\mspace{3mu}\right)\right)^*\right)^{op}} & \left(\Fun\left(\Pro\left(\cE\right),\widehat{\widehat{\Spc}}\right)\right)^{op}\\
&  & \Pro\left(\cE\right) \ar@{^{(}->}[u]^-{y},}$$
where $\widehat{\widehat{\Spc}}$ denotes the $\i$-category of very large spaces.

Finally, there is a natural transformation $$\eta:id_{\Topi/\cE} \Rightarrow \cE/\Shape_{\cE}$$ whose component at $f$ under the equivalence
$$\map_{\Pro\left(\cE\right)}\left(\Shape_{\cE}\left(f\right),\Shape_{\cE}\left(f\right)\right) \simeq \map_{\Topi/\cE}\left(f,\cE/\Shape_{\cE}\left(f\right)\right)$$ corresponds to $id_{\Shape_{\cE}\left(f\right)}.$ We conclude that $\eta$ is a unit transformation exhibiting the desired adjunction.
\end{proof}

\begin{definition} Given a geometric morphism $f: \cF \rightarrow \cE$, we define the \emph{relative shape} of $f$ as the pro-object in $\cE$ $$\Shape_{\cE}\left(f\right).$$
\end{definition}
The case of greatest interest for the rest of the paper is when $\cE = \Shv \left( S_{\et} \right)$ where $S$ is a scheme. We shall denote the relative shape of a geometric morphism $f: \cF \rightarrow \Shv \left( S_{\et} \right)$ simply as $\Shape_S \left( f\right)$.

\begin{proposition} The functor $f^*: \cE \rightarrow \cF$  induces a functor $$\Pro\left(f^*\right): \Pro\left(\cE\right) \rightarrow \Pro \left(\cF\right).$$

This functor admits a left adjoint: $$\Pro\left(f_!\right): \Pro\left(\cF\right) \rightarrow \Pro \left(\cE\right)$$ and a right adjoint: $$\Pro\left(f_*\right): \Pro\left(\cF\right) \rightarrow \Pro\left(\cE\right).$$ 
\end{proposition}

\begin{proof} By functoriality of the $\Pro$ construction, the right adjoint $f_*$ of $f^*$ induces a functor $$\Pro\left(f_*\right): \Pro\left(\cF\right) \rightarrow \Pro\left(\cE\right)$$ which is easily verified to be right adjoint to $\Pro\left(f^*\right).$

Unlike above, despite the notation, $\Pro\left(f_!\right)$ is not, in general, the functor $\Pro$ applied to a left adjoint of $f^*$ (but it will be when such a left adjoint exists). We will describe the construction of this left adjoint by starting with the construction of a functor $\cF \rightarrow \Pro\left(\cE\right):$

$$\xymatrix@C=2.1cm{
\cF \ar[r]^-{\chi_{\cF}} & \Topi^{\et}{/\cF} \ar[r] & \Topi{/\cF} \ar[r]^-{\hat{f}} & \Topi{/\cE} \ar[r]^-{\Shape_{\cE}} & \Pro\left(\cE\right), 
}$$
where  $\chi_{\cF}$ is induced by taking slice topoi and $\hat{f}$ is induced by composition with $f$. Unwinding definitions, this sends an object $F$ of $\cF$ to the pro-object corresponding the left exact functor 
\begin{eqnarray*}
\cE & \rightarrow & \Spc\\
 E & \mapsto & \Gamma_{\sF}\left(\pi_{F}^*\left(f^*E\right)\right),
\end{eqnarray*}
where $\pi_F: \cF/F \rightarrow \cF$ is the canonical \'{e}tale geometric morphism. 

By the universal property of $\Pro\left(\cF\right)$, this extends essentially uniquely to a cofiltered limit preserving functor $$\Pro\left(f_!\right): \Pro\left(\cF\right) \rightarrow \Pro\left(\cE\right).$$
We claim that this is the desired left adjoint. To see this, we consider the following sequence of equivalences which are natural in each variable: write $T \in \Pro\left(\cF\right)_0$ as $T =\underset{i} \lim j\left(F_i\right)$ where each $F_i$ is an object of $\cF$ and $Z = \underset{k}\lim j\left(E_k\right)$ be a pro-object in $\cE$, then
\begin{eqnarray*}
\map_{\Pro\left(\cE\right)}\left(\Pro\left(f_!\right)\left(T\right), Z\right) & \simeq & \underset{k} \lim \map_{\Pro\left(\cE\right)}\left(\Pro\left(f_!\right)\left(T\right), j\left(E_k\right)\right)\\
 & \simeq & \underset{i} \colim \underset{k} \lim \map_{\Pro\left(\cE\right)}\left(\Pro\left(f_!\right)\left(j\left(F_i\right)\right), j\left(E_k\right)\right)\\
& \simeq & \underset{i} \colim \underset{k} \lim \left(\Pro\left(f_!\right)\left(j\left(F_i\right)\right)\left(E_k\right)\right)\\
& \simeq & \underset{i} \colim \underset{k} \lim \Gamma_{\sF}\left(\pi_{F_i}^*\left(f^*E_k\right)\right)\\
& \simeq &  \underset{i} \colim \underset{k} \lim \map_{\cF/F}\left(1, \pi^*_{F_i}f^*E_k\right)\\
& \simeq & \underset{i} \colim \underset{k} \lim \map_{\cF/F}\left(id_{F_i}, F_i \times f^*\left(E_k\right) \to F_i\right)\\
& \simeq & \underset{i} \colim \underset{k} \lim \map_{\cF}\left(F_i, f^*E_k\right) \\
& \simeq & \map_{\Pro\left(\cF\right)}\left(T, \Pro\left(f^*\right)\left(Z\right)\right).
\end{eqnarray*}
\end{proof}

\begin{remark} \label{remark:shriek} By construction, given a geometric morphism $f: \cF \rightarrow \cE$, 
$$\Pro\left(f_! \right)\left(1_{\cF} \right) \simeq \Shape_{\cE}\left(f\right).$$ In particular, when $f^*$ admits a left adjoint $f_!$ (i.e. when the geometric morphism $f$ is locally $\i$-connected), then the relative shape of $f$ is corepresented by the object $f_!\left(1\right)$ of $\cE.$
\end{remark}

We will now describe how relative shapes interact with each other.

\begin{proposition} \label{prop:composable} Given composable geometric morphisms $\cG \stackrel{g}{\rightarrow} \cF \stackrel{f}{\rightarrow} \cE$, there is an equivalence $$\Shape_{\cE}\left(f \circ g\right) \simeq \Pro\left(f_!\right)\left(\Shape_{\cF}\left(g\right)\right).$$
\end{proposition}

\begin{proof} By Remark \ref{remark:shriek}, $$\Shape_{\cE}\left(f \circ g\right) \simeq \Pro\left( \left(f \circ g \right)_! \right)\left(1_{\cG}\right),$$ and $$\Pro\left(f_!\right)\left(\Shape_{\cF}\left(g\right)\right) \simeq \Pro\left( f_! \right) \left( \Pro\left( g_! \right)\left(1_{\cG} \right)\right).$$ 
Notice that $\Pro\left(f_!\right) \circ \Pro\left(g_!\right)$ and $\Pro\left( \left(f\circ g\right)_! \right)$ are both left adjoints to $\Pro\left( \left(f \circ g\right)^* \right).$ The result now follows by uniqueness of adjoints. 
\end{proof}

\subsubsection{Relative \'etale homotopy types of (higher) algebraic stacks}\label{sect:relative-types-stacks}
For a base scheme $S,$ recall that $\Sch^{\ft}_S$ is the category of $S$-schemes of finite type. The colimit preserving functor
$$\Shv\left(\left(\mspace{3mu}\cdot\mspace{3mu}\right)_{\et}\right):\Shv_{\et}\left(\Sch^{\ft}_S\right)\to \Topi$$ of \cite[Lemma 2.2.6]{dave-etale}, is used in an essential way to construct the \'etale homotopy type functor of \cite[Definition 2.3.1]{dave-etale}:
$$\Pi^{\et}_\i:\Shv\left(\left(\mspace{3mu}\cdot\mspace{3mu}\right)_{\et}\right) \to \Pro\left(\Spc\right),$$
namely it is simply $\Shape \circ \Shv\left(\left(\mspace{3mu}\cdot\mspace{3mu}\right)_{\et}\right)$\footnote{In \cite{dave-etale}, it is assumed that $S$ is affine, but this is assumed merely for simplicity, and is not necessary.}. Even though $S$ is the terminal $S$-scheme, there is in general no reason for $\Shv\left(S_{\et}\right)$ to be the terminal $\i$-topos $\Spc.$ Therefore the above construction can be refined using the relative shape functor with respect to the $\i$-topos $\Shv\left(S_{\et}\right)$. For example, if $S=\Spec\left(k\right),$ for $k$ a field, then $\Shv\left(S_{\et}\right) \simeq \cB \Gal\left(k^{\sep}/k\right)$ is the classifying $\i$-topos for the absolute Galois group of $k.$ In particular, this $\i$-topos is only terminal if $k$ is separably closed. By functoriality, the functor $\Shv\left(\left(\mspace{3mu}\cdot\mspace{3mu}\right)_{\et}\right)$ naturally induces a colimit preserving functor
$$\Shv\left(\left(\mspace{3mu}\cdot\mspace{3mu}\right)_{\et}\right)/S:\Shv_{\et}\left(\Sch^{\ft}_S\right)\to \Topi/\Shv\left(S_{\et}\right).$$
Explicitly, given an $\i$-sheaf $\cX$ on the big \'etale site of $\Sch^{\ft}_S,$ $\cX$ gets sent to the image of the unique morphism $\cX \to S$ under the functor $\Shv\left(\left(\mspace{3mu}\cdot\mspace{3mu}\right)_{\et}\right).$

\begin{definition} \label{defn:rel-et-type}
The \textbf{$S$-relative \'etale homotopy type functor}
$$\Pi^{S,\et}_{\i}:\Shv_{\et}\left(\Sch^{\ft}_S\right) \to \Pro\left(\Shv\left(S_{\et}\right)\right)$$
is defined to be $\Shape_S \circ \Shv\left(\left(\mspace{3mu}\cdot\mspace{3mu}\right)_{\et}\right)/S.$ 

When $S=\Spec\left(k\right)$ for a field $k,$ we denote this functor by
$$\Pi^{Gal_k,\et}_{\i}:\Shv_{\et}\left(\Sch^{\ft}_k\right) \to \Pro\left(\mathcal{B}Gal\left(k^{\sep}/k\right)\right),$$ and it will be called the \textbf{Galois-equivariant \'etale homotopy type functor} (over $k.$)
\end{definition}

\begin{remark}\label{rmk:recover}
Since all algebraic Artin stacks locally of finite type over $S$ are naturally objects in $\Shv_{\et}\left(\Sch^{\ft}_S\right),$ the functor $\Pi^{S,\et}_{\i}$ naturally associates to such algebraic stacks a pro-object in $\Shv\left(S_{\et}\right).$ By Proposition \ref{prop:composable}, one can recover the ordinary \'etale homotopy type from this pro-object, so this definition gives a more refined invariant than the non-relative \'etale homotopy type of an algebraic stack. For example if $p: X \rightarrow S$ is a morphism of schemes, the composite of the geometric morphisms $$\Shv\left(X_{\et} \right) \stackrel{p}{\longrightarrow} \Shv\left( S_{\et} \right) \stackrel{e_S}{\longrightarrow} \Spc,$$ where $e_S$ is the essentially unique geometric morphism from $\Shv\left( S_{\et} \right)$ to the terminal $\infty$-topos, is the essentially unique geometric morphism $e_X$. Proposition~\ref{prop:composable} then tells us that

\begin{eqnarray*}
\Pro \left(e_S\right)_! \left( \Pi^{S,\et}_{\i}\left(X\right) \right) & =  & \Pro\left(e_S\right)_!\left(\Shape_S\left(p\right)\right)\\
& \simeq & \Shape \left(\Shv\left(X_{\et} \right) \right) \\
&  = & \Pi^{\et}_{\i}\left(X\right).
\end{eqnarray*}
\end{remark}

\subsubsection{The case of a field}
When the base scheme $S$ is a field $k$, the $S$-relative \'etale homotopy type functor encodes an action of the absolute Galois group $\Gal\left(k^{\sep}/k\right).$ In this subsection, we will explain in what way this action is encoded, and in what sense the ordinary \'etale homotopy type of a scheme $X$ is a homotopy quotient of its relative \'etale homotopy type by this action. 

The origin of the action boils down to the classical equivalence of categories $$\Gal\left(k^{\sep}/k\right)\mbox{-}\Set \simeq \Shv^{\delta}\left(k_{\et}\right),$$ which we will now recall. Fix an embedding $\varphi:k \hookrightarrow k^{sep}$ of $k$ into a separable closure. The group $\Gal\left(k^{\sep}/k\right)$ is the cofiltered limit of $\Gal\left(k_i/k\right)$ with $i$ ranging over all finite Galois extensions of $k,$ and hence is profinite. To simplify notation, we will denote this profinite group simply by $G_k$ for the rest of this subsection. The category $G_k\mbox{-}\Set$ is the category of sets equipped with a continuous $G_k$-action, where the sets are endowed with the discrete topology, i.e. each stabilizer group of the actions must be open with respect to the profinite topology. This category is the \emph{classifying topos} for the profinite group $G_k.$ Given any $k$-scheme $X,$ one may consider the set of $k^{sep}$-points $X\left(k^{sep}\right),$ which caries a canonical continuous $G_k$-action. This construction restricts to a functor $$\omega:k_{\et} \to G_k\mbox{-}\Set$$ from the small \'etale site of $\Spec\left(k\right)$ to the classifying topos of its absolute Galois group. This functor induces an adjoint equivalence of categories. The right adjoint, which we will simply denote by $$R:G_k\mbox{-}\Set \to \Shv^{\delta}\left(k_{\et}\right),$$ is such that for a $G_k$-set $Y,$ $R\left(Y\right)$ is the presheaf on $k_{\et}$
$$R\left(Y\right)\left(X\right) = \Hom_{G_k\mbox{-}\Set}\left(X\left(k^{sep}\right),Y\right),$$ which can be verified to be a sheaf by reducing to covers by finite Galois extensions of $k.$ By the Yoneda lemma, $R$ is right adjoint to the left Kan extension $\Lan_y \omega,$ where $y$ is the  Yoneda embedding of $\Shv^{\delta}\left(k_{\et}\right).$ For further details see \cite[Section 03QW]{stacks-project}.

There is another description of the classifying $\i$-topos $\cB\Gal\left(k^{\sep}/k\right),$ which we now describe. Let $\Orb\left(G_k\right)$ be full subcategory of $G_k\mbox{-}\Set$ on those $G_k$-sets of the form $G_k/U,$ for $U$ an open subgroup. Equip it with the atomic Grothendieck topology, i.e. all non-empty sieves are covering sieves. By \cite[III.9, Theorem 1]{macmoer}, there is a canonical equivalence of categories $$G_k\mbox{-}\Set \simeq \Shv^{\delta}\left(\Orb\left(G_k\right)\right).$$ It is not obvious however that there is an induced equivalence at the level of sheaves of spaces, since the classifying $\i$-topos $\cB \Gal\left(k^{\sep}/k\right)$ is often not hypercomplete, and $\Orb\left(G_k\right)$ is not closed under finite limits, so we prove this by hand. First, we need an important simple observation:

\begin{lemma}\label{lem:gamma}
Let $U$ be an open subgroup of $G_k,$ and let $\left(k^{sep}\right)^U$ be the fixed field of $U.$ Then $$R\left(G_k/U\right)\cong y\left(\Spec\left(\left(k^{sep}\right)^U\right)\right) \in \Shv^{\delta}\left(k_{\et}\right).$$
\end{lemma}

\begin{proof}
It suffices to prove that $$\omega\left(\Spec\left(\left(k^{sep}\right)^U\right)\right) \cong G_k/U.$$ Note that this $G_k$-set is nothing but $\Hom_{k}\left(\left(k^{sep}\right)^U,k^{sep}\right)$ with the canonical $G_k$-action, where the subscript denotes morphisms preserving the ground field $k,$ or more precisely restricting to the distinguished embedding $\varphi:k \hookrightarrow k^{sep}.$ The canonical $G_k$-action on this set is transitive with stabilizer isomorphic to $U$. The result now follows.
\end{proof}

In particular, since $R\left(G_k/U\right)$ is in the image of the Yoneda embedding, the functor $R$ restricts to a fully faithful functor $$\gamma:
\Orb\left(G_k\right) \hookrightarrow k_{\et}.$$

\begin{proposition}
The functor $\gamma$ induces an equivalence of $\i$-categories
$$\gamma^*:\Shv\left(k_{\et}\right) \to \Shv\left(\Orb\left(G_k\right)\right).$$
\end{proposition}

\begin{proof}
Precomposition with $\gamma$ sends sheaves to sheaves, so there is such a functor $\gamma^*$. We will show that this functor has both a left and right adjoint. Note that the essential image of $\gamma$ is all finite Galois field extensions $$\Spec\left(k_i\right) \to \Spec\left(k\right).$$ Lets denote this essential image by $S\left(k_{\et}\right)$ and, by abuse of notation, denote by $\gamma$ its canonical inclusion into $k_{\et}$. First, lets make the observation that any \'etale $k$-scheme $X \to \Spec\left(k\right)$ in $k_{\et}$ is isomorphic to a coproduct of finite Galois field extensions $$\underset{i}\coprod \Spec\left(k_i\right) \to \Spec\left(k\right),$$ and moreover each inclusion into the coproduct is \'etale, so this coproduct is preserved by the Yoneda embedding into sheaves of spaces. We claim that the right adjoint $\gamma_*$ is then defined by
$$\gamma_*\left(F\right)\left(X\right):=\underset{i}\prod F\left(\Spec\left(k_i\right)\right).$$
It is easy to see that if $F$ is a sheaf for the atomic topology on $S\left(k_{\et}\right)$ then $\gamma_*F$ is a sheaf for the \'etale topology, since every cover in $k_{\et}$ factors as a coproduct of atomic covers.

Let $\tilde \gamma^*:\PShv\left(k_{\et}\right) \rightarrow \PShv\left(\Orb\left(G_k\right)\right)$ be the functor on \emph{presheaves} induced by precomposition with $\gamma.$ Then, $\tilde \gamma^*$ has a right adjoint $\tilde \gamma_*$ given by the formula
$$\tilde \gamma_*\left(F\right)\left(X\right)=\map_{\PShv\left(k_{\et}\right)}\left(\tilde \gamma^*y\left(X\right),F\right),$$ where $F$ and the mapping space are taken in presheaves of spaces. If we can show that $\tilde \gamma_*$ sends sheaves to sheaves and restricts to our functor $\gamma_*,$ this will establish the claim.

Suppose that $F$ is now a sheaf, and denote by $$a:\PShv\left(S\left(k_{\et}\right)\right) \to \Shv\left(S\left(k_{\et}\right)\right)$$ the sheafification functor, left adjoint to the inclusion $i.$

Then
\begin{eqnarray*}
\tilde \gamma_*\left(iF\right)\left(X\right) &\simeq& \map_{\PShv\left(k_{\et}\right)}\left(y\left(X\right),\tilde \gamma_*iF\right)\\
&\simeq& \map_{\PShv\left(S\left(k_{\et}\right)\right)}\left(\tilde \gamma^*y\left(X\right),iF\right)\\
&\simeq& \map_{\Shv\left(S\left(k_{\et}\right)\right)}\left(a\tilde \gamma^*y\left(X\right),F\right)\\
&\simeq& \map_{\Shv\left(S\left(k_{\et}\right)\right)}\left(\underset{i}\coprod y\left(\Spec\left(k_i\right)\right),F\right)\\
&\simeq& \underset{i}\prod F\left(\Spec\left(k_i\right)\right)\\
&=& \gamma_*\left(F\right)\left(X\right),
\end{eqnarray*}
and since we know that $\gamma_*$ sends sheaves to sheaves, this establishes our claim.

Denote by $y$ the Yoneda embedding $$y:k_{\et}\hookrightarrow \Shv\left(k_{\et}\right),$$ then the composite $$S\left(G_k\right) \stackrel{\gamma}{\longhookrightarrow} k_{\et} \stackrel{y}{\longhookrightarrow} \Shv\left(k_{\et}\right)$$ extends to a unique colimit preserving functor $$\Lan_y\left(y\circ \gamma\right):\PShv\left(S\left(G_k\right)\right) \to \Shv\left(k_{\et}\right),$$ whose right adjoint is the restriction of $\tilde \gamma^*$ to sheaves, which lands in sheaves. Hence, $\Lan_y\left(y\circ \gamma\right)$ restricts to a colimit preserving functor $$\gamma_!:\Shv\left(S\left(G_k\right)\right) \to \Shv\left(k_{\et}\right)$$ which is left adjoint to $\gamma^*.$

We conclude that both $\gamma_!$ and $\gamma^*$ preserve colimits. We will show they form an adjoint equivalence by showing that the components of the unit and co-unit along representables are equivalences. Indeed,
\begin{eqnarray*}
\gamma^*\gamma_!\left(y\left(\Spec\left(k_i\right)\right)\right)\left(\Spec\left(k_j\right)\right) &\simeq &
\gamma_!\left(y\left(\Spec\left(k_i\right)\right)\right)\left(\gamma\left(\Spec\left(k_j\right)\right)\right)\\
&\simeq& \left(\gamma\left(y\left(\Spec\left(k_i\right)\right)\right)\right)\left(\gamma\left(\Spec\left(k_j\right)\right)\right)\\
&\simeq& \map_{k_{\et}}\left(\gamma\left(\Spec\left(k_j\right)\right),\gamma\left(\Spec\left(k_i\right)\right)\right)\\
&\simeq& \map_{S\left(k_{\et}\right)}\left(\Spec\left(k_j\right),\Spec\left(k_i\right)\right)\\
&\simeq& y\left(\Spec\left(k_i\right)\right)\left(\Spec\left(k_j\right)\right),
\end{eqnarray*}
and hence the component of the unit $\eta:id \Rightarrow \gamma^*\gamma_!$ along $\Spec\left(k_i\right)$ is an equivalence, and from \cite[Proposition 5.5.4.20 and Theorem 5.1.5.6]{htt}, we conclude that $\eta$ is an equivalence.

Now, let $X=\underset{i}\coprod \Spec\left(k_i\right)$ be in $k_{\et}$. Then
\begin{eqnarray*}
 \gamma_!\gamma^*y\left(X\right) &\simeq& \underset{i} \coprod \gamma_!\gamma^*y\left(\Spec\left(k_i\right)\right)\\
 &\simeq& \underset{i} \coprod y\left(\Spec\left(k_i\right)\right)\\
 &\simeq& y\left(X\right).
 \end{eqnarray*}
 So, we conclude that the component of $$\varepsilon:\gamma_!\gamma^* \to id$$ along $X$ is also an equivalence, and hence, again by \cite[Proposition 5.5.4.20 and Theorem 5.1.5.6]{htt}, we conclude that $\varepsilon$ is an equivalence. Hence, in particular, $\gamma^*$ is an equivalence of $\i$-categories.
\end{proof}

Since for sheaves of sets, we have $$\G_k\mbox{-}\Set\simeq \Shv^{\delta}\left(\Orb\left(G_k\right)\right),$$ it is tempting to think of the $\i$-topos $\Shv\left(\Orb\left(G_k\right)\right)$ as a good definition for the $\i$-category of spaces with a ``continuous'' action of $G_k.$ In fact, one can describe the objects of $\cB G_k \simeq \Shv\left(\Orb\left(G_k\right)\right)$ as compatible collections of spaces $X^{U_i}$ with actions of $G_k/U_i,$ where $U_i$ ranges over open normal subgroups of $G_k$. Since $G_k$ is compact with the profinite topology, each of these subgroups is of finite index, and hence each $X^{U_i}$ is a space with the action of a finite group. The compatibility condition
is that if $U_i \subseteq U_j$, then $X^{U_j}$ is the homotopy fixed points for the action of $U_j/U_i$ on $X^{U_i}.$ This can be read off from the description of $G_k$ as the cofiltered limit of the finite groups $G_k/U_i.$ However, we argue that this is not a good definition of a space with a continuous $G_k$-action, since the natural candidate for an underlying space functor, is often not conservative. Consider a sheaf $X$ of sets on the small \'etale site of $k.$ Then, to extract a set with a continuous $G_k$-action from $X,$ one takes the set $X\left(\Spec\left(k^{sep}\right)\right),$ and endows it with the canonical action. Topos-theoretically, under the equivalence $$\Set \simeq \Shv^{\delta}\left(k^{sep}_{\et}\right),$$ the set $X\left(\Spec\left(k^{sep}\right)\right)$ can be identified with the stalk of $X$ along the canonical geometric morphism $$\Set \to \Shv^{\delta}\left(k_{\et}\right)$$ induced by $$\Spec\left(k^{sep}\right) \to \Spec\left(k\right).$$ In fact, this is the only geometric morphism to $\Shv^{\delta}\left(k_{\et}\right),$ that is, it is the unique point of the topos. Hence, the natural candidate for the underlying space of an object $X$ of $\cB G_k$ is the stalk of $X$ at the unique point $$\Spc \to \cB G_k.$$ However, since the $\i$-topos $\cB G_k$ is often not hypercomplete, it is possible for the space arising as this stalk, which can be computed as $\underset{i} \colim X^{U_i},$ to be contractible when each of the spaces $X^{U_i}$ are not.

\begin{remark}
 We argue however that the $\i$-topos $\mathbb{H}\mathrm{ypShv}\left(\Orb\left(G_k\right)\right)$ of \emph{hypersheaves} is a good candidate $\i$-category of spaces with an action of $G_k,$ since the stalk functor $$\mathbb{H}\mathrm{ypShv}\left(\Orb\left(G_k\right)\right) \to \Spc$$ is conservative. To see this, one can model the $\i$-category $ \mathbb{H}\mathrm{ypShv}\left(\Orb\left(G_k\right)\right)$ by the local model structure on the category of simplicial presheaves, where weak equivalences are stalk-wise weak equivalences, since the small \'etale site of any scheme has enough points \cite[p. 12]{localize}, and then use the fact that the site also has a \emph{unique} point, since $k$ is a field.
\end{remark}

Given any sheaf $X$ of spaces on $\Orb\left(G_k\right),$ one can always take its hypersheafification, and by the previous remark, obtain a space with an action of $G_k.$ However, if $X$ is not a hypercomplete object, this might result in loss of information, in that $X$ cannot be recovered from its hypercompletion. Hence, one may think of objects of $\Shv\left(\Orb\left(G_k\right)\right)$ as generalized $G_k$-spaces which may contain more information. Given such a generalized $G_k$-space, it is natural to contemplate a good definition of the pro-space arising as its ``homotopy quotient'' by $G_k$. Lets first consider the case when $X$ is a $0$-truncated object, so it can be identified with an actual $G_k$-set. If the action is moreover transitive, with stabilizer $G_i,$ then the homotopy quotient should clearly be the pro-space $BG_i$. Note however, all of the transitive $G_k$-sets are of the form $G_k/U$ for $U$ an open subgroup, i.e. are objects in the site $\Orb\left(G_k\right).$ Since, moreover, taking homotopy quotients should be colimit preserving, this uniquely defines such a functor, provided it exists, that is, if there exists a colimit preserving functor

$$\left(\blank\right)//G_k:\Shv\left(\Orb\left(G_k\right)\right) \to \Pro\left(\Spc\right)$$ such that 
$$\left(G_k/U\right)//G_k \simeq BU,$$ then it is unique.
We claim that not only does such a colimit preserving functor exist, but is computed simply by taking the shape of slice topoi, that is, the functor

$$\cB G_k \stackrel{\chi_{\cB G_k}}{\longlonghookrightarrow} \Topi^{\et}/\cB G_k \to \Topi \stackrel{\Shape}{\longlongrightarrow} \Pro\left(\Spc\right),$$
which sends an object $X$ in the classifying $\i$-topos $\cB G_k$ to the shape of its slice topos, i.e., to $\Shape\left(\cB G_k/X\right).$ Lets denote this functor by $$\left(\blank\right)//G_k.$$

\begin{proposition}\label{prop:quotient}
The functor $\left(\blank\right)//G_k$ is colimit preserving and for all open subgroups $U,$ $$\left(G_k/U\right)//G_k \simeq BU.$$
\end{proposition}
\begin{proof}
By \cite[Proposition 6.3.5.14]{htt}, the composite $$\cB G_k \to \Topi$$ preserve colimits. Moreover, $\Shape$ is a left adjoint, hence $\left(\blank\right)//G_k$ is colimit preserving. Now, suppose that $U$ is an open normal subgroup of $G_k.$ Then by Lemma \ref{lem:gamma}, under the equivalence $$\cB G_k \simeq \Shv\left(k_{\et}\right),$$ $G_k/U$ corresponds to $\Spec\left(\left(k^{sep}\right)^U\right).$ Moreover, there is a canonical equivalence $$\Shv\left(k_{\et}\right)/\Spec\left(\left(k^{sep}\right)^U\right) \simeq \Shv\left(\left(\left(k^{sep}\right)^U\right)_{\et}\right).$$ Hence, the shape of this slice topos is the \'etale homotopy type of $\Spec\left(\left(k^{sep}\right)^U\right),$ which is simply $$B \Gal\left(k^{\sep}/\left(k^{sep}\right)^U\right) = B U.$$
\end{proof}

Unwinding definitions, if $$e_k:\Shv\left(k_{\et}\right) \to \Spc$$ is the unique geometric morphism, then $\Pro\left(e_k\right)_!$ is the unique cofiltered limit preserving functor $$\Pro\left(\Shv\left(k_{\et}\right)\right) \to \Pro\left(\Spc\right)$$ which restricts to $$\left(\blank\right)//G_k$$ along the inclusion $$j:\Shv\left(k_{\et}\right) \hookrightarrow \Pro\left(\Shv\left(k_{\et}\right)\right).$$ By abuse of notation, we may write $\Pro\left(e_k\right)_!=\left(\blank\right)//G_k.$
Hence, Remark \ref{rmk:recover} together with Proposition \ref{prop:quotient} implies that for any scheme $X$ (or any sheaf on the big \'etale site $\Sch^{\ft}$)

$$\Pi^{Gal_k,\et}_{\i}\left(X\right)//G_k \simeq \Pi^{\et}_{\i}\left(X\right),$$
i.e. its ordinary \'etale homotopy type can be recovered, at least morally, as the homotopy quotient of its Galois-equivariant \'etale homotopy type by the absolute Galois group of $k.$

\subsection{An explicit description of the relative \'{e}tale homotopy type} \label{subsec:explicit} The goal is of this section is to prove Theorem~\ref{thm:concrete_relative}, which gives an explicit description of the relative \'{e}tale homotopy type of any $S$-scheme as a left exact functor from the small \'etale site of $S$ to spaces (which works more generally for any sheaf $\cX \in \Shv_{\et}\left( \Sch^{\ft} \right)$). This description persists for the relative realization of motivic spaces, see Proposition~\ref{prop:explicit-a1}.

Informally, if $X$ is an $S$-scheme of finite type, the functor $\Pi^{S,\et}_{\i}\left(X \right)$ from $\Shv_{\et}\left( \Sch^{\ft} \right)$ to spaces is one ``corepresented by $X$" as one could guess from the description of the shape functor as a pro-left adjoint (see Remark~\ref{remark:shriek}). To obtain this explicit formula, we will describe two geometric morphisms
$$\rho:\Shv\left(S_{\et}\right) \to \Shv_{\et}\left(\Sch_S^{\ft}\right)$$ 
$$\lambda:\Shv_{\et}\left(\Sch_S^{\ft}\right) \to \Shv\left(S_{\et}\right),$$ between the small \'etale site and the big \'etale site of $S,$ such that $\rho \circ \lambda \simeq id$ (see Proposition~\ref{sec:small-vs-big}), i.e., the geometric morphism $\lambda$ is a section of the geometric morphism $\rho$.

\subsubsection{Deligne-Mumford schemes and their functor of points} We first review the theory of Deligne-Mumford stacks (see Definition~\ref{defn:dm-stk}) in general, from the point of view of its associated $\infty$-topos (called Deligne-Mumford schemes) versus its functor of points.  The former is the discussed by Lurie in~\cite[Chapter I.1]{luriespec}, in the spectral setting.

Let $X$ be a scheme, and consider its small \'etale site $X_{\et}$ of affine schemes \'etale over $X.$ There is a canonical sheaf of rings on this site $\mathcal{O}^{\et}_X$: to an \'etale map $\Spec\left(A\right) \to X,$ $\mathcal{O}^{\et}_X$ assigns the ring $A.$ This sheaf can be canonically identified with a sheaf of rings on the small \'etale $\i$-topos $\Shv\left(X_{\et}\right),$ which we will also denote by $\mathcal{O}^{\et}_X.$ The stalk of $\mathcal{O}^{\et}_X$ at any geometric point of $X$ is a strictly Henselian local ring. By \cite[Theorem 2.2.12]{dagv}, this construction restricted to affine schemes gives rise to a fully faithful functor
$$\Spec_{\et}:\Aff \hookrightarrow \Htop$$
from the category of affine schemes to the $\i$-category of $\i$-topoi locally ringed in strictly Henselian algebras (see \cite[Definition 1.2.2.5]{luriespec}). For an affine scheme $\Spec\,R$, we call $\Spec_{\et}(\Spec\,R) = (\Shv(R_{\et}), \mathcal{O}^{\et}_{\Spec\,R})$ the \'{e}tale spectrum of $R$.

\begin{definition}
A morphism $f:\left(\cE,\mathcal{O}_\cE\right) \to \left(\cF,\mathcal{O}_\cF\right)$ of $\i$-topoi locally ringed in Henselian algebras, is \textbf{\'etale} if the underlying geometric morphism of $\i$-topoi $\cE \to \cF$ is \'etale, and the morphism of sheaves $f^*\mathcal{O}_\cF \to \mathcal{O}_\cE$ is an isomorphism.
A collection of \'etale morphisms $\left(f_\alpha:\left(\cE_{\alpha},\mathcal{O}_\alpha\right) \to \left(\cE,\mathcal{O}_{\cE}\right)\right)$ is said to \textbf{cover} $\cE$ if under the equivalence of $\i$-categories

$$\chi_{\cE}:\cE \stackrel{\sim}{\longrightarrow} \Topi^{\et}/\cE$$ of Proposition \ref{lem:htt6.3.5.10}, the maps $f_\alpha$ correspond to objects $E_\alpha$ of $\cE$ such that the unique morphism $$\underset{\alpha} \coprod E_\alpha \to 1_{\cE}$$ is an epimorphism.
\end{definition}

Just as a scheme is locally the spectrum of a ring, a Deligne-Mumford scheme is locally the \'{e}tale spectrum of a ring. 

\begin{definition} \label{dm:scheme}
Let $\left(\cE,\mathcal{O}_\cE\right)$ be an $\i$-topos locally ringed in strictly Henselian algebras. It is a \textbf{Deligne-Mumford scheme} if there exists an \'etale covering family $\left(f_\alpha:\left(\cE_{\alpha},\mathcal{O}_\alpha\right) \to \left(\cE,\mathcal{O}_{\cE}\right)\right)$ such that for each $\alpha$ there is a commutative algebra $A_\alpha$ and an equivalence $$\left(\cE_{\alpha},\mathcal{O}_\alpha\right) \simeq \Spec_{\et}\left(A_\alpha\right).$$ Denote the full subcategory of $\Htop$ on the Deligne-Mumford schemes by $\mathfrak{DM}_\i$.
\end{definition}

The following theorem says that a Deligne-Mumford scheme can be recovered by its functor of points.

\begin{theorem} \label{thm:higherdave} \cite[Theorem 2.4.1]{dagv},\cite[Theorem 6.2.1]{higherdave}
Given a Deligne-Mumford scheme $\left(\cE,\mathcal{O}_\cE\right),$ and an affine scheme $X,$ the space $\map\left(\Spec_{\et}\left(X\right),\left(\cE,\mathcal{O}_\cE\right)\right)=:Y\left(\left(\cE,\mathcal{O}_\cE\right)\right)\left(X\right)$ is essentially small, and the corresponding functor $$Y\left(\left(\cE,\mathcal{O}_\cE\right)\right):\Aff^{op} \to \Spc$$ satisfies \'etale descent.

This construction assembles into a fully faithful functor
\begin{equation} \label{scheme-to-fop}
Y:\mathfrak{DM}_\i \hookrightarrow \Shv\left(\Aff,\et\right),
\end{equation} 
and the $1$-truncated objects (i.e. stacks of $1$-groupoids) in the essential image are precisely the classical Deligne-Mumford stacks (with no separation conditions). In particular, the essential image contains all schemes: if $X$ is a scheme, $X \simeq Y\left(\left(Shv\left(X_{\et}\right),\mathcal{O}^{\et}_X\right)\right).$
\end{theorem}

\begin{corollary}\label{cor:schff}
The small \'etale spectrum functor extends to a fully faithful functor $$\widetilde{\Spec_{\et}}:\Sch \hookrightarrow \Htop,$$ sending a scheme $X$ to $\left(Shv\left(X_{\et}\right),\mathcal{O}^{\et}_X\right).$
\end{corollary}

\begin{definition}
A \textbf{Deligne-Mumford $\i$-stack} is an object of $\Shv\left(\Aff,\et\right)$ in the essential image of $Y$ as in~\eqref{scheme-to-fop} above. We denote the corresponding $\i$-category by $\mathfrak{DMST}_\i.$ Let $S$ be a scheme. Denote by $\left(\mathfrak{DMST}_\i\right)_S$ the slice category $\mathfrak{DMST}_\i/S,$ and refer to the objects as \textbf{$S$-Deligne-Mumford $\i$-stacks}.
\end{definition}
Hence, by definition, we have an equivalence of $\infty$-categories between Deligne-Mumford schemes and Deligne-Mumford stacks.
\begin{definition} \label{defn:dm-stk}
A map $\cX \to \cY$ between Deligne-Mumford $\i$-stacks will be called \textbf{\'etale} if it is equivalent to $Y\left(\varphi\right),$ for $\varphi$ an \'etale map of strictly Henselian ringed $\i$-topoi. 
\end{definition}

\begin{definition}
Let $S$ be a scheme, and let $\left(\cE,\mathcal{O}_\cE\right) \to \widetilde{\Spec_{\et}}\left(S\right)$ be a map whose codomain is Deligne-Mumford scheme. Then $\left(\cE,\mathcal{O}_\cE\right)$ is \textbf{locally of finite type over $S$} if there exists an \'etale covering family $$\left(f_\alpha:\left(\cE_\alpha,\mathcal{O}_\alpha\right)_\alpha \to \left(\cE,\mathcal{O}_\cE\right)\right)_\alpha$$ such that for each $\alpha$ there is an equivalence $$\left(\cE_\alpha,\mathcal{O}_\alpha\right)\simeq \widetilde{\Spec_{\et}}\left(X_\alpha\right),$$ for $X_\alpha$ a scheme, such that the induced map $\widetilde{\Spec_{\et}}\left(X_\alpha\right) \to \widetilde{\Spec_{\et}}\left(S\right)$ corresponds under Corollary \ref{cor:schff} to a morphism $X_\alpha \to S$ of schemes which is locally of finite type. Denote by $\mathfrak{DM}^{lft/S}_\i$ the full subcategory of the slice category $\mathfrak{DM}_\i/\widetilde{\Spec_{\et}}\left(S\right)$ on those objects which are locally of finite type over $S.$
\end{definition}

Since $\widetilde{\Spec_{\et}}$ is fully faithful, there is an induced fully faithful functor
$$\varphi:\Sch^{\ft}/S \hookrightarrow \mathfrak{DM}^{lft/S}_\i.$$ The following proposition follows by a similar argument to that of \cite[Theorem 2.4.1]{dagv}:
\begin{proposition}\label{prop:phifff}
The composite $$\mathfrak{DM}^{lft/S}_\i \stackrel{y}{\longhookrightarrow} \PShv\left(\mathfrak{DM}^{lft/S}_\i\right) \stackrel{\varphi^*}{\longrightarrow} \PShv\left(\Sch^{\ft}_S\right)$$ is fully faithful and factors through the inclusion $$\Shv_{\et}\left(\Sch_S^{\ft}\right) \hookrightarrow \PShv\left(\Sch_S^{\ft}\right).$$
\end{proposition}

\begin{definition}
Denote by $\mathfrak{DMST}^{lft/S}_\i$ the essential image of $\mathfrak{DM}^{lft/S}_\i$ in $\Shv_{\et}\left(\Sch_S^{\ft}\right)$ and refer to its objects as $S$-Deligne-Mumford $\i$-stacks locally of finite type.
\end{definition}

\begin{proposition}\label{prop:et-ft}
Every \'etale map $\left(\cE,\mathcal{O}_\cE\right) \to \widetilde{\Spec_{\et}}\left(S\right)$ is locally of finite type.
\end{proposition}

\begin{proof}
As any \'etale map of schemes is locally of finite presentation and hence locally of finite type, it suffices to prove that there exists an \'etale covering family $$\left(f_\alpha:\left(\cE_\alpha,\mathcal{O}_\alpha\right)_\alpha \to \left(\cE,\mathcal{O}_\cE\right)\right)_\alpha$$ with equivalences $$\left(\cE_\alpha,\mathcal{O}_\alpha\right)\simeq \widetilde{\Spec_{\et}}\left(X_\alpha\right),$$ for $X_\alpha$ a scheme, such that the induced map $\widetilde{\Spec_{\et}}\left(X_\alpha\right) \to \widetilde{\Spec_{\et}}\left(S\right)$ corresponds to an \'etale morphism $X_\alpha \to S$ of schemes. By \cite[Proposition 2.3.5, (5)]{dagv}, there is an equivalence of $\i$-categories
\begin{eqnarray*}
\Shv\left(S_{\et}\right) &\to& \mathfrak{DM}^{\et}_\i/\widetilde{\Spec_{\et}}\left(S\right)\\
F &\mapsto& \left(\left(\Shv\left(S_{\et}\right)/F,\pi_F^*\left(\mathcal{O}^{\et}_S\right)\right) \to \left(\Shv\left(S_{\et}\right),\mathcal{O}^{\et}_S\right)\right)
\end{eqnarray*}
where $\mathfrak{DM}^{\et}_\i/\widetilde{\Spec_{\et}}\left(S\right)$ is the full subcategory of $\mathfrak{DM}_\i/\widetilde{\Spec_{\et}}\left(S\right)$ spanned by the \'etale maps, and where $\pi_F:\Shv\left(S_{\et}\right)/F \to \Shv\left(S_{\et}\right)$ is the canonical \'etale geometric morphism associated to $F.$ Therefore, it suffices to show that for any sheaf $F$ in the small \'etale $\i$-topos $\Shv\left(S_{\et}\right),$ there exists representable objects $X_\alpha$ and an epimorphism $$\underset{\alpha} \coprod X_\alpha \to F$$ in $\Shv\left(S_{\et}\right).$ (Here by a representable object, we mean a representable sheaf on the small \'etale site $S_{\et}.$) To this end, denote by $S_{\et}/F$ the full subcategory of $\Shv\left(S_{\et}\right)/F$ spanned by arrows with representable codomain. Then, by \cite[Proposition 5.1.5.3]{htt}, $F$ is the colimit of the composite
$$S_{\et}/F \hookrightarrow \Shv\left(S_{\et}\right)/F \to \Shv\left(S_{\et}\right).$$ The result now follows from \cite[Lemma 6.2.3.13]{htt}.
\end{proof}

\begin{definition}
Let $\mathfrak{DMST}^{\et/S}_\i$ denote the full subcategory of $\mathfrak{DMST}^{lft/S}_\i$ on those $S$-Deligne Mumford $\i$-stacks whose essentially unique map to $S$ is \'etale.
\end{definition}

\subsubsection{Geometric morphisms between small and big sites} 
Classically, there is an equivalence of $1$-categories between the small \'etale shaves (of sets) on $S$ and the category of \'etale algebraic spaces over $S$ --- see \footnote{The reference proves the theorem only for locally Noetherian schemes and locally separated algebraic spaces, but the proof generalizes to the claimed equivalence.}, for example \cite[Section V.I, Theorem 1.5]{milne-book}. The following is the analogue for $\infty$-topoi.

\begin{proposition}\label{prop:Lff}
There is a canonical fully faithful embedding
$$\rho_!:\Shv\left(S_{\et}\right) \hookrightarrow \Shv_{\et}\left(\Sch_S^{\ft}\right),$$ which restricts to an equivalence of $\i$-categories $$\Shv\left(S_{\et}\right)\stackrel{\sim}{\longrightarrow}\mathfrak{DMST}^{\et/S}_\i.$$
Moreover, if $F \in\Shv\left(S_{\et}\right)_0,$ then for $p:X \to S$ an $S$-scheme, 
$$\map\left(X,\rho_!\left(F\right)\right) \simeq \Gamma\left(p^*F\right).$$
\end{proposition}

\begin{proof}
Using Propositions~\ref{prop:et-ft} and~\ref{prop:phifff}, the $\i$-category $\mathfrak{DMST}^{\et/S}_\i$ can be canonically identified with the essential image of the following composite of fully faithful functors:
$$\Shv\left(S_{\et}\right) \stackrel{\sim}{\longrightarrow} \left(\mathfrak{DM}^{\et}_\i\right)/\widetilde{\Spec^{\et}\left(S\right)} \hookrightarrow \mathfrak{DM}^{lft/S}_\i \hookrightarrow \Shv_{\et}\left(\Sch_S^{\ft}\right).$$This establishes the first claim. Hence, for $p:X \to S$ an $S$-scheme locally of finite type, the space $\map\left(X,\rho_!\left(F\right)\right)$ can be identified with the space of lifts
$$\xymatrix{ & \left(\Shv\left(S_{et}\right)/F,\cO^{et}_S|_F\right) \ar[d]\\
\left(\Shv\left(X_{\et}\right),\cO^{\et}_X\right) \ar[r]^{p} \ar@{-->}[ru] & \left(\Shv\left(S_{\et}\right),\cO^{\et}_S\right).}$$
By \cite[Remark 2.3.4]{dagv}, there is a pullback diagram 
$$\xymatrix{\left(\Shv\left(X_{et}\right)/p^*F,\cO^{et}_X|_{p^*F}\right) \ar[d]_-{pr_1} \ar[r]^-{pr_2}& \left(\Shv\left(S_{et}\right)/F,\cO^{et}_S|_F\right) \ar[d]\\
\left(\Shv\left(X_{\et}\right),\cO^{\et}_X\right) \ar[r]^{p}& \left(\Shv\left(S_{\et}\right),\cO^{\et}_S\right),}$$ hence $\map\left(X,\rho_!\left(F\right)\right)$ can be identified with the space of sections of the map $pr_1$ above. By \cite[Proposition 2.3.5 (5)]{dagv}, this can in turn be identified with the space of maps $$\map_{\Shv\left(X_{\et}\right)}\left(1,p^*F\right)=:\Gamma\left(p^*F\right).$$
\end{proof}

\begin{lemma}\label{lemma:Lcolim}
The fully faithful embedding $$\rho_!:\Shv\left(S_{\et}\right) \hookrightarrow \Shv_{\et}\left(\Sch_S^{\ft}\right)$$ of Proposition \ref{prop:Lff} preserves small colimits.
\end{lemma}

\begin{proof}
Denote by $r:\Sch^{\ft}_S \hookrightarrow \Sch/S$ the canonical fully faithful inclusion. The restriction functor $r^*:\Shv_{\et}\left(\Sch/S\right) \to \Shv_{\et}\left(\Sch^{\ft}\right)$ has a left adjoint $r_!,$ given by left Kan extension; it is the unique colimit preserving functor sending each $S$-scheme of finite type to itself (regarded as an object in $\Shv_{\et}\left(\Sch/S\right)$). Since $r$ is a map of sites, $r^*$ also has a right adjoint $$r_*:\Shv_{\et}\left(\Sch^{\ft}\right) \to \Shv_{\et}\left(\Sch/S\right).$$ Explicitly, for an $S$-scheme $Z,$ $$r_*\left(\cX\right)\left(Z\right)\simeq \map_{\Shv_{\et}\left(\Sch^{\ft}\right)}\left(r^*y\left(Z\right),\cX\right),$$ where $y$ denotes the Yoneda embedding. Notice that the unit map $$\eta:id_{\Shv_{\et}\left(\Sch^{\ft}\right)} \Rightarrow r^*r_!$$ is an equivalence along representables, and as $id_{\Shv_{\et}\left(\Sch^{\ft}\right)}$ and $r^*r_!$ are colimit preserving, it follows from \cite[Theorem 5.1.5.6 and Proposition 5.5.4.20]{htt}, that $\eta$ is an equivalence, and hence we conclude that $r_!$ is fully faithful. Therefore, it suffices to prove that 
$$\Shv\left(S_{\et}\right) \stackrel{\sim}{\longrightarrow} \left(\mathfrak{DM}^{\et}_\i\right)/\widetilde{\Spec^{\et}\left(S\right)} \hookrightarrow \mathfrak{DM}^{lft/S}_\i \hookrightarrow \Shv_{\et}\left(\Sch_S^{\ft}\right) \stackrel{r_!}{\longlongrightarrow} \Shv_{\et}\left(\Sch/S\right)$$ preserves colimits.

Notice that the following diagram commutes:
$$\xymatrix{\left(\mathfrak{DM}^{\et}_\i\right)/\widetilde{\Spec^{\et}\left(S\right)} \ar[d] \ar[r] & \mathfrak{DM}^{lft/S}_\i \ar[r] & \Shv_{\et}\left(\Sch_S^{\ft}\right)  \ar[d]^-{r_!}\\
\mathfrak{DM}^{\et}_\i \ar[d]^-{j} \ar[r] & \Shv_{\et}\left(\Sch\right) & \Shv_{\et}\left(\Sch/S\right), \ar[l]\\
\mathfrak{DM}_\i \ar[ru]^-{\tilde y} & &}$$
where the functor $\Shv_{\et}\left(\Sch/S\right) \to \Shv_{\et}\left(\Sch\right),$ under the canonical equivalence $$\Shv_{\et}\left(\Sch/S\right) \simeq \Shv_{\et}\left(\Sch\right)/S,$$ (see \cite[Remark 2.4]{higherdave}) corresponds to the forgetful functor. By \cite[Proposition 1.2.13.8]{htt}, it suffices to prove that the composite $$\left(\mathfrak{DM}^{\et}_\i\right)/\widetilde{\Spec^{\et}\left(S\right)} \to \mathfrak{DM}^{\et}_\i \to \Shv_{\et}\left(\Sch\right)$$ preserves colimits. By the same proposition, we reduce to observing the the latter functor preserves colimits, which is the content of \cite[Lemma 2.4.13]{dagv}.
\end{proof}

\begin{proposition}
Denote by $\rho:S_{\et} \hookrightarrow \Sch_S^{\ft}$ the canonical fully faithful embedding. Then (by abuse of notation) $\rho$ induces a geometric morphism $$\rho:\Shv\left(S_{\et}\right) \to \Shv_{\et}\left(\Sch_S^{\ft}\right),$$ and one has that $\rho_!$ is left adjoint to $\rho^*.$
\end{proposition}

\begin{proof}
The inverse image functor of the geometric morphism $\rho$ is the restriction functor $$\rho^*:\Shv_{\et}\left(\Sch_S^{\ft}\right) \to \Shv\left(S_{\et}\right).$$ The direct image functor $\rho_*$ is defined by the equation $$\rho_*\left(F\right)\left(X\right)=\map\left(\rho^*y\left(X\right),F\right),$$ for all $S$-schemes $X.$ This is automatically right adjoint to $\rho^*$ as soon as it is well defined, i.e. as soon as one shows that for all sheaves $F,$ the presheaf $\rho_*F$ is a sheaf. We claim that this is the case. It suffices to show that if $\left(U_\alpha \to X\right)_\alpha$ is an \'etale covering family, then $$\underset{\alpha}\coprod \rho^*y\left(U_\alpha\right) \to \rho^*y\left(X\right)$$ is an epimorphism. In other words, one needs to show that for all \'etale $S$-schemes $T,$ and for all $g:T \to \rho^*y\left(X\right),$ there exists an \'etale covering family $\left(T_\beta \to T\right)_\beta$ and maps $$g_\beta:T_\beta \to \underset{\alpha}\coprod \rho^*y\left(U_\alpha\right)$$ such that the following diagram commutes:
$$\xymatrix{T_\beta \ar[d] \ar[r]^-{g_\beta} & \underset{\alpha}\coprod \rho^*y\left(U_\alpha\right) \ar[d]\\
T \ar[r]^-{g}& \rho^*y\left(X\right).}$$
Notice that the map $g,$ by the Yoneda lemma, corresponds uniquely to a map of $S$-schemes $\tilde g:T \to X.$ One can then pullback the \'etale cover of $X$ via $\tilde g$ to get the desired cover of $T,$ and the maps $g_\beta$ are determined by the pullback diagram itself. Hence, $\rho^* \vdash \rho_*$ is a geometric morphism.

We will now show that $\rho_!$ is left adjoint to $\rho^*.$ Let $T$ be an $S$-scheme and $F$ an \'etale sheaf on $\Sch_S^{\ft}$. Then
\begin{eqnarray*}
\map\left(\rho_!y\left(T\right),F\right) &\simeq& \map\left(y\left(T\right),F\right)\\
&\simeq& F\left(T\right)\\
&\simeq& \rho^*F\left(T\right)\\
&\simeq& \map\left(y\left(T\right),\rho^*F\right).
\end{eqnarray*}
Since every sheaf is a colimit of representables, we are done by Lemma~\ref{lemma:Lcolim}. 
\end{proof}


\begin{definition}
A geometric morphism $f:\cE \to \cF$ between $\i$-topoi is \textbf{$\i$-connected} if the inverse image functor $f^*$ is fully faithful.
\end{definition}

\begin{proposition}\label{prop:infconn}
The functor $$\rho_!:\Shv\left(S_{\et}\right) \hookrightarrow \Shv_{\et}\left(\Sch_S^{\ft}\right),$$ is the inverse image functor of a canonical $\i$-connected geometric morphism $$\lambda:\Shv_{\et}\left(\Sch_S^{\ft}\right) \to \Shv\left(S_{\et}\right).$$ I.e., $$\rho_!=\lambda^*.$$
\end{proposition}

\begin{proof}
By Lemma \ref{lemma:Lcolim}, $\rho_!$ preserves colimits and by Proposition \ref{prop:Lff} it is fully faithful. Hence, it suffices to prove that $\rho_!$ preserves finite limits. For this, it is enough to check that for each $S$-scheme $p:X \to S,$ the functor $$\map_{\Shv_{\et}\left(\Sch_S^{\ft}\right)}\left(X,\rho_!\left(\mspace{3mu}\cdot\mspace{3mu}\right)\right):\Shv\left(S_{\et}\right) \to \Spc$$ preserves finite limits. However, by Proposition \ref{prop:Lff}, this functor can be identified with $\Gamma \circ p^*,$ where $\Gamma$ is the global sections functor of $\Shv\left(X_{\et}\right),$ and since $\Gamma$ and $p^*$ are each left exact, this finishes the proof.
\end{proof}

\begin{proposition} \label{sec:small-vs-big}
The geometric morphism $$\lambda:\Shv_{\et}\left(\Sch_S^{\ft}\right) \to \Shv\left(S_{\et}\right)$$ is a section of $$\rho:\Shv\left(S_{\et}\right) \to \Shv_{\et}\left(\Sch_S^{\ft}\right).$$
\end{proposition}

\begin{proof}We have that
\begin{eqnarray*}
\left(\rho \circ \lambda\right)^*&\simeq& \lambda^* \circ \rho^*\\
&\simeq& \rho_! \circ \rho^*
\end{eqnarray*}
and since $\rho_!$ is fully faithful, and is left adjoint to $\rho^*$, it follows that the unit $$\rho_! \circ \rho^* \to id$$ is an equivalence. Hence $$\rho \circ \lambda \simeq id.$$
\end{proof}

\begin{remark}
Since $\lambda$ is uniquely determined by $\lambda^*=\rho_!,$ and $\rho_!$ is uniquely determined by its right adjoint $\rho^*,$ which in turn uniquely determines $\rho,$ we can conclude that $\rho$ comes equipped with a canonical distinguished section, $\lambda$. Conversely, $\lambda$ uniquely determines $\rho.$
\end{remark}

After our discussion, we obtain the following explicit description of the relative \'{e}tale homotopy type.

\begin{theorem}\label{thm:concrete_relative}
Let $\cX$ be a stack in $\Shv_{\et}\left(\Sch^{\ft}_S\right).$ Viewed as a left exact functor $$\Shv\left(S_{\et}\right) \to \Spc,$$ the $S$-relative \'etale homotopy type $\Pi^{S,\et}_{\i}\left(\cX\right)$ can be identified with the functor $$\map_{\Shv_{\et}\left(\Sch^{\ft}_S\right)}\left(\cX,\rho_!\left(\mspace{3mu}\cdot\mspace{3mu}\right)\right).$$
\end{theorem}

\begin{proof}
On one hand, by definition, the functor $$\Pi^{S,\et}_{\i}\left(\mspace{3mu}\cdot\mspace{3mu}\right):\Shv_{\et}\left(\Sch^{\ft}_S\right) \to \Pro\left(\Shv\left(S_{\et}\right)\right)$$ is the composite 
$$\Shv_{\et}\left(\Sch^{\ft}_S\right) \stackrel{\Shv\left(\left(\mspace{3mu}\cdot\mspace{3mu}\right)_{\et}\right)/S}{\longlonglongrightarrow} \Topi/\Shv\left(S_{\et}\right) \stackrel{\Shape_S}{\longlongrightarrow} \Pro\left(\Shv\left(S_{\et}\right)\right)$$ of two colimit preserving functors, hence it is colimit preserving. On the other hand, we claim that the construction $$\cX \mapsto \map_{\Shv_{\et}\left(\Sch_S^{\ft}\right)}\left(\cX,\rho_!\left(\mspace{3mu}\cdot\mspace{3mu}\right)\right)$$ can be described as the composite $\Theta$ of functors
$$\resizebox{6.5in}{!}{$\Shv_{\et}\left(\Sch^{\ft}_S\right) \stackrel{\sim}{\longrightarrow} \Topi^{\et}/\Shv_{\et}\left(\Sch^{\ft}_S\right) \rightarrow \Topi/\Shv_{\et}\left(\Sch^{\ft}_S\right) \stackrel{\tilde \lambda}{\longrightarrow} \Topi/\Shv\left(S_{\et}\right) \stackrel{\Shape_S}{\longlongrightarrow} \Pro\left(\Shv\left(S_{\et}\right)\right),$}$$
where $\tilde \lambda$ is induced by composition with the geometric morphism $\lambda$ of Proposition \ref{prop:infconn}. Let us justify this claim.

Suppose that $\cX$ in $\Shv_{\et}\left(\Sch^{\ft}_S\right)$. Then $\Theta$ computes the $S$-shape of the composite geometric morphism $$\Shv_{\et}\left(\Sch^{\ft}_S\right)/\cX \stackrel{\pi_\cX}{\longrightarrow} \Shv_{\et}\left(\Sch^{\ft}_S\right) \stackrel{\lambda}{\longrightarrow} \Shv\left(S_{\et}\right).$$ Let $\cF$ be an \'etale sheaf on $S_{\et},$ then
\begin{eqnarray*}
\Shape_S\left(\lambda \circ \pi_{\cX}\right)\left(\cF\right) &\simeq& \Gamma_{\cX}\left(\pi_{\cX}^*\lambda^*\cF\right)\\
&\simeq& \Gamma_{\cX}\left(\pi_{\cX}^*\rho_!\cF\right)\\
&\simeq& \map_{\Shv_{\et}\left(\Sch^{\ft}_S\right)/\cX}\left(1,\pi_{\cX}^*\rho_!\cF\right)\\
&\simeq& \map_{\Shv_{\et}\left(\Sch^{\ft}_S\right)}\left(\left(\pi_{\cX}\right)_!\left(1\right),\rho_!\cF\right)\\
&\simeq& \map_{\Shv_{\et}\left(\Sch^{\ft}_S\right)}\left(\cX,\rho_!\cF\right),\\
\end{eqnarray*}
where we have used Proposition~\ref{prop:infconn} in the second equivalence. 

Notice that the functor $\tilde \lambda$ has a right adjoint, which sends an $\i$-topos $\cE \to \Shv\left(S_{\et}\right)$ over $\Shv\left(S_{\et}\right)$ to the pullback $\i$-topos $\Shv_{\et}\left(\Sch^{\ft}_S\right) \times_{\Shv\left(S_{\et}\right)} \cE \to \Shv_{\et}\left(\Sch^{\ft}_S\right).$ Observe further that, in light of \cite[Theorem 6.3.5.13]{htt}, the composite $$\Topi^{\et}/\Shv_{\et}\left(\Sch^{\ft}_S\right) \to \Topi/\Shv_{\et}\left(\Sch^{\ft}_S\right) \to \Topi$$ preserves colimits, since it factors as $$\Topi^{\et}/\Shv_{\et}\left(\Sch^{\ft}_S\right) \to \Topi^{\et} \to \Topi.$$ Hence, by \cite[Proposition 1.2.13.8]{htt}, we have that $$\Topi^{\et}/\Shv_{\et}\left(\Sch^{\ft}_S\right) \to \Topi/\Shv_{\et}\left(\Sch^{\ft}_S\right)$$ preserves colimits. It follows that $\Theta$ preserves colimits as well. 

Now, $\Theta$ and $\Pi^{S,\et}_{\i}$ are both colimit preserving functors $\Shv_{\et}\left(\Sch^{\ft}\right) \to \Pro\left(\Shv\left(S_{\et}\right)\right),$ so to conclude that they are equivalent functors, it suffices to show that they agree along representables, i.e. along $S$-schemes $p:X \to S$ of finite type. But this follows immediately from Proposition \ref{prop:Lff}, as $\Gamma \circ p^*$ is precisely $\Shape_S\left(\Shv\left(X_{\et}\right)\right).$
\end{proof}

The following fact will be useful:

\begin{proposition} \label{prop:right-adjoint} The functor $\Pi^{S,\et}_{\i}:\Shv_{\et}\left( \Sch^{\ft}_S \right) \to \Pro\left(\Shv\left(S_{\et}\right)\right)$ admits a right adjoint, which identifies with the composite 
\[
R:\Pro\left(\Shv\left(S_{\et}\right) \right) \stackrel{\Pro\left(\rho_!\right)}{\longlongrightarrow} \Pro\left(\Shv_{\et}\left( \Sch^{\ft}_S \right) \right) \stackrel{\lim}{\longlongrightarrow} \Shv_{\et}\left( \Sch^{\ft}_S \right)
\].
\end{proposition}

\begin{proof} 
Let $\underset{i}\lim j\left(F_i\right)$ be an arbitrary pro-object of $\Pro\left(\Shv\left(S_{\et}\right)\right)$ and $\cX$ an \'etale sheaf. Then
\begin{eqnarray*}
 \map_{\Pro\left(\Shv_{\et}\left( \Sch^{\ft}_S \right) \right)}\left(\Pi^{S,\et}_{\i}\left(\cX\right),\underset{i}\lim j\left(F_i\right)\right) &\simeq& \underset{i}\lim \map_{\Pro\left(\Shv_{\et}\left( \Sch^{\ft}_S \right) \right)}\left(\Pi^{S,\et}_{\i}\left(\cX\right), j\left(F_i\right)\right)\\
 &\simeq& \underset{i} \lim \Pi^{S,\et}_{\i}\left(\cX\right)\left(F_i\right)\\
 &\simeq& \underset{i} \lim \map_{\Shv_{\et}\left( \Sch^{\ft}_S \right)}\left(\cX,\rho_!\left(F_i\right)\right)\\
  &\simeq&  \map_{\Shv_{\et}\left( \Sch^{\ft}_S \right)}\left(\cX,\underset{i} \lim\rho_!\left(F_i\right)\right)\\
  &\simeq&  \map_{\Shv_{\et}\left( \Sch^{\ft}_S \right)}\left(\cX, \lim\left(\Pro\left(\rho_!\right)\left(\underset{i} \lim j\left(F_i\right)\right)\right)\right),
\end{eqnarray*}
where the second equivalence is the content of Theorem \ref{thm:concrete_relative}.
\end{proof}

\subsection{Relative \'{e}tale realization of motivic spaces} \label{subsec:rel-real}

We now come to the main construction of this paper. For $f:X \to Y$ a map of schemes, denote the direct and inverse image functors corresponding to the induced geometric morphism $$\Shv\left(f_{\et}\right):\Shv\left(X_{\et}\right) \to \Shv\left(Y_{\et}\right)$$ by $f_*$ and $f^*$  respectively.

The following is a relative version of Definition~\ref{defn:sa1}.

\begin{definition} \label{defn:a1s} An object $\cF \in \Shv\left(S_{\et}\right)$ is \textbf{$\bA^1$-invariant} if for all smooth $S$-schemes $p:X \to S$,
$$\eta_{p^*\cF}:p^*\cF \to \left(\pi_X\right)_*\left(\pi_X\right)^*p^*\cF$$ is an equivalence, where $\pi_X:X \times \bA^1 \to X$ is the canonical projection, and $\eta$ is the unit of the adjunction $$\left(\pi_X\right)^* \dashv \left(\pi_X\right)_*.$$ Denote by Let $S_{\bA^1}$ the full subcategory of $\Shv\left(S_{\et}\right)$ spanned by the $\bA^1$-invariant objects.
\end{definition}

We have the relative analog of Proposition~\ref{prop:sa1-ft}.


\begin{proposition}\label{prop:LAinv}
A sheaf $\cF$ in $\Shv\left(S_{\et}\right)$ is $\mathbb{A}^1$-invariant if and only if the corresponding Deligne-Mumford stack $\rho_!\left(\cF\right)$ is $\mathbb{A}^1$-invariant in $\Shv_{\et}\left(\Sch^{\ft}\right),$ in the sense of Definition \ref{defn:sa1}.
\end{proposition}

\begin{proof}
Notice that
\begin{eqnarray*}
\map_{\Shv\left(X_{\et}\right)}\left(1,\left(\pi_X\right)_*\left(\pi_X\right)^*p^*\cF\right) &\simeq& \map_{\Shv\left(\left(X\times\mathbb{A}^1\right)_{\et}\right)}\left(\left(\pi_X\right)^*\left(1\right),\left(\pi_X\right)^*p^*\cF\right)\\
&\simeq& \map_{\Shv\left(\left(X\times\mathbb{A}^1\right)_{\et}\right)}\left(1,\left(\pi_X\right)^*p^*\cF\right).
\end{eqnarray*}
So, if $\cF$ is $\mathbb{A}^1$-invariant, it follows that the canonical map $$\Gamma_{X \times \mathbb{A}^1} \left(p \circ \pi_X\right)^*\cF \to \Gamma_{X} p^*\cF$$ is an equivalence. Hence, by Proposition \ref{prop:Lff}, it follows that $\rho_!\left(\cF\right)$ is $\mathbb{A}^1$-invariant.

Conversely, suppose that $\rho_!\left(\cF\right)$ is $\mathbb{A}^1$-invariant. Let $u:U \to X$ be an \'etale map, i.e. a representable object in $\Shv\left(X_{\et}\right).$ We want to show that the canonical map $$p^*\cF\left(u\right) \to \left(\pi_X\right)_*\left(\pi_X\right)^*p^*\cF\left(u\right)$$ is an equivalence. Notice that $p\circ u:U \to S$ is an $S$-scheme locally finite type, and hence also an object of $\Shv_{\et}\left(\Sch^{\ft}\right).$ Moreover, the canonical geometric morphism $$\Shv\left(u_{\et}\right):\Shv\left(U_{\et}\right) \to \Shv\left(X_{\et}\right)$$ is \'etale, so the inverse image functor $u^*$ has a left adjoint, which is induced by composition with $u.$ In particular, $u_!\left(1\right)\simeq u \in \Shv\left(X_{\et}\right).$ Hence
\begin{eqnarray*}
p^*\cF\left(u\right) &\simeq& \map_{\Shv\left(X_{\et}\right)}\left(u_!\left(1\right),p^*\cF\right)\\
&\simeq&\map_{\Shv\left(U_{\et}\right)}\left(1,u^*p^*\cF\right)\\
&\simeq& \Gamma_U\left(p\circ u\right)^*\cF\\
&\simeq& \map_{\Shv_{\et}\left(\Sch^{\ft}\right)}\left(U,\rho_!\left(\cF\right)\right).
\end{eqnarray*}
Consider the following pullback diagram
$$\xymatrix@C=2.5cm{U \times \mathbb{A}^1 \ar[d]_-{\pi_U} \ar[r]^-{u\times id_{\mathbb{A}^1}} & X \times \mathbb{A}^1 \ar[d]^-{\pi_X}\\
U \ar[r]^-{u} & X.}$$
Similarly, to above
\begin{eqnarray*}
\left(\pi_X\right)_*\left(\pi_X\right)^*p^*\cF\left(u\right) &\simeq& \map_{\Shv\left(X_{\et}\right)}\left(u,\left(\pi_X\right)_*\left(\pi_X\right)^*p^*\cF\right)\\
&\simeq& \map_{\Shv\left(\left(X \times \mathbb{A}^1\right)_{\et}\right)}\left(\left(\pi_X\right)^*\left(u\right),\left(\pi_X\right)^*p^*\cF\right)\\
&\simeq& \map_{\Shv\left(\left(X \times \mathbb{A}^1\right)_{\et}\right)}\left(u\times id_{\mathbb{A}^1},\left(\pi_X\right)^*p^*\cF\right)\\
&\simeq& \map_{\Shv\left(\left(X \times \mathbb{A}^1\right)_{\et}\right)}\left(\left(u\times id_{\mathbb{A}^1}\right)_!\left(1\right),\left(\pi_X\right)^*p^*\cF\right)\\
&\simeq& \map_{\Shv\left(\left(X \times \mathbb{A}^1\right)_{\et}\right)}\left(1,\left(u\times id_{\mathbb{A}^1}\right)^*\left(\pi_X\right)^*p^*\cF\right)\\
&\simeq& \map_{\Shv\left(\left(X \times \mathbb{A}^1\right)_{\et}\right)}\left(1,\left(\pi_U\right)^*u^*p^*\cF\right)\\
&\simeq& \Gamma_{U \times \mathbb{A}^1}\left(\left(p\circ u\circ \pi_U\right)^*\cF\right)\\
&\simeq& \map_{\Shv_{\et}\left(\Sch^{\ft}\right)}\left(U \times \mathbb{A}^1,\rho_!\left(\cF\right)\right).
\end{eqnarray*}
The result now follows.
\end{proof}

\begin{proposition} \label{prop:sa1isbetter} The $\infty$-category $S_{\bA^1}$ has finite limits and is accessible. Hence the fully faithful embedding $\isa: \Pro\left(S_{\bA^1}\right) \hookrightarrow \Pro\left(\Shv\left(S_{\et}\right)\right)$ admits a left adjoint: $$\isa^*: \Pro\left(\Shv\left(S_{\et}\right)\right) \rightarrow \Pro\left(S_{\bA^1}\right),$$
given by sending a pro-object, regarded as an left-exact accessible functor $$Z:\Shv\left(S_{\et}\right) \to \Spc,$$ to the composite $$S_{\bA^1} \stackrel{\isa}{\longhookrightarrow} \Shv\left(S_{\et}\right) \stackrel{Z}{\longrightarrow} \Spc.$$
\end{proposition}

\begin{proof} The proof is completely analogous to Proposition \ref{prop:sa1good}.
\end{proof}

\begin{definition}
We will denote the localization functor $\isa^*$ by $\widehat{\left(\mspace{3mu}\cdot\mspace{3mu}\right)}_{S_{\bA^1}}.$
\end{definition}

\subsubsection{Construction of the functor} \label{subsec:rel}Now we are ready to define our relative \'{e}tale realization functor. 

\begin{theorem}\label{thm:main-functor} Let $S$ be a scheme, then there exists a colimit preserving functor $$\Et^S_{\bA^1}: \Spc\left( S \right) \rightarrow \Pro\left(S_{\bA^1}\right)$$ whose value on $\LL_{\mot}\left(X\right)$ where $X$ is a smooth $S$-scheme is given by the $S_{\bA^1}$-localization of the $S$-relative \'etale homotopy type of $X$: $$\widehat{\Pi^{S,\et}_\i\left(X\right)}_{S_{\bA^1}}.$$
\end{theorem}

\begin{proof}
Consider the functor $$\LL_{\et}:\Shv_{\Nis}\left(\Sm_S\right) \to \Shv_{\et}\left(\Sch^{\ft}_S\right)$$ from the proof of Theorem \ref{thm:absolute_realization}, and denote by $\widetilde{\Et^S_{\bA^1}}$ the composite of colimit preserving functors
$$\xymatrix@C=2cm{\Shv_{\Nis}\left(\Sm_S\right) \ar[r]^-{\LL_{\et}} & \Shv_{\et}\left(\Sch^{\ft}_S\right) \ar[r]^-{\Pi^{S,\et}_\i} & \Pro\left(\Shv\left(S_{\et}\right)\right) \ar[r]^-{\widehat{\left(\mspace{3mu}\cdot\mspace{3mu}\right)}_{S_{\bA^1}}} & \Pro\left(S_{\bA^1}\right).}$$ By an analogous argument as in the proof of Theorem \ref{thm:absolute_realization}, in order to show that there is a colimit preserving lift
$$\xymatrix{\Spc\left(S\right) \ar@{-->}[r]^-{\Et^S_{\bA^1}} & \Pro\left(S_{\bA^1}\right),\\
\Shv_{\Nis}\left(\Sm_S\right) \ar[u]^-{\LL_{\mathbb{A}^1}} \ar[ru]_-{\widetilde{\Et^S_{\bA^1}}} &}$$
it suffices to show that $\widetilde{\Et^S_{\bA^1}}$ sends $\mathbb{A}^1$-local equivalences to equivalences. By equation (\ref{eq:oknis}) of the proof of Theorem \ref{thm:absolute_realization}, in light of Theorem \ref{thm:concrete_relative}, this is if and only if for all $\cX$ in $\Shv_{\et}\left(\Sch^{\ft}\right),$ and for all $\cZ$ in $S_{\bA^1},$ the canonical map
$$\map_{\Shv_{\et}\left(\Sch^{\ft}\right)}\left(\LL_{\et}\left(\cX\right),\rho_!\left(\cZ\right)\right) \to \map_{\Shv_{\et}\left(\Sch^{\ft}\right)}\left(\LL_{\et}\left(\cX\right) \times \mathbb{A}^1,\rho_!\left(\cZ\right)\right)$$ is an equivalence. The result now follows from Proposition \ref{prop:LAinv}.
\end{proof}

\begin{remark}\label{rmk:Li2}
Unwinding the definitions, we see that for $\cX$ a motivic space, $$\Et^S_{\bA^1}\left(\cX\right)=\reallywidehat{\Pi_{\i}^{S,\et}\left(\LL^{\et}i\left(\cX\right)\right)}_{S_{\bA^1}},$$ where $$i:\Spc\left( S \right) \hookrightarrow \Shv_{\Nis}\left(\Sm_S\right)$$ is the canonical inclusion.
\end{remark}

\subsubsection{A concrete description of the relative realization functor} 

Denote by $\Shv_{\et}\left(\Sch^{\ft}_S\right)_{\bA^1},$ the full subcategory of $\Shv_{\et}\left(\Sch^{\ft}_S\right)$ on the $\mathbb{A}^1$-invariant objects, i.e. it is the $\i$-category of $\mathbb{A}^1$-invariant presheaves of spaces on $\Sch^{\ft}_S$ which satisfy \'etale descent. This $\infty$-category fits in the following commutative diagram
\[
\xymatrix{
\PShv(\Sch_S^{\ft}) \ar[r] & \PShv(\Sm_S) & \\
\Shv_{\et}(\Sch^{\ft}_S)_{\bA^1} \ar[r] \ar@{^{(}->}[u] & \Spc_{\et,S}^{\bA^1} \ar@{^{(}->}[u] \ar@{^{(}->}[r] &  \Spc\left( S \right) \ar@{^{(}->}[ul].
}
\]

We define the horizontal composite as $$R_{\mathbb{A}^1}:\Shv_{\et}\left(\Sch^{\ft}_S\right)_{\bA^1} \to \Spc\left( S \right).$$ Notice furthermore that by Proposition \ref{prop:LAinv}, the functor $$\rho_!:\Shv\left(S_{\et}\right) \hookrightarrow \Shv_{\et}\left(\Sch_S^{\ft}\right)$$ of Proposition \ref{prop:Lff}
restricts to a functor 
$$\rho_{!,\bA^1}:\Shv\left(S_{\et}\right)_{\bA^1} \hookrightarrow \Shv_{\et}\left(\Sch_S^{\ft}\right)_{\bA^1}.$$

\begin{proposition}\label{prop:smm}
The composite $R_{\mathbb{A}^1} \circ \rho_{!,\bA^1}:\Shv\left(S_{\et}\right)_{\bA^1} \to  \Spc\left( S \right)$ is fully faithful, hence one may view an $\bA^1$-invariant \'etale sheaf $\cZ$ on $S$ as a motivic space in a natural way.
\end{proposition}

\begin{proof}
We will freely use notation from the proof of Theorem \ref{thm:absolute_realization}. It suffices to prove that the composite
$$\Shv\left(S_{\et}\right) \stackrel{\rho_!}{\longhookrightarrow} \Shv_{\et}\left(\Sch^{\ft}\right) \stackrel{R_{\et}}{\longrightarrow} \Shv_{\Nis}\left(\Sm_S\right)$$ is fully faithful. Since every \'etale sheaf is a Nisnevich sheaf, it furthermore suffices to prove that the composite
$$\Shv\left(S_{\et}\right) \stackrel{\rho_!}{\longhookrightarrow} \Shv_{\et}\left(\Sch^{\ft}\right) \stackrel{q^*}{\longrightarrow} \Shv_{\et}\left(\Sm_S\right)$$ is fully faithful. By Proposition \ref{prop:Lff}, $\rho_!$ restricts to an equivalence $\Shv\left(S_{\et}\right)\simeq \mathfrak{DMST}^{\et/S}_\i,$ and moreover, $\mathfrak{DMST}^{\et/S}_\i$ can be canonically identified with the full subcategory of the $\i$-category of Deligne-Mumford schemes on those which are \'etale over $\widetilde{\Spec}_{\et}\left(S\right),$ via Proposition \ref{prop:phifff}. Consider the analogously defined $\i$-category $\mathfrak{DMST}^{\Sm/S}_\i$ of Deligne-Mumford schemes smooth over $S.$ Notice that there is a canonical inclusion $$\mathfrak{DMST}^{\et/S}_\i \hookrightarrow \mathfrak{DMST}^{\Sm/S}_\i.$$ Hence it suffices to prove that the functor of points $$\mathfrak{DMST}^{\Sm/S}_\i \to \Shv_{\et}\left(\Sm_S\right)$$ is fully faithful. Denote by $\mathcal{C}_S$ the subcategory of Deligne-Mumford schemes on those of the form $\widetilde{\Spec}_{\et}\left(X\right)$ for $X$ a smooth $S$-scheme of finite type. Then $\mathcal{C}_S$ is a strong \'etale blossom in the sense of \cite[Definition 5.1.7]{dave-etale} (see \cite[Example 5.1.6]{dave-etale}), and hence it follows from \cite[Theorem 5.2.2]{dave-etale} that the functor of points $$\mathfrak{DMST}^{\Sm/S}_\i \to \Shv_{\et}\left(\Sm_S\right)$$ is fully faithful. 
\end{proof}

\begin{definition}
Denote the fully faithful functor $R_{\mathbb{A}^1} \circ \rho_{!,\bA^1}$ of Proposition \ref{prop:smm} by $$\psi:\Shv\left(S_{\et}\right)_{\bA^1} \hookrightarrow \Spc\left( S \right).$$
\end{definition}

\begin{proposition} \label{prop:explicit-a1}
Let $\cX$ be a motivic space. Then, viewed as a left exact functor $$S_{\bA^1} \to \Spc,$$ the pro-object $\Et^S_{\bA^1}\left(\cX\right)$ can be identified with the functor $\map_{\Spc\left( S \right)}\left(\cX,\psi\left(\mspace{3mu}\cdot\mspace{3mu}\right)\right).$ 
\end{proposition}

\begin{proof}
Consider the composite of adjunctions $$\xymatrix@1{\Shv_{\et}\left(\Sch^{\ft}_S\right)_{\bA^1}\mspace{4mu} \ar@{^{(}->}[r]<-0.9ex>_-{\operatorname{R}_{\mot}^{\et}} & \ar@<-0.5ex>[l]_-{\LL^{\et}_{\mot}} \Shv_{\et}\left(\Sch^{\ft}_S\right)\mspace{4mu} \ar[r]<-0.9ex>_-{\operatorname{R}^{\et}}  & \Shv_{\Nis}\left(\Sm_S\right). \ar@<-0.5ex>[l]_-{\LL_{\et}}},$$
where the left adjoints are displayed as the top arrows.

Since $\LL^{\et}$ is left exact, the composite $\LL^{\et}_{\mot} \circ \LL^{\et}$ sends $\mathbb{A}^1$-local equivalences to equivalences. Hence, there is an induced colimit preserving functor $$\LL^{\bA^1}_{\et}:\Spc\left( S \right) \to \Shv_{\et}\left(\Sch^{\ft}_S\right)_{\bA^1}$$ whose right adjoint can be identified with $R_{\bA^1}.$ Moreover, there is a factorization
$$\LL^{\bA^1}_{\et}=\LL^{\et}_{\mot} \circ \LL^{\et} \circ i,$$ where $$i:\Spc\left( S \right) \hookrightarrow \Shv_{\Nis}\left(\Sm_S\right)$$ is the canonical inclusion. Fix $\cX$ a motivic space, and $\cZ$ an object $S_{\bA^1}.$ There is the following string of natural equivalences
\begin{eqnarray*}
\map_{\Spc\left( S \right)}\left(\cX,\psi\left(\cZ\right)\right)&\simeq& \map_{\Spc\left( S \right)}\left(\cX,R_{\bA^1}\rho_{!,\bA^1}\left(\cZ\right)\right)\\&\simeq & \map_{\Shv_{\et}\left(\Sch^{\ft}_S\right)_{\bA^1}}\left(\LL^{\bA^1}_{\et}\left(\cX\right),\rho_{!,\bA^1}\left(\cZ\right)\right)\\
&\simeq & \map_{\Shv_{\et}\left(\Sch^{\ft}_S\right)_{\bA^1}}\left(\LL^{\et}_{\mot} \left( \LL^{\et} \left( i\left(\cX\right)\right)\right),\rho_{!,\bA^1}\left(\cZ\right)\right)\\
&\simeq & \map_{\Shv_{\et}\left(\Sch^{\ft}_S\right)}\left(\LL^{\et} \left( i\left(\cX\right)\right),\operatorname{R}_{\mot}^{\et}\left(\rho_{!,\bA^1}\left(\cZ\right)\right)\right)\\
&\simeq & \map_{\Shv_{\et}\left(\Sch^{\ft}_S\right)}\left(\LL^{\et} \left( i\left(\cX\right)\right),\rho_!\left(\cZ\right)\right)\\
&\simeq& \Pi_{\i}^{\et,S}\left(\LL^{\et} \left( i\left(\cX\right)\right)\right)\left(\cZ\right)\\
&\simeq & \Et^S_{\bA^1}\left(\cX\right)\left(\cZ\right).
\end{eqnarray*}

Here, we use Theorem~\ref{thm:concrete_relative} for the second last equivalence.
\end{proof}

\subsubsection{Comparison with the absolute realization functor} 

\begin{lemma}
There is a (essentially unique) functor $\delta$ making the following diagram commute
\begin{equation}\label{eq:delta}
\xymatrix{\Spc \ar[r]^-{\Delta_S} & \Shv\left(S_{\et}\right) \\
\Spc_{S,\bA^1} \ar@{-->}[r]^-{\delta} \ar@{^{(}->}[u] & S_{\bA^1}, \ar@{^{(}->}[u]}
\end{equation}
and $\delta$ is moreover left exact.
\end{lemma}

\begin{proof}
Consider the geometric morphism $$\lambda:\Shv_{\et}\left(\Sch^{\ft}_S\right) \to \Shv\left(S_{\et}\right)$$ of Proposition \ref{prop:infconn}. Then, since $\Spc$ is the terminal $\i$-topos, (\ref{eq:geomterm}) factors as $$\Shv_{\et}\left(\Sch^{\ft}_S\right) \stackrel{\lambda^*}{\longrightarrow} \Shv\left(S_{\et}\right) \stackrel{s}{\longrightarrow} \Spc.$$ So if $\Delta_S:=s^*,$ then since $\lambda^*= \rho_!,$ it follows that $$\Delta\simeq \rho_! \circ \Delta_S.$$ It follows from Proposition \ref{prop:LAinv} that $\Delta_S$ restricts to a functor $$\delta:\Spc_{S,\bA^1} \to S_{\bA^1},$$ which moreover is unique since both vertical arrows in the diagram in the statement of the lemma are fully faithful. Since $\Delta_S$ is left exact, and $\Spc_{S,\bA^1}$ and $S_{\bA^1}$ are closed under finite limits in $\Spc$ and $\Shv\left(S_{\et}\right)$ respectively, it follows that $\delta$ is left exact.
\end{proof}

The following follows from Proposition \ref{prop:ladj}:

\begin{corollary}
There is an induced adjunction:
$$\xymatrix@C=2.3cm{\Pro\left(\Spc_{S,\bA^1}\right) \ar@<-0.5ex>[r]_-{\Pro\left(\delta\right)} & \Pro\left(S_{\bA^1}\right). \ar@<-0.5ex>[l]_-{\delta^*}}$$
\end{corollary}

\begin{proposition}\label{prop:deltasquare}
Let $s:\Shv\left(S_{\et}\right) \to \Spc$ be the (essentially) unique geometric morphism.
The following diagram commutes up to equivalence:
$$\xymatrix@C=1.5cm{\Pro\left(\Shv\left(S_{\et}\right)\right) \ar[r]^-{\Pro\left(s_!\right)} \ar[d]_-{\widehat{\left(\mspace{3mu}\cdot\mspace{3mu}\right)}_{S_{\bA^1}}} & \Pro\left(\Spc\right) \ar[d]^-{\widehat{\left(\mspace{3mu}\cdot\mspace{3mu}\right)}_{\Spc_{S,\bA^1}}}\\
\Pro\left(S_{\bA^1}\right) \ar[r]^-{\delta^*} & \Pro\left(\Spc_{S,\bA^1}\right).}$$
\end{proposition}

\begin{proof}
By uniqueness of adjoints, it suffices to prove that
$$\xymatrix@C=2.5cm{\Pro\left(\Shv\left(S_{\et}\right)\right)  & \Pro\left(\Spc\right) \ar[l]_-{\Pro\left(s^*\right)=\Delta_S} \\
\Pro\left(S_{\bA^1}\right) \ar@{^{(}->}[u]^-{i_{S_{\bA^1}}} & \Pro\left(\Spc_{S,\bA^1}\right) \ar@{^{(}->}[u]_-{i_{\Spc_{S,\bA^1}}} \ar[l]^-{\Pro\left(\delta\right)}}$$
commutes up to equivalence. This follows by applying the functor $\Pro$ to the diagram (\ref{eq:delta}) of $\i$-categories.
\end{proof}

\begin{proposition}
The following diagram commutes up to equivalence:
$$\xymatrix{ \Spc\left( S \right) \ar[r]^-{\Et^S_{\bA^1}} \ar[rd]_-{\Et_{\bA^1}} & \Pro\left(S_{\bA^1}\right) \ar[d]^-{\delta^*}\\
& \Pro\left(\Spc_{S,\bA^1}\right).}$$
\end{proposition}

\begin{proof}
By definition, the following diagram commutes
$$\xymatrix@C=2.5cm{\Spc^{\bA^1}_{S} \ar@{^{(}->}[r]^-{i} \ar[rrrdd]_-{\Et^{S}_{\bA^1}} & \Shv_{\Nis}\left(\Sm_S\right) \ar[r]^-{L^{\et}} & \Shv_{\et}\left(\Sch^{\ft}_S\right) \ar[r]^-{\Shv\left(\mspace{3mu}\cdot\mspace{3mu}\right)/\Shv\left(S_{\et}\right)} & \Topi/\Shv\left(S_{\et}\right) \ar[d]^-{\Shape_S}\\
& & & \Pro\left(\Shv\left(S_{\et}\right)\right) \ar[d]^-{\widehat{\left(\mspace{3mu}\cdot\mspace{3mu}\right)}_{S_{\bA^1}}}\\
& & & \Pro\left(S_{\bA^1}\right).}$$
However, the following diagram also commutes
$$\xymatrix{\Spc^{\bA^1}_{S} \ar@{^{(}->}[r]^-{i} \ar[rrrdd]_-{\Et^{S}_{\bA^1}} & \Shv_{\Nis}\left(\Sm_S\right) \ar[r]^-{L^{\et}} & \Shv_{\et}\left(\Sch^{\ft}_S\right) \ar[r] & \Topi/\Shv\left(S_{\et}\right) \ar[d]_-{\Shape_S} \ar[r] & \Topi \ar[d]^-{\Shape}\\
& & & \Pro\left(\Shv\left(S_{\et}\right)\right) \ar[d]^-{\widehat{\left(\mspace{3mu}\cdot\mspace{3mu}\right)}_{S_{\bA^1}}} \ar[r]^-{\Pro\left(s_!\right)} & \Pro\left(\Spc\right) \ar[d]^-{\widehat{\left(\mspace{3mu}\cdot\mspace{3mu}\right)}_{\Spc_{S,\bA^1}}}\\
& & & \Pro\left(S_{\bA^1}\right) \ar[r]^-{\delta^*} & \Pro\left(\Spc_{S,\bA^1}\right),}$$
since the right most upper square commutes by Proposition \ref{prop:composable}, and the right most bottom square commutes by Proposition \ref{prop:deltasquare}. Now note that the composition of the top horizontal functors with the right most vertical functors is precisely the definition of $\Et_{\bA^1}.$
\end{proof}

\begin{remark}
An alternative proof may be given as follows:
\begin{eqnarray*}
\delta^*\left(\Et^S_{\bA^1}\left(\cX\right)\right)\left(V\right) & = & \Et^S_{\bA^1}\left(\cX\right)\left(\Delta_S\left(V\right)\right)\\
&\simeq& \map_{\Shv_{\et}\left(\Sch^{\ft}_S\right)}\left(\LL^{\et}i\left(\cX\right),L\Delta_S\left(V\right)\right)\\
&\simeq& \map_{\Shv_{\et}\left(\Sch^{\ft}_S\right)}\left(\LL^{\et}i\left(\cX\right),\Delta\left(V\right)\right)\\
&\simeq& \reallywidehat{\Pi_{\i}^{\et}\left(\LL^{\et}i\left(\cX\right)\right)}_{\Spc_{S,\bA^1}}\left(V\right)\\
&\simeq& \Et_{\bA^1}\left(\cX\right)\left(V\right).
\end{eqnarray*}
\end{remark}

\begin{remark}\label{realization:versus-comparison}
There is another adjunction that relates motivic homotopy theory with \'{e}tale homotopy theory \cite[Section 4.1]{morel-voevodsky}: $$\pi^*:\Spc\left( S \right) \rightleftarrows: \Spc^{\bA^1}_{S, \et}: \pi_*$$ where the $\infty$-category $\Spc^{\bA^1}_{S, \et}$ is constructed in the same way as in its Nisnevich counterpart, just replacing everything with the \'{e}tale topology. The way this relates to our realization functors is as follows: our functor is induced by the functor $$\Sm_S \rightarrow \Topiy \rightarrow \Pro(\Topiy^{\et} \simeq \sY)$$ $$X \mapsto \Shv_{\et}(X) \mapsto \Shape_{\Shv_{\et}(S)}(1_{\Shv_{\et}(X)})$$ 

while the functor above is induced by $$\Sm_S \rightarrow \Topiy \rightarrow \sY$$ $$X \mapsto \Shv_{\et}(X) \mapsto \Gamma_{\Shv_{\et}(S)}(1_{\Shv_{\et}(X)}).$$ In other words our functor is induced by the pro-left adjoint to the functor $$f^*: \Shv_{\et}(S) \rightarrow \Shv_{\et}(X),$$ while the other functor is induced by its right adjoint.
\end{remark}

Lastly, we have the following

\begin{proposition} \label{prop:right-adjoint-a1} The functor $\Et^S_{\bA^1}: \Spc\left( S \right) \rightarrow \Pro\left(S_{\bA^1}\right)$ admits a right adjoint, which identifies with the composite 
\[
R_{\bA^1}: \Pro\left(S_{\bA^1}\right) \rightarrow \Pro\left(\Spc\left(S \right)\right) \stackrel{\lim}{\longrightarrow}\Spc\left( S \right).
\]
\end{proposition}
whose proof is analogous to that of Proposition \ref{prop:right-adjoint}.

\subsubsection{$\bA^1$-invariant \'{e}tale stacks} \label{subsec:examples}

In this section, we give examples of objects in $S_{\bA^1}$, i.e., objects in $\Shv\left(S_{\et}\right)$ which are $\bA^1$-invariant. The target of our relative \'etale realization functors are constructed out of formal cofiltered limits of these concrete objects.

\begin{example} \label{em-examples} The first example comes from the classical theory of the derived category of \'etale sheaves with torsion coefficients. Let $S$ be any scheme and $\Lambda$ a commutative ring, then we may consider the $\infty$-category of \'etale sheaves valued in complexes of $\Lambda$-modules, i.e., presheaves of $H\Lambda$-modules satisfying \'etale descent. We denote this by $\Shv\left(S_{\et}, \Lambda \right)$. \footnote{A discussion of this $\infty$-category may be found in \cite[Section 2.2]{gaitsgory-lurie}, specifically \cite[Definition 2.2.1.2]{gaitsgory-lurie}, where it is proved that this $\infty$-category is obtained by taking the derived $\infty$-category (in the sense of \cite[Section 1.3.5]{higheralgebra}) of the abelian category of \'etale sheaves valued in $\Lambda$-modules.} We simply refer to $\Shv\left(S_{\et}, \Lambda \right)$ as the \emph{derived category of \'etale sheaf on $S$ with coefficients in $\Lambda$} and objects there as complexes of \'etale sheaves of $\Lambda$-modules.

We have an adjunction 
\[
C_*: \Shv\left(S_{\et} \right) = \Fun\left(\Et_S^{\op}, \Spc \right) \rightleftarrows \Fun\left(\Et_S^{\op}, \Mod_{H\Lambda} \right) = \Shv\left(S_{\et}, \Lambda \right): K
\]
induced by the usual adjunction
\[
H\Lambda \otimes -: \Spc \rightleftarrows \Mod_{H\Lambda}: K
\]
where $H \Lambda \otimes -$ takes a space to its free $H\Lambda$-module spectrum whose homotopy groups are just the homology of $X$ with coefficients in $\Lambda$ and $K$ is the generalized Eilenberg-Maclane space functor. In this, way, we can consider any complex of \'etale sheaf of $\Lambda$-modules as a sheaf of spaces on the small \'etale site.

\begin{proposition} \label{lem:lambda} Let $S$ be a locally Noetherian scheme. Let $P$ be a nonempty collection of primes and let $\Lambda$ be a ring which is $P$-torsion \footnote{This means that every element of $\Lambda$ is killed by an integer which is a product of primes in $P$}. Suppose that $S$ is a scheme such that every $p \in P$ is invertible in the residue fields of $S$. Then for any $M \in \Shv\left(S_{\et}, \Lambda \right)$, $K\left(M\right)$ is in $S_{\bA^1}$.
\end{proposition}

\begin{proof} Suppose that $p:X \rightarrow S$ is a smooth $S$-scheme and $\pi_X: X \times \bA^1 \rightarrow X$ is the projection map. We have an adjunction $\pi_X^*: \Shv\left(S_{\et}, \Lambda \right) \rightleftarrows \Shv\left(S \times \bA^1, \Lambda \right): \pi_{X*}$. It suffices to prove that the unit map $p^*M \rightarrow \pi_X^*\pi_{X*}p^*M$ is an equivalence which is proved in, for example, \cite[Theorem 1.3.2]{etalemotives}.
\end{proof}

Hence discrete sheaves of torsion abelian groups, with torsion prime to the residue characteristics of $S$, are example. A particularly important one, which is also a representable sheaf in $\Shv(S_{\et})$, is the sheaf classifying roots of unity: $\mu_{\ell}: T \mapsto \ker\left( \mathscr{O}_T \stackrel{\times \ell}{\rightarrow} \mathscr{O}_T \right)$; here $\ell$ is assumed to be invertible in the residue characteristics of $S$.

\end{example}

\begin{example} \label{example-strict} Recall the following terminology (and their \'etale variants) 

\begin{definition} Let $S$ be a base scheme and suppose that $\sC$ is either $\Sm_S$ or $\Et_S$ 

\begin{enumerate}
\item An $\mathcal{E}_1$-monoid $M \in \Mon\left(\Shv_{\Nis}\left(\sC \right) \right)$ (resp. $\Mon(\Shv_{\et}(\sC))$) is \emph{strongly $\bA^1$-invariant} (\emph{strongly $\bA^1$-invariant for the \'etale topology}) if it is $\bA^1$-invariant and $\B_{\Nis}M$ (resp. $\B_{\et}M$) is $\bA^1$-invariant.
\item An $\mathcal{E}_{\infty}$-monoid $M \in \CMon\left(\Shv_{\Nis}\left(\sC \right)\right)$ (resp. $\CMon\left(\Shv_{\et}\left(\sC\right)\right)$) is \emph{strictly $\bA^1$-invariant} (\emph{strictly $\bA^1$-invariant for the \'etale topology}) if it is $\bA^1$-invariant and $\B^n_{\Nis}M$  (resp. $\B_{\et}M$) is $\bA^1$-invariant for all $n \geq 1$.
\end{enumerate}

\end{definition}

Of course, the above definition applies for $M$ a sheaf of discrete groups/abelian groups. When $S$ is the spectrum of a perfect field, $\sC = \Sm_S$ and $\tau = \Nis$ we have a theorem of Morel asserting that an $\mathcal{E}_1$-monoid (resp. $\mathcal{E}_{\infty}$-monoid) $M$ is strongly (resp. strictly) $\bA^1$-invariant if and only if the discrete sheaf $\pi_0(M)$ is; see \cite[Theorem 3.1.12]{ehksy} for a proof. Unfortunately, the \'etale analogue of this result is not known.

\begin{lemma} \label{lem:post} Let $S$ be a base scheme and suppose that $\sC$ is either $\Sm_S$ or $S_{\et}$. Suppose that $F$ is a hypercomplete sheaf on $\sC$ in the Nisnevich (resp. \'etale) topology satisfying the following assumptions
\begin{enumerate}
\item The sheaf $F$ is connected, i.e., $\pi_0\left(F\right)$ is terminal.
\item The homotopy sheaf of groups $\pi_1\left(F\right)$ is strongly $\bA^1$-invariant (resp. strongly $\bA^1$-invariant for the \'etale topology).
\item For $k \geq 2$ the homotopy sheaf of abelian groups $\pi_k\left(F\right)$ is strictly $\bA^1$-invariant  (resp. strictly $\bA^1$-invariant for the \'etale topology).
\end{enumerate}
Then $F$ is $\bA^1$-invariant.
\end{lemma}

\begin{proof} We prove this by induction along the Postnikov tower of $F$ (in all the possible settings):
\[
\cdots \tau_{\leq k}F \rightarrow \cdots \rightarrow \tau_{\leq 1}F \rightarrow \tau_{\leq 0}F.
\]
Indeed, property of being $\bA^1$-invariant is closed under limits, hence it suffices to prove that for each $k$, $\tau_{\leq k}F$ is $\bA^1$-invariant and appeal to the fact that $F$ is hypercomplete so that $F \simeq \lim \tau_{\leq n}F$. The base case of the induction uses point $(1)$. We also have an equivalence $B_{\tau}\pi_1(F) \rightarrow \tau_{\leq 1}F$, whence item $(2)$ is used to conclude $\bA^1$-invariance for $\tau_{\leq 1}F$

In general, for $k \geq 2$, we assume that $\tau_{\leq n-1}F$ is $\bA^1$-invariant. We we have a fiber sequence (for $\tau = \et$ or $\Nis$)
\[
B^n\pi_n(F) \rightarrow \tau_{\leq n}F \rightarrow \tau_{\leq n-1}F,
\]
We want to prove that for any $X \rightarrow S$ a smooth $S$-scheme and $\pi_X: X \times \bA^1 \rightarrow X$, the map $\tau_{\leq n}F \rightarrow \pi_{X*}\pi_X^*\tau_{\leq n}F$ is an equivalence. Without loss of generality we can assume that $S = X$. Since $\pi_{X*}\pi_X^*$ preserves finite limits, we have a diagram where the rows are fiber sequences.
\[
\xymatrix{
B^n\pi_n(F) \ar[r] \ar[d] & \tau_{\leq n}F \ar[r] \ar[d] &  \tau_{\leq n-1}F \ar[d] \\
\pi_{S*}\pi_S^*B^n\pi_n(F)  \ar[r] &  \pi_{S*}\pi_S^*\tau_{\leq n}F  \ar[r] & \pi_{S*}\pi_S^*B^n\tau_{\leq n-1}F.
}
\]
The right vertical map is an equivalence by assumption and the left vertical map is an equivalence by point $(3)$. We thus conclude using the long exact sequence in homotopy sheaves \cite[Remark 6.5.1.5]{htt}.

\end{proof}

The above lemma gives us a mechanism to check that if a connected object in $\Shv(S_{\et})$ is indeed in $S_{\bA^1}$. Using Example~\ref{em-examples} above, if the homotopy sheaves of $F$ are sheaves of torsion abelian groups  with torsion prime to the residue characteristics of $S$, we get a whole family of examples over locally Noetherian bases. If we work over fields, then we have the following Lemma, using results of Voevodsky in \cite{pst-coh} and Suslin rigidity \cite[Lecture 7]{mvw}, which lets us work with integral coefficients (after inverting the exponential characteristic of $k$)
\begin{lemma} \label{lem:voev} [\cite[Lemma 4.3.7]{asok-morel}] Let $k$ be a field of exponential characteristic $c$. Suppose that $A$ is an $\bA^1$-invariant \'etale sheaf of $\mathbb{Z}[1/c]$-modules with transfers (in the sense of Suslin-Voevodsky \cite[Chapter 6]{mvw}\footnote{This means that $\pi_iF$ is the restriction of a sheaf $\tilde{F}: \Sm\Cor_k^{\op} \rightarrow \mathrm{Ab}$ along the map $\Et_k \hookrightarrow \Sm_k \rightarrow \Sm\Cor_k$; see \cite[Chapter 1]{mvw} for details on $\Sm\Cor_k$.}). Then $A$ is strictly $\bA^1$-invariant for the \'etale topology.
\end{lemma}

Hence, putting Lemmas~\ref{lem:voev} and Lemmas~\ref{lem:post} together, we get that the restriction of an \'etale sheaf of $\mathbb{Z}[1/c]$-modules transfers to the small \'etale site produces a class of examples of objects in $S_{\bA^1}$ whose homotopy sheaves need not be torsion.

There are also examples coming from algebraic groups using \cite[Lemme 3.1.2]{orgogozo}: if $G$ is a smooth commutative $k$-group scheme, then the underlying \'etale sheaf naturally has transfers. This means that if $k$ is of characteristic zero then $G$ is strictly $\bA^1$-invariant whence for all $n \geq 1$, $\B^n_{\et}G \in S_{\bA^1}$. 

\end{example}

\begin{example} We again work over a field $k$. The previous example discusses mainly \'etale sheaves whose homotopy groups are abelian groups. There are also interesting examples of objects in $S_{\bA^1}$ which are of the form $\B_{\et}G$ where $G$ is group which is not necessarily abelian. These are strongly $\bA^1$-invariant sheaves of groups as discussed above. Over a field $k$ of characteristic zero we have the following classification result due to Asok and Morel \cite[Proposition 4.4.3]{asok-morel}

\begin{proposition} \label{prop:asok-morel} Suppose that $k$ is a field of characteristic zero and $G$ is a smooth group scheme over $k$. Then $G$ is strongly $\bA^1$-invariant if and only if its connected component, $G^0$, is the extension of an abelian variety by a torus. 
\end{proposition}

Lastly, if $X \in S_{\bA^1}$, then consider the homotopy sheaf $\pi_1X$, i.e., the $\bA^1$-\'etale fundamental group. Unlike its Nisnevich counterpart \cite[Theorem 5.1]{morel-book}, it is \emph{not known} that $\pi_1X$ is a strongly $\bA^1$-invariant sheaf of groups, but expected \cite[Remark 5.1.6]{asok-morel}. These groups would provide an interesting class of examples of objects in $S_{\bA^1}$.

\end{example}

\section{Stable Realization Functor} \label{sect:stable}

We now discuss stabilizations of the previous construction. There are four parts to this section. Firstly, in subsection~\ref{sect:stable-prof}, we develop a fully $\infty$-categorical approach to stable profinite homotopy theory. From this discussion, we stabilize the functor from Theorem~\ref{thm:main-functor} in Theorem~\ref{thm:stable}. We finish off with two applications of our technology. In subsection~\ref{subsec:revisit} we revisit Dwyer and Friedlander's construction of \'etale $K$-theory and give a refinement of it. In subsection~\ref{subsec:noncomm} we define a new invariant  of an $R$-linear stable $\infty$-category $\sC$. We call this the (relative) \emph{Dwyer-Friedlander $K$-theory} associated to $\sC$, as well as its $\bA^1$-invariant analog. Using standard results in motivic homotopy theory, we compute the value of the connective, $\bA^1$-invariant version on the unit object of $\Cat_R$ in Proposition~\ref{thm:point}. In subsection~\ref{sec:blancdf}, we show that over a field and with finite coefficients (coprime to the characteristic of the field), this invariant agrees with usual algebraic $K$-theory. This is one of the expected properties that this theory should have.

%
%

\subsection{Stable Profinite Homotopy Theory} \label{sect:stable-prof}

A discussion of stable profinite homotopy theory using model categories was first due to Quick in \cite[Section 2.2]{Qu3}, and \cite[Section 2.8]{Qu}. An important aspect of the theory is that it is \emph{not} the category of pro-objects in (finite) spectra but, rather, spectra constructed from profinite spaces; in the language of higher algebra we mean \emph{spectrum objects} in profinite spaces. 

\subsubsection{Stable $\infty$-categories of profinite objects}

To begin we quickly review the process of stabilization as explained by Lurie in \cite[Section 1.4.2]{higheralgebra}. Let $\Cat_{\infty}^{\lex}$ be the $\infty$-category of small $\infty$-categories with finite limits and functors that preserve finite limits (in other words, left exact). There is a functor $$\Spt: \Cat_{\infty}^{\lex} \rightarrow \Cat_{\infty, \stab}^{\lex}$$ which assigns to $\sC$, $\Spt\left(\sC \right) $, its $\infty$-category of \emph{spectrum objects} \cite[Definition 1.4.2.8]{higheralgebra}. By definition $\Spt\left( \sC \right)$ is the full subcategory of $\Fun\left(\Spc_*^{\fin}, \sC\right)$ spanned by functors which 
\begin{itemize}
\item Takes the terminal object to the terminal object (\emph{reduced}), and
\item takes a pushout square to a pullback square (\emph{excisive}).
\end{itemize}

The $\infty$-category $\Spt \left(\sC \right)$ is stable \cite[Corollary 1.4.2.17]{higheralgebra} and there is a canonical finite limit-preserving functor $\Omega^{\infty}_{\sC}: \Spt \left(\sC \right) \rightarrow \sC$ satisfying the following universal property \cite[Corollary 1.4.2.23]{higheralgebra}: given a stable $\infty$-category $\sD$, then the induced functor $$\Fun^{\lex}\left(\sD, \Spt\left(\sC\right)\right) \rightarrow \Fun^{\lex}\left(\sD, \sC \right)$$ is an equivalence. Lastly, assuming that $\sC$ is furthermore \emph{pointed}, the $\infty$-category $\Spt\left(\sC\right)$ is calculated as the limit:  $$\cdots \rightarrow \sC \stackrel{\Omega_{\sC}}{\rightarrow} \sC \stackrel{\Omega_{\sC}}{\rightarrow} \sC $$ in $\Cat_{\infty}$  \cite[Proposition 1.4.2.24, Remark 1.4.2.25]{higheralgebra}. Here the functor $\Omega_{\sC}: \sC \rightarrow \sC$ is given by $$X \mapsto \lim\left( * \rightarrow X \leftarrow *\right),$$ where $*$ is the terminal object. The functor $\Omega^{\infty}_{\sC}: \Spt\left( \sC \right) \rightarrow \sC$ identifies with the tautological functor. 

\begin{remark} \label{rem:base-points} We remark that adding the basepoint is a harmless procedure from the point of view of stabilization. Indeed, if $\sC$ is an $\infty$-category with finite limits and $* \in \sC$ is the terminal object, then define $\sC_*:= \sC_{*/}$. In this case \cite[Remark 1.4.2.18]{higheralgebra} furnishes us with equivalences
\[
\Spt\left( \sC_* \right) \simeq \Spt\left(\sC\right)_* \simeq \Spt\left(\sC\right).
\]
\end{remark}

In the situation of interest, $\sC$ is accessible and has finite limits. Therefore, $\Pro\left(\sC\right)$ has all finite limits and $\sC \rightarrow \Pro\left( \sC \right)$ preserves finite limits by Proposition~\ref{prop:finlim}; in particular the terminal object of $\Pro \left( \sC \right)$ is the corepresented by the terminal object of $\sC$.

\begin{definition} \label{def:stables}We define the following stable $\infty$-categories
\begin{enumerate}
\item The $\infty$-category of \emph{profinite spectra} is the $\infty$-category $\Prof\left( \Spt \right) := \Spt{\left(\Pro\left(\Spc^\pi \right) \right)}$. 
\item The $\infty$-category of \emph{$p$-profinite spectra} is the $\infty$-category $\Prof_p\left( \Spt \right):= \Spt{\left( \Pro\left(\Spc^{p-\pi} \right) \right)}$.
\item Let $S$ be a scheme, the $\infty$-category of \emph{$S$-relatively $\bA^1$-invariant profinite spectra} is the $\infty$-category $\Prof_{S,\bA^1}\left( \Spt \right) :=  \Spt \left( \Pro\left( \Spc_{S,\bA^1} \right) \right)$.
\item  Let $S$ be a scheme, the $\infty$-category of \emph{$S$-profinite $\bA^1$-invariant spectra} is the $\infty$-category $\Prof_{S,\bA^1,\et}\left(\Spt \right):=\Spt{\left(\Pro\left(S_{\bA^1} \right) \right)}$.
 \end{enumerate}
\end{definition}

Let $\sC$ be any of the $\infty$-categories defined in~\ref{def:stables}. To proceed further, we define a left adjoint to $\Omega^{\infty}_{\Pro\left( \sC \right)}$. If $\sC$ is a presentable $\infty$-category, then by \cite[Proposition 1.4.4.4]{higheralgebra}, one obtains an adjunction
\begin{equation} \label{susp-loop}
\Sigma^{\infty}_+: \sC \rightleftarrows \Spt(\sC): \Omega^{\infty}.
\end{equation} This, of course, does not immediately apply to our situation since pro-categories are not, in general, presentable. 

\subsubsection{Co-spectra}
To construct the adjunction~\eqref{susp-loop} and various completions of spectra, it will be useful to consider the following notions

\begin{definition} \label{def:co-sp} Suppose that $\sC$ is an $\infty$-category with finite colimits and $\sD$ is an $\infty$-category with finite limits. Then, a functor $F: \sD \rightarrow \sC$ is
\begin{itemize}
\item \emph{coexcisive} if it takes a Cartesian square in $\sD$ to a coCartesian square in $\sC$. 
\item If $\sD$ furthermore admits initial objects, it is \emph{coreduced} if it takes an initial object to an initial object.
\end{itemize}

We denote by $\coExc\left(\sD, \sC\right)$ (resp. $\coExc_{\emptyset}\left(\sD, \sC\right)$ if $\sC$ has a final object) the full subcategory of $\Fun\left(\sC, \sD) \right)$ spanned by coexcisive (resp. coexcisive and coreduced) functors.
\end{definition}

\begin{definition} Suppose that $\sC$ is an $\infty$-category with finite colimits. The $\infty$-category of \emph{co-spectrum objects in $\sC$} is the full subcategory of $\Fun\left(\Spc^{\fin,\op}_*, \sC\right)$ spanned by functors which are coreduced and coexcisive, i.e.,
 \[
 \cSpt\left( \sC \right):=\coExc_{\emptyset}\left(\Spc^{\fin,\op}_*,\sC \right).
 \]
 We refer to objects in $\cSpt\left( \sC \right)$ simply as a \emph{cospectrum}.
\end{definition}

These definitions are motivated by the following observation

\begin{lemma} \label{prop:cosp}  Let $\sC$ be an $\infty$-category with finite limits and $\sD$ be an $\infty$-category with finite colimits and a terminal object, then there is a canonical equivalences
\[
\Exc\left(\sD, \sC\right) \simeq \coExc\left(\sD^{\op}, \sC^{\op}\right)^{\op},
\]
and
\[
\Exc_*\left(\sD, \sC\right)\simeq \coExc_{\emptyset}(\sD^{\op}, \sC^{\op})^{\op}.
\]
In particular, there is a canonical equivalence,
\[
 \left( \cSpt\left( \sC^{\op} \right) \right)^{\op} \simeq \Spt\left( \sC \right).
\]
\end{lemma}

\begin{proof} We have a fully faithful embedding 
\[
\coExc_{\emptyset}\left(\sD^{\op}, \sC^{\op}\right)^{\op} \hookrightarrow \left( \Fun\left(\sD^{\op}, \sC^{\op} \right) \right)^{\op} \simeq \Fun\left( \sD, \sC \right),
\]
where the last equivalence is implemented by taking $F: \sD^{\op} \rightarrow \sC^{\op}$ to $F^{\op}: \sD \rightarrow \sC$. Now, by definition, we see that $F$ is coexcisive if and only $F^{\op}$ is excisive and $F$ is coreduced if and only if $F^{\op}$ is reduced.

\end{proof}

In particular, Proposition~\ref{prop:cosp} tells us that if $\sC$ be a small $\infty$-category, there is a canonical equivalence
\begin{equation} \label{eq:useful}
\left( \cSpt\left(\Ind\left(\sC^{\op} \right)\right) \right)^{\op} \simeq  \Spt \left( \Pro \left( \sC \right) \right).
\end{equation}
Hence, to construct various adjunctions for spectrum objects in pro-categories, we will first construct these adjunctions on the level of co-spectrum objects in $\Ind\left( \sC^{\op} \right)$ and pass to opposites. The main point is that co-spectra of $\Ind\left( \sC^{\op} \right)$ is presentable whenever $\sC$ is small, see Lemma~\ref{lem:coexc-pres} below.

Using Lemma~\ref{prop:cosp}, we deduce some basic properties of cospectra

\begin{proposition} \label{prop:list} Let $\sC$ be an $\infty$-category with finite colimits and $\sD$ be an $\infty$-category with finite limits and an initial object. Then the following properties hold:
\begin{enumerate}
\item Suppose that $\sD$ is furthermore pointed, then $\coExc_{\emptyset}\left(\sD, \sC\right)$ is pointed and admits finite colimits.
\item Let $K$ be a simplicial set and suppose that $\sC$ admits $K$-indexed colimits. Then the $\infty$-category $\coExc\left(\sD, \sC\right)$ is closed under $K$-indexed colimits. If $\sD$ furthermore has initial objects, then $\coExc_{\emptyset}\left(\sD, \sC\right)$ is also closed under $K$-indexed colimits. More precisely, the inclusions
\begin{equation*} 
\coExc_{\emptyset}\left(\sD, \sC\right) \hookrightarrow \coExc\left(\sD, \sC\right) \hookrightarrow \Fun\left(\sD, \sC\right),
\end{equation*}
preserves $K$-indexed colimits.
\item The $\infty$-category $\coExc_{\emptyset}(\sD, \sC)$ is stable.
\item Suppose that $\sD$ is furthermore pointed, there is a canonical equivalence $\coExc_{\emptyset}\left(\sD, \sC_*\right) \simeq \coExc_{\emptyset}\left(\sD, \sC\right)$. In particular, there is a canonical equivalence $\cSpt \left( \sC_* \right) \simeq \cSpt \left( \sC \right)$.
\end{enumerate}
\end{proposition}

\begin{proof} Lemma~\ref{prop:cosp} tells us that $\coExc_{\emptyset}\left(\sD, \sC \right) \simeq \left(\Exc_*\left(\sC^{\op}, \sD^{\op}\right)\right)^{\op}$.  For (1), we apply \cite[Lemma 1.4.2.10]{higheralgebra} to $\Exc_*\left(\sC^{\op}, \sD^{\op}\right)$ to conclude that it is pointed and admits finite limits, and then use the fact that the opposite of a pointed $\infty$-category is pointed. For (3), we apply~\cite[Proposition 1.4.2.16]{higheralgebra} to $\Exc_*\left(\sC^{\op}, \sD^{\op}\right)$ to conclude that it is stable, and then use the fact that the opposite of a stable $\infty$-category is stable \cite[Remark 1.1.1.13]{higheralgebra}. For (4), we argue as in \cite[Remark 1.4.2.18]{higheralgebra}: first note that there is a canonical isomorphism of simplicial sets~$\coExc_{\emptyset}\left(\sD, \sC_*\right) \simeq \coExc_{\emptyset}\left(\sD, \sC\right)_*$. Then use (1) to conclude that $\coExc_{\emptyset}\left(\sD, \sC\right)$ is already pointed.

To conclude (2), we argue as follows: suppose that $F: K \rightarrow \coExc\left(\sD, \sC\right)$ is diagram and $\colim F$ is the colimit in $\Fun\left(\sD, \sC\right)$. Then, since colimits in functor categories are computed pointwise, it is clear that $F$ is coexcisive. Similarly since an initial object is a colimit, the second claim follows.

\end{proof}

\newcommand{\ZZ}{\mathbb{Z}}
Recall that an $\infty$-category $\sC$ is \emph{differentiable} if (1) it admits finite limits, (2) it admits $\ZZ_{\geq 0}$-indexed colimits, and (3) the functor $\colim: \Fun\left( N\left( \ZZ_{\geq 0} \right), \sC\right) \rightarrow \sC$ is left exact. If $\sD$ is an $\infty$-category with finite colimits and final objects, then the differentiability hypotheses ensures that the inclusion $\Exc\left( \sD, \sC \right) \hookrightarrow \Fun\left(\sC, \sD \right)$ admits a left adjoint; this is the case $n=1$ of \cite[Theorem 6.1.1.10]{higheralgebra}.

\begin{lemma} \label{lem:coexc-coloc}  Let $\sC$ be an $\infty$-category such that $\sC^{\op}$ is differentiable and $\sD$ is a small $\infty$-category with finite limits. Then the inclusion $\coExc\left(\sD, \sC\right) \hookrightarrow \Fun\left(\sD, \sC\right)$ admits a right adjoint
\[
\mathrm{co}P_1:  \Fun\left(\sD, \sC\right) \rightarrow \coExc\left(\sD, \sC\right)
\]
\end{lemma}

\begin{proof} To construct a right adjoint $\mathrm{co}P_1$ it suffices to construct a left adjoint after taking opposites
\[
P_1: \Fun\left(\sD, \sC\right)^{\op} \simeq \Fun\left(\sD^{\op}, \sC^{\op}\right) \rightarrow \coExc\left(\sD, \sC\right)^{\op} \simeq \Exc_*\left(\sD^{\op}, \sC^{\op}\right).
\]
The result is then an immediate consequence of \cite[Theorem 6.1.1.10]{higheralgebra} since $\sD^{\op}$ has finite colimits and $\sC^{\op}$ is differentiable.
\end{proof}

In other words, in the situation of Lemma~\ref{lem:coexc-coloc}, the $\infty$-category $\coExc\left(\sD, \sC\right)$ is a colocalization of $\Fun\left(\sD, \sC\right)$.

\begin{proposition} \label{prop:sigma-infty} 
Suppose that $\sC$ is a small, accessible $\infty$-category with finite limits and colimits, then the functor $\Omega^{\infty}_{\Pro(\sC)}: \Spt \left( \Pro\left( \sC \right) \right) \rightarrow \Pro \left(\sC \right)$ admits a left adjoint 
\[
\Sigma^{\infty}_{\Pro\left( \sC \right)+}:   \Pro \left(\sC \right) \rightarrow  \Spt \left( \Pro\left( \sC \right) \right)
\]
\end{proposition}

\begin{proof} We first remark that the functor $\Omega^{\infty}: \Spt \left( \Pro\left( \sC \right) \right) \rightarrow \Pro \left(\sC \right)$ is computed by the functor of evaluation at $S_0$:
\[
\mathrm{ev}_{S^0}: \Exc\left(\Spc^{\fin}_*, \Pro\left(\sC \right)\right) \rightarrow \Pro\left(\sC \right); F \mapsto F(S^0).
\]

To construct the left adjoint $\Sigma^{\infty}_{\Pro\left( \sC \right)+}$, it suffices to construct a right adjoint to $\left(\mathrm{ev}_{S^0}\right)^{\op}$ after taking opposites and applying~\eqref{eq:useful}. Hence our goal is to construct a functor
\[
\mathrm{co}\Sigma^{\infty}_{+}:\Ind\left(\sC^{\op}\right) \rightarrow \coExc\left( \Spc^{\fin,\op}_*, \Ind\left(\sC^{\op}\right)\right).
\]
First, observe that the evaluation functor $\left(\mathrm{ev}_{S^0}\right)^{\op}: \Fun\left( \Spc^{\fin,\op}_*, \Ind\left(\sC^{\op}\right)\right) \rightarrow \Ind\left(\sC^{\op}\right)$ preserves small colimits since colimits are computed pointwise in functor categories. The $\infty$-category $\Ind\left(\sC^{\op}\right)$ is presentable by definition and $\Fun\left( \Spc^{\fin,\op}_*, \Ind\left(\sC^{\op}\right)\right)$ is presentable by \cite[Proposition 5.5.3.6]{htt} and hence the adjoint functor theorem \cite[Corollary 5.5.2.9]{htt} applies to give us a right adjoint: $\mathrm{co}\Sigma^{\infty'}_{+}: \Ind\left(\sC^{\op}\right) \rightarrow \Fun\left( \Spc^{\fin,\op}_*, \Ind\left(\sC^{\op}\right)\right)$. Now the $\infty$-category $\Ind\left(\sC^{\op} \right)$ is differentiable by \cite[Remark 6.1.1.9]{higheralgebra}, hence Lemma~\ref{lem:coexc-coloc} gives us a colocalization functor $\mathrm{co}P_1$. The desired right adjoint is then the composite $\mathrm{co}P_1 \circ \mathrm{co}\Sigma^{\infty'}_{+}$.

\end{proof}

%

\begin{lemma} \label{lem:coexc-pres} Let $\sC$ be a presentable $\infty$-category and $\sD$ a small pointed $\infty$-category with finite colimits. Then the $\infty$-categories $\coExc\left(\sD, \sC\right)$ and $\coExc_{\emptyset}\left(\sD, \sC \right)$ are presentable.
\end{lemma}

\begin{proof}Denote by $Q$ the set of all representatives of pullback diagrams in $\sD$. For any $\square \in Q$, we have the functor
\[
\mathrm{ev}_\square: \Fun\left(\sD, \sC\right) \rightarrow \Fun\left( \Delta^1 \times \Delta^1, \sC\right).
\]
By definition,  $\coExc\left(\sD, \sC\right)$ can presented as the following pullback in $\Cat_{\infty}$:

\begin{equation} \label{eq:coexc-as-pull}
\xymatrix{
\coExc\left(\sD, \sC\right) \ar[d] \ar[r] & \prod_{\square \in Q} \Fun\left( \Delta^1 \times \Delta^1, \sC \right) \ar[d]\\
\Fun\left(\sD, \sC \right) \ar[r]^-{\prod_{\mathrm{ev}_{\square}}} &  \prod_{\square \in Q} \Fun\left( \Lambda^2_0, \sC \right).
}
\end{equation}
Now, the $\infty$-categories $\Fun\left( \Delta^1 \times \Delta^1, \sC \right),  \Fun\left( \Lambda^2_0, \sC \right)$, and $\Fun\left(\sD, \sC \right) $ are all presentable by~\cite[Proposition 5.5.3.6]{htt} and the right vertical and bottom horizontal functors clearly preserve colimits since colimits are computed pointwise in functor categories. According to \cite[Proposition 5.5.3.13]{htt}, the $\infty$-category $\mathfrak{Pr}^L$ is closed under all limits and the forgetful functor $\mathfrak{Pr}^L \rightarrow \Cat_{\infty}$ preserves limits. Hence the square~\eqref{eq:coexc-as-pull} is Cartesian in $\mathfrak{Pr}^L$, whence $\coExc\left(\sD, \sC\right)$ is presentable.
Now the $\infty$-category $\Fun_{\emptyset}(\sD, \sC)$ of functors that preserves initial objects is also presentable as it is the full subcategory of functors that preserves the empty colimit and $\Exc_{\emptyset}\left(\sD, \sC \right) = \Exc\left(\sD, \sC \right) \cap \Fun_{\emptyset}(\sD, \sC)$. We thus conclude by applying~\cite[Proposition 5.5.3.13]{htt} again.

\end{proof}

\begin{lemma} \label{lem:stab-pro} Let $\sC$ be a presentable, differentiable $\infty$-category and let $\sD$ be a small $\infty$-category and suppose that we have an adjunction $F:\sC \rightleftarrows  \Pro\left(\sD\right): G$. Then there is an adjunction 
\begin{equation} \label{eq:stab-pro}
\partial F:\Spt \left(  \sC \right) \rightleftarrows \Spt \left(\Pro\left( \sD \right) \right): \partial G,
\end{equation}
fitting into the following commutative diagram
\begin{equation} \label{lefts}
\xymatrix{
\Spt \left(  \sC \right) \ar[r]^-{\partial F}  & \Spt \left(\Pro\left( \sD \right) \right)  \\
\sC \ar[r]^-{F} \ar[u]^{\Sigma^{\infty}_{\sC+}} & \Pro\left( \sD \right)\ar[u]_{\Sigma^{\infty}_{\Pro\left( \sD \right)+}}.
}
\end{equation}

\end{lemma}

\begin{proof} Since $G$ is a right adjoint, the induced functor $G_*: \Fun\left(\Spc^{\fin}_*, \Pro\left( \sD\right) \right)  \rightarrow  \Fun\left(\Spc^{\fin}_*, \sC \right) $ factors through the subcategory of reduced, excisive functors, whence there is a limit-preserving functor $ \partial G: \Spt \left(\Pro\left( \sD \right) \right) \rightarrow \Spt \left(  \sC \right)$ which fits into a commutative diagram
\begin{equation} \label{eq:big-diagram}
\xymatrix{
 \Spt\left(\sC\right) \ar[d] & &  \Spt\left(\Pro \left( \sD \right)\right) \ar[d] \ar[ll]_-{\partial G }\\
 \Fun\left(\Spc^{\fin}_*, \sC \right)  \ar[d]_{\Omega_{\sC}^{\infty}} & & \Fun\left(\Spc^{\fin}_*, \Pro \left( \sD \right) \right)  \ar[d]^{\Omega^{\infty}_{\Pro \left( \sD \right)} }  \ar[ll]_-{ G_*}\\
\sC  & &  \Pro\left( \sD \right) \ar[ll].}
\end{equation}
By uniqueness of adjoints, it suffices to produce a left adjoint to $\partial G $, from which the commutativity of~\eqref{lefts} is automatic. Taking opposites of the top diagram, we get a commutative diagram of colimit preserving functors
\[
\xymatrix{
 \cSpt\left(\sC^{\op}\right) \ar[d] & &  \cSpt\left(\Ind \left( \sD^{\op} \right)\right) \ar[d] \ar[ll]_-{\partial  G^{\op}}\\
 \Fun\left(\Spc^{\fin,\op}_*, \sC^{\op} \right)  & & \Fun\left(\Spc^{\fin,\op}_*, \Ind \left( \sD^{\op} \right) \right)   \ar[ll]_-{ G_*}\\
}
\]
According to Lemma~\ref{lem:coexc-pres}, $\cSpt\left(\Ind \left( \sD^{\op} \right)\right)$ is presentable. By the adjoint functor theorem \cite[Corollary 5.5.2.9]{htt} and \cite[Remark 5.5.2.10]{htt}, we thus obtain right adjoints to $G_*$ and $\partial  G^{\op}$  whence a diagram of right adjoints (the vertical are the colocalizaton functors furnished by Lemma~\ref{lem:coexc-coloc})
\[
\xymatrix{
 \cSpt\left(\sC^{\op}\right)  \ar[rr]^{\partial F^{op}}& &  \cSpt\left(\Ind \left( \sD^{\op} \right)\right)  \\
 \Fun\left(\Spc^{\fin,\op}_*, \sC^{\op} \right)  \ar[u]  \ar[rr]^{F^{\op}} & & \Fun\left(\Spc^{\fin,\op}_*, \Ind \left( \sD^{\op} \right) \right).   \ar[u]\\
}
\]
The functor $\partial F:= \left(\partial F^{\op}\right)^{\op}$ is the desired left adjoint.

\end{proof}

\begin{proposition} \label{prop:stab-pro} Let $\sC$ be a presentable and differentiable $\infty$-category and $f:\sD \hookrightarrow \sC$ be a full subcategory which is accessible and closed under finite limits and colimits. Then there is an adjunction 
\begin{equation} \label{eq:stab-pro}
\partial f^*:\Spt \left(  \sC \right) \rightleftarrows \Spt \left(\Pro\left( \sD \right) \right): \partial \left( \Pro \left( f\right) \right)
\end{equation}
fitting into the following commutative diagram
\begin{equation} \label{lefts}
\xymatrix{
\Spt \left(  \sC \right) \ar[r]^-{\partial f^*}  & \Spt \left(\Pro\left( \sD \right) \right)  \\
\sC \ar[r]^-{f^*} \ar[u]^{\Sigma^{\infty}_{\sC+}} & \Pro\left( \sD \right)\ar[u]_{\Sigma^{\infty}_{\Pro\left( \sD \right)+}}.
}
\end{equation}
\begin{proof} This is an immediate consequence of Lemma~\ref{lem:stab-pro}, the adjunction obtained in Proposition~\ref{prop:ladj} and the fact that $\Pro\left( \sD \right) \stackrel{\Pro\left(f \right)}{\longlongrightarrow} \Pro\left( \sC \right) \stackrel{\lim}{\longrightarrow} \sC$ preserves limits.

\end{proof}

\end{proposition}


As the sum of our effort, we construct various ``profinite completions" on the stable level.

\begin{definition} \label{def:stables} We define the following functors of $\infty$-categories via the left adjoint furnished by Proposition~\ref{prop:stab-pro}
\begin{enumerate}
\item The functor of \emph{stable profinite completion} is the functor
\[\widehat{\left(\blank\right)}:= \partial f^*:\Spt \rightarrow  \Prof\left( \Spt \right) ,
\] for $f: \Spc^{\pi} \hookrightarrow \Spc$.
 \item The functor of \emph{stable $p$-profinite completion} is the functor
\[\widehat{\left(\blank\right)}_p:= \partial f^*:\Spt \rightarrow \Prof_p\left( \Spt \right) ,
\] for $f: \Spc^{p-\pi}\hookrightarrow \Spc$.
\item Let $S$ be a scheme, the functor of \emph{stable absolute $\bA^1$-completion} is the functor
\[
\widehat{\left(\blank\right)}_{S,\bA^1}:= \partial f^*:\Spt \rightarrow \Prof_{S,\bA^1}\left( \Spt \right),
\]
for $f: \Spc_{S,\bA^1} \hookrightarrow \Spc$.
\item Let $S$ be a scheme, the functor of \emph{stable relative $\bA^1$-completion} is the functor
\[\widehat{\left(\blank\right)}_{S, \bA^1, \et}:= \partial f^*:\Spt\left( \Shv\left(S_{\et} \right) \right) \rightarrow \Prof_{S,\bA^1,\et}\left(\Spt \right),
\] for $f:  \left( \Pro\left( \Spc_{S,\bA^1} \right) \right) \hookrightarrow \Shv\left(S_{\et} \right)$.

 \end{enumerate}
\end{definition}

Finally, we discuss the relationship between taking $\Pro$ and $\Spt$. Let $\sC$ be an accessible $\infty$-category with finite limits. Since the functor $j:\sC \hookrightarrow \Pro \left( \sC \right)$ preserves finite limits by Proposition~\ref{prop:finlim}, we obtain a functor $\partial j: \Spt(\sC) \rightarrow \Spt\left(\Pro \left( \sC \right)\right)$ such that the following diagram commutes
\[
\xymatrix{
 \Spt\left(\sC\right) \ar[r]^-{\partial j} \ar[d]_{\Omega_{\sC}^{\infty}} & \Spt\left(\Pro \left( \sC \right)\right) \ar[d]^{\Omega^{\infty}_{\Pro \left( \sC \right)}} \\
\sC \ar[r]^-{j} & \Pro\left( \sC \right).}
\]
By the universal property of $\Pro$, there exists a cofiltered limit-preserving extension of $\partial j$: 
\[
\Pro\left(\partial j\right): \Pro\left(\Spt\left(\sC\right)\right) \rightarrow \Spt\left(\Pro \left( \sC \right)\right).
\]
In other words, we can always bootstrap the completion functors of Definition~\ref{def:stables} to a functor out of pro-spectra.

\subsubsection{Relative \'{e}tale realizations of presheaves of spectra}

Suppose that $S$ is a base scheme, denote by $\PShv_{\Spt} \left(\Sch^{\ft}_S \right):= \Fun(\Sch^{\ft}_S, \Spt)$ the $\infty$-category of presheaves of spectra on finite type $S$-schemes. We have the relative \'{e}tale realization functor of Definition~\ref{defn:rel-et-type} $$\Pi^{S,\et}_{\i}:\Shv_{\et}\left( \Sch^{\ft}_S \right) \to \Pro\left(\Shv\left(S_{\et}\right)\right).$$ 

We denote the right adjoint from Proposition~\ref{prop:right-adjoint} by $R:  \Pro\left(\Shv\left(S_{\et}\right)\right) \rightarrow \Shv_{\et}\left( \Sch^{\ft}_S \right)$

\begin{proposition} \label{prop:stab-et-real} There is a canonical adjunction
\begin{equation} \label{eq:stab-adj}
\Pi^{S,\et}_{\i,\Spt}: \Shv_{\et,\Spt}\left( \Sch^{\ft}_S \right)  \rightleftarrows \Spt\left(\Pro\left(\Shv\left(S_{\et}\right)\right) \right): \partial R,
\end{equation}
such that the following diagram commutes
\begin{equation} \label{eq:susp-pi}
\xymatrix@C=2.5cm{
\Shv_{\et,\Spt}\left( \Sch^{\ft}_S \right)  \ar[r]^-{\Pi^{S,\et}_{\i,\Spt}}  &  \Spt\left( \Pro\left(\Shv\left(S_{\et}\right)\right) \right) \\
\Shv_{\et}\left( \Sch^{\ft}_S \right) \ar[r]^-{\Pi^{S,\et}_{\i,\Spt}} \ar[u]^{\Sigma^{\infty}_{+}} & \Pro\left(\Shv\left(S_{\et}\right)\right)\ar[u]_{\Sigma^{\infty}_{+}}.
}
\end{equation}
\end{proposition}

\begin{proof}  This is an immediate consequence of Lemma~\ref{lem:stab-pro} and the adjunction from Proposition~\ref{prop:right-adjoint}.
\end{proof}

\begin{definition} \label{defn:stab-et-real} The \emph{stable relative \'{e}tale realization functor} is the left adjoint of~\eqref{eq:stab-adj} 
\[\Pi^{S,\et}_{\i,\Spt}: \Shv_{\et,\Spt}\left( \Sch^{\ft}_S \right)  \rightarrow  \Spt \left( \Pro\left( \Shv\left(S_{\et}\right)\right) \right).
\]
\end{definition}

We will also have occasion to consider the stable realization of presheaves (res. Nisnevich sheaves) of spectra. We note that the sheafification functor $$\LL_{\et}: \PShv\left(\Sch^{\ft}\right) \rightarrow \Shv_{\et}\left(\Sch^{\ft}\right)$$ (resp. $\LL_{\et}:\Shv_{\Nis}\left( \Sch^{\ft} \right) \rightarrow \Shv_{\et}\left(\Sch^{\ft}\right)$)  stabilizes to a functor $$\LL_{\et}:  \PShv_{\Spt}(\Sch^{\ft}) \rightarrow \Shv_{\et,\Spt}\left( \Sch^{\ft}_S \right)$$ (resp. $\LL_{\et}:\Shv_{\Nis,\Spt}\left( \Sch^{\ft} \right) \rightarrow \Shv_{\et,\Spt}\left( \Sch^{\ft}_S \right)$). If $E$ is a presheaf of spectra (resp. Nisnevich sheaf of spectra) on $\Sch^{\ft}_S$, then the stable relative \'{e}tale realization of $E$ is understood to be $\Pi^{S,\et}_{\i,\Spt}\left(\LL_{\et}E\right)$, which we abusively write as  $\Pi^{S,\et}_{\i,\Spt}\left(E\right)$; the context will always be clear.

\begin{remark} There is another way to construct the stable \'{e}tale realization functor. We have the \emph{spectral Yoneda embedding} \footnote{note that this functor is, however, not an embedding} $y_{\Spt}: \Sch^{\ft}_S \rightarrow\PShv_{\Spt}\left(\Sch_S\right),$  which is just the composite $$\Sch^{\ft}_S \stackrel{y}{\rightarrow}  \Fun\left(\Sch_S, \Spc \right) \stackrel{\Sigma^{\infty}_+}{\longrightarrow} \Spt \left(  \Fun\left(\Sch^{\ft}_S, \Spc \right) \right) \simeq \PShv_{\Spt}\left(\Sch_S\right).$$
We have the relative \'{e}tale realization functor of Definition~\ref{defn:rel-et-type} $$\Pi^{S,\et}_{\i}:\Shv_{\et}\left( \Sch^{\ft}_S \right) \to \Pro\left(\Shv\left(S_{\et}\right)\right),$$ which we can postcompose with the stabilization functor obtained in Proposition~\ref{prop:sigma-infty}, $$\Sigma^{\infty}_+: \Pro \left(\Shv\left(S_{\et}\right)\right) \rightarrow  \Spt \left( \Pro\left( \Shv\left(S_{\et}\right)\right) \right).$$ We can then define the stable \'{e}tale realization functor as the left Kan extension,
$$
\xymatrix{
\Sch^{\ft}_S \ar[rrr]^{\Sigma^{\infty}_+ \circ \Pi^{S,\et}_{\i}}  \ar[dd]_{y_{\Spt}} & &  & \Spt \left( \Pro\left( \Shv\left(S_{\et}\right)\right) \right)\\
 & & &\\
\PShv_{\Spt} \left(\Sch^{\ft}_S \right)\ar@{.>}[uurrr]^{\Pi^{S,\et}_{\i,\Spt}}& & & .\\
}$$
Since the functor of Definition~\ref{defn:stab-et-real} preserves colimits and agrees with the value of the above functor on $\Sch^{\ft}_S$, we see that both functors are equivalent. This approach is closer to the definition of semi-topological $K$-theory as in \cite{blanc}.
\end{remark}

Lastly, we explain the relationship between $\mathcal{E}_{\infty}$-monoids and stabilization. If $\sE$ is an $\infty$-topos, then the usual ``recognition principle" states there is an equivalence of $\infty$-categories between group-like $\mathcal{E}_{\infty}$-monoids and connective spectra \cite[Theorem 5.2.6.15 and Remarks 5.2.6.12]{higheralgebra}. We will not prove this here for pro-objects in an $\infty$-topos. However, we will factor the diagram~\eqref{eq:susp-pi} through the free $\mathcal{E}_{\infty}$-monoids functor. 

Recall that if $\sC$ is an $\infty$-category with finite products and equipped with the Cartesian monoidal structure \cite[Section 2.4.1.1]{higheralgebra}, the $\infty$-category of \emph{commutative monoids} $\CMon\left(\sC^{\times} \right)$ is the full subcategory of $\Fun\left(\Fin_*, \sC^{\times} \right)$ satisfying the Segal condition (see, for example, \cite[Appendix C.1]{Bachmann:2017aa} and \cite[Section 2.4.2]{higheralgebra} for more details) where $\Fin_*$ is the discrete category of pointed finite sets. By Proposition~\ref{prop:finlim}, for any $\infty$-topos $\cE$, the $\infty$-category $\Pro\left(\cE \right)$ also has finite products. In particular, $\Pro\left(\Shv\left(S_{\et}\right)\right)^{\times}$ is a Cartesian symmetric monoidal $\infty$-category.

\begin{lemma} \label{lem:lax} The functor $\Pi^{S,\et}_{\i}: \Shv_{\et}\left( \Sch^{\ft}_S \right)^{\times} \rightarrow \Pro\left(\Shv\left(S_{\et}\right)\right)^{\times}$ is lax monoidal
\end{lemma}

\begin{proof} Since both source and target are Cartesian symmetric monoidal, any functor is lax symmetric monoidal \cite[Proposition 2.4.1.7]{higheralgebra}; this just comes from the universal property of products.
\end{proof}

\begin{proposition} \label{prop:pro-recog} Let $S$ be a scheme. There is a functor 
\[
\B^{\infty}: \CMon\left(\Pro\left(\Shv\left(S_{\et}\right)\right) \right)^{\times} \rightarrow \Spt\left(\Pro\left( \left(\Shv\left(S_{\et}\right)\right) \right)\right),
\]
such that the following diagram commutes
\[
\xymatrix{
 \CMon\left(\Shv_{\et}\left( \Sch^{\ft}_S \right) ^{\times}\right) \ar[rr]^{\Pi^{S,\et}_{\i}} \ar[d]_{\B^{\infty}} & &  \CMon\left(\Pro\left(\Shv\left(S_{\et}\right)\right)^{\times} \right)\ar[d]^{\B^{\infty}}\\
 \Shv_{\et,\Spt}\left( \Sch^{\ft}_S \right) \ar[rr]^{\Pi^{S,\et}_{\i,\Spt}} & & \Spt\left(\Pro\left(\Shv\left(S_{\et}\right)\right) \right).
}
\]
\end{proposition}

\begin{proof} Let $\star$ be the discrete category with a single object and no nontrivial morphism. We have the following commutative diagram of $\infty$-categories, where all the arrows are fully faithful
\begin{equation} \label{eq:pt-fin}
\xymatrix{
 &  \star \ar[dl]_{i} \ar[dr]^{j} & \\ 
\Fin_*  \ar[rr]^{k} & & \Spc^{\fin}_*
}
\end{equation}
We apply the triangle~\eqref{eq:pt-fin} to the functor
\[
\Fun\left(-, \Pro\left(\Shv\left(S_{\et}\right) \right)\right) \stackrel{R_*}{\longrightarrow} \Fun\left(-, \Shv_{\et,\Spt}\left(\Sch^{\ft} \right)\right)
\]
where $R$ of the right adjoint constructed in Proposition~\eqref{prop:right-adjoint}, to obtain a commutative diagram
\begin{equation} \label{eq:ppd}
\xymatrix{
 & \Pro\left(\Shv\left(S_{\et}\right) \right) \ar[ddd]_{R_*} |!{[d]}\hole    &  \\
\Fun\left(\Fin_*, \Pro\left(\Shv\left(S_{\et}\right) \right) \right)\ar[ddd]_{R_*} \ar[ur]^{i^*} & & \Fun\left( \Spc^{\fin}_*,\Pro\left(\Shv\left(S_{\et}\right) \right)\right) \ar[ddd]_{R_*}  \ar[ll]_(.45){k^*} \ar[ul]_{j^*} \\
& & \\
& \Shv_{\et}\left(\Sch^{\ft} \right)   & \\
\Fun\left(\Fin_*,\Shv_{\et}\left(\Sch^{\ft} \right)\right) \ar[ur]^{i^*} & & \Fun\left( \Spc^{\fin}_*, \Shv_{\et}\left(\Sch^{\ft} \right)\right) \ar[ll]_{k^*} \ar[ul]_{j^*}.
}
\end{equation}
Now, arguing as in Lemma~\ref{lem:stab-pro} with $\Fin_*$ in place of $\Spc^{\fin}_*$ we have that
\begin{itemize}
\item the arrows in~\eqref{eq:ppd} admits left adjoints; we denote them by $i_!, k_!$,  $j_!,$ and $R_!$. (Note also, that when restricting to sheaves, $R_!$ gets identified with $\Pi^{S,\et}_{\i}$ by uniqueness of adjoints.) 
\item The functor $k^*$ (for both  $ \Pro\left(\Shv\left(S_{\et}\right) \right)$ and $\Shv_{\et}\left(\Sch^{\ft} \right)$) takes commutative monoids to spectra by noting that $k^*$ takes excisive functors to functors satisfying the the Segal condition to be a commutative monoid, and thus its left adjoint $k_!$ takes commutative monoids to spectra.
\item By Lemma~\ref{lem:lax} the functor $\Pi^{S,\et}_{\i}$ is lax monoidal and hence induces a functor
\[
\Pi^{S,\et}_{\i}: \CMon\left(\Shv_{\et}\left(\Sch_S^{\ft} \right)^{\times}\right) \rightarrow \CMon\left( \Pro\left(\Shv\left(S_{\et}\right)^{\times} \right) \right).
\]
compatibly with the the left adjoint $i_!$.
\end{itemize}
As a result we obtain the following diagram of left adjoints
\begin{equation} \label{eq:ppd-left}
\xymatrix{
 & \Pro\left(\Shv\left(S_{\et}\right) \right)  \ar[dl]_{i_!} \ar[dr]^{j_!}  &  \\
\CMon\left( \Pro\left(\Shv\left(S_{\et}\right)^{\times} \right) \right)  \ar[rr]^{k_!}  & & \Spt\left(\Pro\left(\Shv\left(S_{\et}\right) \right) \right) \\
& & \\
& \Shv_{\et}\left(\Sch^{\ft} \right)  \ar[dl]_{i_!} \ar[dr]^{j_!} \ar[uuu]^{\Pi^{S,\et}_{\i}}  |!{[uu]}\hole & \\
\CMon\left(\Shv_{\et}\left(\Sch^{\ft} \right)^{\times}\right)  \ar[rr]^{k_!} \ar[uuu]^{R_!} & & \Shv_{\et,\Spt}\left(\Sch^{\ft} \right) \ar[uuu]^{\Pi^{S,\et}_{\i,\Spt}}.
}
\end{equation}
We set the functor $\B^{\infty}$ to be the uppermost $k_!$.
\end{proof}

\subsubsection{Relative \'etale realization functor for $S^1$-motivic spectra} We now have a stable version of Theorem~\ref{thm:main-functor}

\begin{theorem} \label{thm:stable} Let $S$ be a scheme, then there exists a colimit preserving functor 
\[
\Et^S_{\bA^1,\Spt}: \SH^{S^1}\left(S \right) \rightarrow \Spt \left( \Pro\left(S_{\bA^1}\right) \right)
\] 
such that the following diagram commutes
\[
\xymatrix{
 \SH^{S^1}\left(S \right) \ar[rr]^{\Et^S_{\bA^1,\Spt}} & &   \Spt \left( \Pro\left(S_{\bA^1}\right) \right)  \\
\Spc\left( S \right) \ar[rr]^{\Et^S_{\bA^1}} \ar[u]^{\Sigma^{\infty}_+}  &  & \Pro\left(S_{\bA^1}\right)\ar[u]_{\Sigma^{\infty}_+}  .
}
\]
In particular, the value on $\Sigma^{\infty}_{+,S^1}X$ where $X$ is a smooth $S$-scheme is given by $\Sigma^{\infty}_+\widehat{\Pi^{S,\et}_\i\left(X\right)}_{S_{\bA^1}}$. 
\end{theorem}

\begin{proof} The $\infty$-category $\Spc\left(S \right)$ is compactly generated by motivic localization of smooth affine $S$-schemes see, for example \cite[Proposition 2.2]{elso}. Hence it is a differentiable $\infty$-category by \cite[Remark 5.5.2.10]{higheralgebra}. In this case, we may apply directly Lemma~\ref{lem:stab-pro} to the adjunction 
\[
\Et^S_{\bA^1}: \Spc\left(S \right) \rightleftarrows \Pro\left(S_{\bA^1}\right): R_{\bA^1},
\]
from Theorem~\ref{thm:main-functor} and Proposition~\ref{prop:right-adjoint-a1}.

\end{proof}

\subsection{Dwyer-Friedlander \'{e}tale $K$-theory revisited} \label{subsec:revisit}

%
\newcommand{\Zl}{\mathbb{Z}[\frac{1}{\ell}]}
\newcommand{\Vect}{\mathrm{Vect}}

In a series of papers, Friedlander and Dwyer-Friedlander defined \'{e}tale $K$-theory of schemes \cite{friedlander-ekt1}, \cite{friedlander-ekt2}, \cite{dwyer-friedlander} using \'{e}tale homotopy theory. We refine their definition in the relative setting.  Before doing so, let us phrase their definition in our language.  In \cite{dwyer-friedlander}, \'etale $K$-theory is defined in the following way: we have the presheaf 
\begin{eqnarray*}
\Mod^{\mathsf{fgp},\simeq}_{\Zl}: \Aff\Sch^{\ft}_{\Zl} &\rightarrow& \Spc\\
\Spec\,R &\mapsto& \Mod^{\mathsf{fgp},\simeq}_R:=\Mod^{\mathsf{fgp},\simeq}_{\Zl}\left(R\right),
\end{eqnarray*}
assigning spaces of finitely generated projective modules. They then apply the $\ell$-completed absolute \'etale homotopy type functor $\widehat{\left(\Pi^{\et}_\i\left(\Mod^{\mathsf{fgp},\simeq}_{\Zl}\right)\right)}_{\ell}$, obtaining an object in $\Prof_{\ell}\left(\Spc \right)$. By Lemma~\ref{lem:lax} and Proposition~\ref{prop:pro-recog}, one then obtains a spectrum object in $\Prof_{\ell}\left(\Spc \right)$, which we denote by
\[
K^{\DF}_{\ell}:= \B^{\infty}\widehat{\left(\Pi^{\et}_\i\left(\Mod^{\mathsf{fgp},\simeq}_{\Zl}\right)\right)}_{\ell}.
\]

The relative extension of their theory is easily defined in our language. Let $S$ be a base scheme, we have the stack $\Vect^{\simeq}: \Sch^{\ft}_S \rightarrow \Spc$ classifying vector bundles on $S$. It restricts to affine schemes to be the stack $\Mod^{\mathsf{fgp},\simeq}$ classifying finitely generated projective modules. 

\begin{definition} The \emph{relative $DF$-$K$-theory spectrum} of $S$ is defined to as
\[
K^{\DF}_S := \B^{\infty} \left(\mspace{3mu}\Pi^{S,\et}_{\i} \Vect^{\simeq} \mspace{3mu}\right) \in \Spt\left(\Pro\left(\Shv\left(S_{\et}\right)\right) \right). 
\]
\end{definition}

To recover the usual Dwyer-Friedlander spectrum from its relative version, we use the following compatibility result

\begin{lemma} \label{lem:push-b} Let $S$ be a scheme, then the following diagrams commutes:
\begin{equation} \label{eq:push-b}
\xymatrix{
 \CMon\left(\Pro\left(\Shv\left(S_{\et}\right) \right)^{\times}\right) \ar[rr]^{\B^{\infty}} \ar[d]_{\Pro\left(e_!\right)} & &  \Spt\left(\Pro\left( \Shv\left(S_{\et}\right) \right) \right) \ar[d]^{\Pro\left(e_!\right)}  \\
 \CMon\left(\Pro\left(\Spc \right)^{\times}\right)  \ar[rr]^{\B^{\infty}} & & \Spt\left(\Pro\left( \Spc \right) \right).}
\end{equation}

\end{lemma}

\begin{proof} This follows from the same argument as in Proposition~\ref{prop:pro-recog} by applying the functor
\[
\Pro\left(e^*\right)_*: \Fun\left(-,\Pro\left(\Spc\right)\right) \rightarrow \Fun\left(-,\Pro\left(\Shv\left(S_{\et}\right)\right)\right)
\]
to the triangle~\eqref{eq:pt-fin}.
\end{proof}

We have the stable version of the profinite completion functor defined in Definition~\ref{def:stables}, $\widehat{\left(\blank\right)}_{\ell}:\Spt \rightarrow \Spt\left( \Prof_{\ell}\left(\Spc\right)\right)$.  By the same argument as in Lemma~\ref{lem:push-b} there is a canonical equivalence $\widehat{\left(\blank\right)}_{\ell} \circ \B^{\infty} \simeq \B^{\infty} \widehat{\left(\blank\right)}_{\ell}:\CMon\left(\Pro\left(\Spc \right)^{\times}\right) \rightarrow \Spt\left( \Prof_{\ell}\left(\Spc\right)\right)$.

\begin{proposition} \label{prop:df-rel} There is a canonical equivalence $\widehat{\left(\blank\right)}_{\ell} \circ \Pro\left(e_!\right)\left( K^{\DF}_{\bZ[\frac{1}{\ell}]} \right) \simeq K^{\DF}_{\ell}$.
\end{proposition}

\begin{proof} This follows by the definition of $K^{\DF}_S$, Lemma~\ref{lem:push-b} and Remark~\ref{rmk:recover}.
\end{proof}

\subsection{Dwyer-Friedlander $K$-theory of Categories} \label{subsec:noncomm}

In this section, we apply the machinery that we have developed to extend the theory of \'{e}tale $K$-theory  as first developed by Dwyer and Friedlander to non-commutative schemes. This is the section where our efforts in constructing a functor out of the motivic homotopy and the $S^1$-stable motivic homotopy category pays off; we exploit techniques and results from motivic homotopy theory to prove the expected properties of this new invariant.

\subsubsection{Semi-topological $K$-theory of dg-categories after Blanc and Antieau-Heller} We are about to define the $\ell$-adic analogue of \emph{semi-topological $K$-theory} of dg-categories. This invariant was first constructed by Blanc \cite{blanc}. In the setting of complex varieties these invariants were introduced by Friedlander and Walker in \cite{fw1}, \cite{fw2}.

 For motivation, and for later use, let us briefly recall their construction. Let $\Cat^{\perf}_{\infty}$ be the (large) $\infty$-category of small idempotent complete stable $\infty$-categories. This stable $\infty$-category is symmetric monoidal under the Lurie tensor product (see, for example, \cite[Section 3]{bgt}), which we will use to discuss $\mathcal{E}_{\infty}$-monoids and modules over them.  We have the functor of \emph{algebraic $K$-theory}  $K:\Cat^{\perf}_{\infty} \rightarrow \Spt$ (as in \cite[Section 9]{bgt}), and the functor of \emph{connective algebraic $K$-theory} $K^{\conn}:\Cat^{\perf}_{\infty} \rightarrow \Spt$ (as in~\cite[Section 7]{bgt}).
  
 Each functor is the universal localizing (as in~\cite[Definition 8.1]{bgt}) and additive (as in~\cite[Definition 6.1]{bgt}) invariant respectively. Suppose that $X$ is a scheme \footnote{for definitions $X$ could easily be a derived scheme.}; we have the small, idempotent-complete stable $\infty$-category $\Perf_X$ of perfect complexes on $X$; this is a symmetric monoidal stable $\infty$-category whose tensor product commutes with colimits in each variable. Consequently $\Perf_X \in \CAlg\left(\Cat^{\perf, \times}_{\infty}\right)$, i.e., it is a presentably symmetric monoidal stable $\infty$-category. We define the $\infty$-category of \emph{$X$-linear stable $\infty$-categories} to be 
 \[
 \Cat_{X} := \Mod_{\Perf_X}\left( \Cat^{\perf,\times}_{\infty} \right).
 \]
When $X = \Spec\,R$ is affine, we write $\Cat_R := \Cat_X$. 


 \begin{construction} \label{constrct:twist}
Suppose that $X$ is a fixed scheme, and we are provided with a functor
\[
F: \Cat_{X} \rightarrow \Spt.
\]
The examples that we care about are taking $F$ to be the algebraic $K$-theory functor, and its connective variant, precomposed with the forgetful functor $\Cat_{X} \rightarrow \Cat^{\perf}_{\infty}$. 

Suppose that $\sC \in \Cat_{X}$. We define a new functor
\[
\underline{F}(\sC): \Sch^{\ft,\op}_{X} \rightarrow \Spt, T \mapsto F(\Perf_T \otimes \sC),
\]
where the symbol $\otimes$ indicate the tensor product of modules over an $\mathcal{E}_{\infty}$-algebra object. If $F$ is $K$-theory or connective $K$-theory, we call the presheaf $\underline{K}(\sC)$ (resp. $\underline{K}^{\conn}(\sC)$) \emph{$\sC$-twisted $K$-theory} (resp. \emph{$\sC$-twisted connective $K$-theory}). 
\end{construction}

For the purposes of semi-topological $K$-theory, we consider $X = \Spec\,\mathbb{C}$.  Recall that there is a \emph{Betti realization functor}
\[
\mathrm{Be}: \PShv_{\Spt}\left(\Sch^{\ft}_{\mathbb{C}}\right) \rightarrow \Spt
\]
which is a colimit preserving functor whose values on (suspension spectra) of representables are given by
\[
\mathrm{Be}(\Sigma^{\infty}_+ j(X)) \simeq \Sigma^{\infty}_+ X\left(\mathbb{C}\right)
\]
where $X\left(\mathbb{C}\right)$ is the \emph{analytification of $X$} --- it is a topological space whose points are the $\mathbb{C}$-points of $X$ and given the analytic topology; see \cite[Equation 1]{antieau-heller} for a quick definition or \cite[Section 3]{blanc} for a more extensive discussion.

\begin{definition} Let $\sC \in \Cat_{\mathbb{C}}$, then the \emph{semi-topological $K$-theory of $\sC$} (resp. \emph{connective semi-topological $K$-theory of $\sC$}) is the spectrum
$K^{\st}\left( \sC \right) := \mathrm{Be}\left(\underline{K}\left(\sC\right)\right)$ (resp. $\mathrm{Be}\left(\underline{K}^{\conn}\left(\sC\right)\right)$).
\end{definition}

Some of the salient properties of $K^{\st}\left( \sC \right)$ are as follows:

\begin{enumerate}
\item there is a canonical equivalence of spectra $K^{\st}\left(\mathbb{C}\right) \simeq ku$ where $ku$ is the connective topological $K$-theory spectrum \cite[Theorem 4.5]{blanc}.
\item There is a Chern character map $K^{\st}\left( \sC \right) \rightarrow HP\left(\sC \right)$ where $HP$ denotes periodic cyclic homology \cite[Section 4.4]{blanc}.
\item There is a canonical equivalence $K\left( \sC \right)/n \simeq K^{\st}\left( \sC \right)/n$, i.e., the torsion part of semi-topological $K$-theory and algebraic $K$-theory agrees \cite[Theorem 3.3]{antieau-heller}.
\end{enumerate}

We will soon prove analogues of (1) and (3), while postponing (2) and further explorations of this Chern character map to a sequel. We note that the $\bA^1$-localized version of this theory is representable by a motivic spectrum; when $\sC = \Perf_X$ we are considering Weibel's homotopy $K$-theory where representability is well-known (see, for example, \cite{cisinski-desc}).

\begin{proposition} \label{prop:representability} Let $S$ be a base scheme. The $\bA^1$-localization of the restriction of $\LL_{\bA^1}\underline{K}(\sC)$ to smooth $S$-schemes is representable by a motivic spectrum, i.e., there is a motivic spectrum $\KGL(\sC) \in \SH(R)$ and a functorial equivalence of spectra
\[
\map_{\SH(R)}(\Sigma^{\infty}_TX_+, \KGL(\sC)) \simeq \LL_{\bA^1}\underline{K}(\sC) \mid_{\Sm_S}(X),
\]
for any $X \in \Sm_R$.
\end{proposition}

\begin{proof} This follows by the same argument in \cite[Proposition 3.2]{antieau-heller}, noting that the fact that they are working over the complex numbers is not used for the proof of \emph{loc.cit}.
\end{proof}

\subsubsection{The construction}

\newcommand{\fp}{\mathbb{F}_p}

For the rest of the paper we work over some base commutative ring $R$; we write $\Sch^{\ft}_R$ (resp. $\Cat_R$) instead of $\Sch^{\ft}_{\Spec\,R}$ (resp. $\Cat_{\Spec\,R}$) and subcategories thereof. Following Construction~\ref{constrct:twist}, if $\sC$ is an $R$-linear stable $\infty$-category, we have the presheaves $ \underline{K}\left(\sC\right)$ and $\underline{K}^{\conn}\left(\sC\right)$ of spectra and connective spectra respectively.
Following the discussion on semi-topological $K$-theory above, we define some new natural invariants of $R$-linear stable $\infty$-categories.


\begin{definition} \label{def:df-k} Let $\sC \in \Cat_{R}$. We define the following invariants:
\begin{enumerate}
\item The \emph{Dwyer-Friedlander $K$-theory} of $\sC$ is defined to be 
\[K_{R}^{\DF}(\sC):= \Pi^{R,\et}_{\i,\Spt} \LL_{\et} \underline{K}(\sC) \in \Spt \left( \Pro\left( \Shv\left(S_{\et}\right)\right) \right).
\]
\item The \emph{$\bA^1$-local Dwyer-Friedlander $K$-theory} of $\sC$ is defined to be 
\[
K_{R,\bA^1}^{\DF}(\sC):= \Et^R_{\bA^1,\Spt}\LL_{\mot}\underline{K}(\sC)  \in \Prof_{S,\bA^1,\et}\left( \Spt \right).
\]
\end{enumerate}
\end{definition}

We will also consider the mod-$n$ variants of the above theories which we denote by
\[
K_{R}^{\DF}(\sC)/n := \Pi^{R,\et}_{\i,\Spt}\left( \LL_{\et}\underline{K}(\sC)/n  \right)
\]
and
\[
K_{R,\bA^1}^{\DF}(\sC)/n:= \Et^R_{\bA^1,\Spt}\left( \LL_{\mot}\underline{K}(\sC)/n \right).
\]

\begin{remark} Just as in \cite[Definition 4.1]{blanc}, we can define the connective variants of these theories, which we denote by $K_{R}^{\DF,\conn}(\sC)$ and $K_{R,\bA^1}^{\DF,\conn}\left(\sC\right)$
\end{remark}

The next proposition summarizes the basic properties of $K^{\DF}_R$

\begin{proposition} \label{lem:basics-df} Let $\sC \in \Cat_R$, then 
\begin{enumerate}
\item The presheaf $\underline{K}\left(\sC\right):\Sch_R^{\ft,\op} \rightarrow \Spt$ is a Nisnevich sheaf.
\item The presheaf $\LL_{\bA^1}\underline{K}\left(\sC\right):\Sch_R^{\ft,\op} \rightarrow \Spt$ is a cdh sheaf.
\item The canonical map $\LL_{\bA^1}\underline{K}\left(\sC\right) \rightarrow \LL_{\mot}\underline{K}\left(\sC\right)$ induces an equivalence
\[
\widehat{\left(\mspace{3mu}\Pi^{R,\et}_{\i,\Spt} \LL_{\bA^1}\underline{K}\left(\sC\right) \mspace{3mu}\right)}_{S_{\bA^1}} \simeq K_{\bA^1,R}^{\DF}\left(\sC\right).
\]
\item Suppose that $n$ is invertible in $R$, then the canonical map $\underline{K}\left(\sC\right)/n \rightarrow \LL_{\bA^1}\underline{K}(\sC)/n$ induces an equivalence
\[
K_{R}^{\DF}\left(\sC\right)/n \simeq K_{R,\bA^1}^{\DF}\left(\sC\right)/n.
\]
\end{enumerate}
\end{proposition}

\begin{proof} For (1), the presheaf $\underline{K}\left(\sC\right)$ is the restriction of a localizing invariant to schemes of finite type, hence it is a Nisnevich sheaf by \cite[Lemma A.1]{tamme-land}. Point (2) comes from the fact that $\LL_{\bA^1}\underline{K}\left(\sC \right)$ is representable in $\SH(k)$ by Proposition~\ref{prop:representability} and \cite[Proposition 3.7]{cisinski-desc}.
The point of (3) is that for any Nisnevich sheaf of spectra $F$, the $\LL_{\bA^1}(F)$ is already a Nisnevich sheaf whence $\LL_{\bA^1}F \simeq \LL_{\mot}F$ by Proposition~\ref{prop:la1-lnis}, whence the claim is immediate (recall that there is an implicit \'etale sheafification to apply the realization functor). Point (4) follows from~\cite[Theorem 1.2]{tab-1} which asserts that the functor on $\Cat_R$ given by $\sC \mapsto K(\sC)$ is $\bA^1$-invariant as was proved by Weibel in the case of associative rings in \cite[Proposition 1.6]{weib-kh}.
\end{proof}

Here is a first computation --- the Dwyer-Friedlander $K$-theory of the unit object in $\Cat_k$; compare with \cite[Theorem 4.5]{blanc}. Over a base scheme $S$, we denote by $\Grass_S$ the infinite Grassmannian. This is the ind-scheme $\Grass_S := \colim_{n,k} \Grass_S\left(n,k \right)$ where each $\Grass_S\left(n, k\right)$ classifies locally free quotients of $\mathscr{O}^n$ of rank $n-k$. Since each $\Grass_S\left(n, k\right)$ is a smooth scheme, its relative \'etale homotopy type is explicit. The next theorem expresses the connective $\bA^1$-local Dwyer-Friedlander as ($\B^{\infty}$ of) the relative \'etale realization of the infinite Grassmannian.

%
%
%
%

\begin{proposition} \label{thm:point}  There is a canonical equivalence in $\Prof_{R,\bA^1,\et}\left( \Spt \right)$
\[
 K_{R,\bA^1}^{\DF,\conn}(\Perf_R) \simeq \B^{\infty} \left( \Et^S_{\bA^1} \left( \Grass_{R} \times \bZ \right) \right).
\]
\end{proposition}

\begin{proof} By \cite[Proposition 3.10]{morel-voevodsky} there is an equivalence in $\Spc(S)$
\[
\LL_{\mot}\left(\Grass_S \times \bZ\right) \simeq \LL_{\mot}\left( \mathrm{BGL}_{\infty} \times \bZ \right),
\]
whence $\LL_{\mot}\left(\Grass_S \times \bZ\right)$ is a commutative monoid in $\Spc(S)^{\times}$.  In fact, this equivalence holds in $\underline{\Spc\left( S \right)}$ the $\infty$-category of $\bA^1$-invariant Nisnevich sheaves on $\Sch^{\ft}_S$, i.e., motivic homotopy theory built from $\Sch^{\ft}_S$ since both $\Grass$ and $\mathrm{BGL}$ are colimits of smooth schemes. Therefore we have the computation

\begin{eqnarray*}
 K_{R,\bA^1}^{\DF,\conn}(\Perf_R) & = & \Et^R_{\bA^1, \Spt}\LL_{\mot}\underline{K}^{\conn}\left( \Perf_R \right)\\
 & \simeq & \Pi^{k,\et}_{\i,\Spt}\LL_{\mot}\B^{\infty}\LL_{\mot}\left( \mathrm{BGL_{\infty}}^+ \times \bZ \right)\\
 & \simeq & \Pi^{k,\et}_{\i,\Spt}\B^{\infty}\LL_{\mot}\left( \mathrm{BGL}_{\infty} \times \bZ \right)\\
& \simeq & \Pi^{k,\et}_{\i,\Spt}\B^{\infty}\LL_{\mot}\left(\Grass_R \times \bZ\right)\\
& \simeq & \B^{\infty} \Pi^{R,\et}_{\i,\Spt}\LL_{\mot}\left(\Grass_R \times \bZ\right)\\
\end{eqnarray*}

Here, the first line is by definition, the second equivalence follows from the fact that connective $K$-theory is a Nisnevich sheaf of \emph{connective spectra} \footnote{We remark that $K^{\conn}$ is not a Nisnevich sheaf of spectra, but is a Nisnevich sheaf of connective spectra; indeed, this follows because the truncation functor preserves limits.} and that connective $K$-theory of affine schemes can be computed using the $+$-construction of $\mathrm{BGL}_{\infty}$. The third equivalence is the well known $\LL_{\bA^1}$-equivalence between $\mathrm{BGL}(R)$ and $\mathrm{BGL}(R)^+$ (see, for example, \cite{antieau-e}), the fourth equivalence comes from the equivalence discussed in the above paragraph, and the last equivalence follows from Proposition~\ref{prop:pro-recog}.

\end{proof}

\subsection{Blanc's conjecture for $K^{\DF}$} \label{sec:blancdf} In this final section, we prove Theorem~\ref{thm:agreement} which should be viewed as an analog of Blanc's conjecture/Antieau-Heller theorem \cite[Theorem 3.3]{antieau-heller} for $K^{\DF}$. This theorem states that, over a field, $K_k^{\DF}\left(\sC\right)/n$ is equivalent to $K\left(\sC\right)/n$ where $n$ is prime to the characteristic of $k$, a property also enjoyed by semi-topological $K$-theory of categories over the complex numbers.  After \cite[Corollary 4.4]{dave-etale}, which compares the \'etale homotopy type of a scheme (in fact a higher stack) with its analytification up to profinite completion (which generalizes the comparison theorem of Artin-Mazur \cite{ArtinMazur}), it is easy to see Theorem~\ref{thm:agreement} recovers \cite[Theorem 3.3]{antieau-heller} over the complex numbers.

To begin, we need the following crucial input: if $\ell$ is a prime invertible in $R$, then the spectrum $K_{R}^{\DF}\left(\sC\right)/\ell$ only depends on its restriction to smooth $R$-schemes. To formulate this, consider the stabilized adjunction
\[
j_!: \Shv_{\Nis,\Spt}\left(\Sm_R\right) \leftrightarrows \Shv_{\Nis,\Spt}\left(\Sch^{\ft}_R\right): j^*.
\]
The functor $j^*$ is given by restricting a presheaf to smooth $R$-schemes, while $j_!$ is obtained by left Kan extension. The functor $\Pi^{R,\et}_{\i,\Spt}$ of Definition~\ref{defn:stab-et-real} restricts along (the stabilized version of) $j^*$ to give a functor 
\[
\Pi^{R,\et}_{\i,\Spt}\mid_{\Sm}:  \Spt \left( \Shv_{\Nis}\left( \Sch^{\ft}_R \right) \right) \rightarrow \Spt \left( \Pro\left( \Shv\left(R_{\et}\right)\right) \right),
\] 
which factors through the $\infty$-category of \'etale sheaves of spectra. In general, there is no reason for the map $\Pi^{R,\et}_{\i,\Spt} j^*F \rightarrow \Pi^{R,\et}_{\i,\Spt}F$ to be an equivalence for an arbitrary Nisnevich sheaf of spectra $F$.

%
%
%
%
%

\begin{proposition} \label{prop:sm-res} Let $k$ be a field and $\ell$ be prime such that $1/\ell \in k$, then the canonical map
\[
\left( \Pi^{S,\et}_{\i,\Spt}\mid_{\Sm} j^*\underline{K}\left(\sC\right) \right)/\ell \rightarrow K_{k}^{\DF}(\sC)/\ell,
\]
is an equivalence.
\end{proposition}

\begin{proof} By \'etale descent we may assume that $k$ is a perfect field. The key input is $\ell$dh-descent of Kelly \cite{kelly}; see \cite[Definition 2.1.11]{kelly} for a definition. In fact we will prove a more precise claim which we now formulate. We have an adjunction at the level of Nisnevich sheaves of spectra
\[
j_!: \Shv_{\Nis,\Spt}\left(\Sm_k\right) \rightleftarrows \PShv_{\Nis,\Spt}\left(\Sch_k^{\ft}\right): j^*.
\]
The $K$-theory sheaf is a Nisnevich sheaf of $\mathcal{E}_{\infty}$-ring spectra, whence we have an induced adjunction 
\[
j_!: \Mod_{K}\left(\Shv_{\Nis,\Spt}\left(\Sm_k\right)\right) \rightleftarrows \Mod_{K}\left(\Shv_{\Nis,\Spt}\left(\Sch^{\ft}_k\right)\right): j^*.
\]
We will in fact work with the $\ell$-localized version
\begin{equation} \label{eq:s1-kl}
j_!: \Mod_{K_{(\ell)}}\left(\Shv_{\Nis,\Spt}\left(\Sm_k\right)\right) \rightleftarrows \Mod_{K_{(\ell)}}\left(\Shv_{\Nis,\Spt}\left(\Sch^{\ft}_k\right)\right): j^*.
\end{equation}
Now, $\underline{K}\left(\sC\right)/\ell$ is an object of $\Mod_{\underline{K}_{(\ell)}}\left(\Shv_{\Nis,\Spt}\left(\Sch^{\ft}_k\right)\right)$ and we claim:
\begin{itemize} 
\item The counit map
\[
j_!j^*\underline{K}\left(\sC\right)/\ell \rightarrow \underline{K}\left(\sC\right)/\ell 
\]
induces an equivalence of $K_{(\ell)}$-modules.
\end{itemize} From Gabber's version of alterations and the fact that it is an equivalence when restricted to smooth $k$-schemes, we see that the counit map $j_!j^*\underline{K}(\sC) \rightarrow \underline{K}(\sC)$ is an $\ell$dh-local equivalence; see \cite[Corollary 4.6]{hoyois2013motivic} for a version of the statement that we need. To bridge the gap between an $\ell$dh-local equivalence and a $\Nis$-local equivalence we need to invoke motivic homotopy theory. In fact, using Proposition~\ref{lem:basics-df}.4, we may replace $\underline{K}\left(\sC\right)/\ell$ with $\LL_{\bA^1}\underline{K}(\sC)/\ell$. Using again Gabber's alterations and the fact that the it is an equivalence when restricted to smooth $k$-schemes, the map
\begin{equation} \label{eq:la1}
j_!j^*\LL_{\bA^1}\underline{K}\left(\sC\right)/\ell \rightarrow \LL_{\bA^1}\underline{K}\left(\sC\right)/\ell
\end{equation}
 of $\LL_{\bA^1}K$-modules is an $\ell$dh-local equivalence. Hence the claim above reduces to the following one:
 \begin{itemize}
 \item Both   $j_!j^*\LL_{\bA^1}\underline{K}\left(\sC\right)/\ell$ and $ \LL_{\bA^1}\underline{K}\left(\sC\right)/\ell$ are $\ell$dh sheaves on $\Sch^{\ft}_k$.
 \end{itemize}  
We use \cite[Lemma 4.8]{hoyois2013motivic} to show this. The first hypothesis of \emph{loc. cit} is verified using  Proposition~\ref{lem:basics-df} (for $j_!j^*\LL_{\bA^1}\underline{K}$ use again that it is an $\LL_{\bA^1}K$-module and the same argument as in \emph{loc. cit}), while the second is verified by  fact that the cdh cohomological dimension of a scheme is bounded above by its Krull dimension \cite{sv-bk}. We verify the last assumption: for all $X \in \Sch^{\ft}_k$ the canonical map 
\[
H^s_{cdh}(X; a_{cdh}\pi_t\LL_{\bA^1}\underline{K}\left(\sC\right)/\ell) \rightarrow H^s_{\ell dh}(X; a_{\ell dh}\pi_t\LL_{\bA^1}\underline{K}\left( \sC \right)/\ell) 
\]
is an isomorphism for all $p,q \in \bZ$ (and similarly for $j_!j^*\LL_{\bA^1}\underline{K}\left(\sC\right)/\ell$). A criterion for this is furnished in \cite[Theorem 4.7]{hoyois2013motivic}: we need $\pi_t\LL_{\bA^1}\underline{K}\left(\sC\right)/\ell$ (resp. $\pi_tj_!j^*\LL_{\bA^1}\underline{K}\left(\sC\right)/\ell$) to be a presheaf of $\bZ_{(\ell)}$-modules with transfers. It is obviously a $\bZ_{(\ell)}$-module since we are working mod $\ell$. To verify that it has transfers we use  \cite[Theorem 4.14]{hoyois2013motivic} which requires that $\pi_t\LL_{\bA^1}K/\ell$ (resp. $\pi_tj_!j^*\LL_{\bA^1}K/\ell$) (1) admits the structure of \emph{traces} ( in the sense of \cite[Definition 2.1.3]{kelly}), (2) is invariant for the map $X_{\red} \rightarrow X$ and (3) its restriction $\Sm_k$ is unramified in the sense of Morel \cite[Definition 2.1]{morel-book}.

Statement (2) is true for any cdh sheaf of spectra since it is excisive for the abstract blowup square
\[
\xymatrix{
\emptyset \ar[d] \ar[r] & X_{\red} \ar[d] \\
\emptyset \ar[r] & X.
}
\]
To verify (3), recall that the Nisnevich homotopy sheaves of any strictly $\bA^1$-invariant sheaf on $\Sm_k$ is unramified \cite[Lemma 6.4.4]{morel-conn}. Now, the Nisnevich sheaves of abelian groups on $\Sm_k$ given by $a_{\Nis}\pi_t\LL_{\bA^1}\underline{K}\left(\sC\right)/\ell \mid_{\Sm_k}$ and $a_{\Nis}j_!j^*\pi_t\LL_{\bA^1}\underline{K}\left(\sC\right)/\ell\mid_{\Sm_k}$ are isomorphic since $j_!$ is fully faithful by \cite[Proposition 4.4.4.8]{htt} while the former is unramified by \cite[Theorem 9]{morel-book} since they are $\bA^1$-homotopy sheaves.

It remains to construct a structure of traces on $\LL_{\bA^1}\underline{K}\left(\sC\right)/\ell$ and $j_!j^*\LL_{\bA^1}\underline{K}\left(\sC\right)/\ell$; the structure on the latter is induced from the former by naturality of the structure of traces on the former. The construction then follows from the construction of   \cite[Proposition 4.18]{hoyois2013motivic}, $\LL_{\bA^1}\underline{K}\left(\sC\right)/\ell$: given a finite flat morphism, $f: Y \rightarrow X$ in $\Sch^{\ft}_k$, we get a morphism in $\Cat^{\perf}_k$
\[
f_* \otimes \id:\Perf_X \otimes_{k} \sC \rightarrow \Perf_Y \otimes_k \sC,
\]
inducing the trace map
\[
f_*:\underline{K}(\sC)(X) \rightarrow \underline{K}(\sC)(Y).
\]
The verification that this is indeed a structure of traces follows exactly as in the first paragraph of \emph{loc. cit}.

To prove our proposition, it suffices by Proposition~\ref{lem:basics-df}.4 to prove the $\bA^1$-local version of the statement, i.e, the map 
\[
\left( \Et^R_{\bA^1,\Spt}\mid_{\Sm} \LL_{\bA^1}j^*\underline{K}\left(\sC\right) \right)/\ell \rightarrow K_{k,\bA^1}^{\DF}(\sC)/\ell.
\] is an equivalence. This is a matter of parsing the equivalences
\begin{eqnarray*}
K_{k,\bA^1}^{\DF}(\sC)/\ell & \simeq & \left( \Et^R_{\bA^1,\Spt} \LL_{\bA^1}\underline{K}/\ell\left(\sC\right) \right)\\
& \simeq & \left( \Et^R_{\bA^1,\Spt} j_!j^*\LL_{\bA^1}\underline{K}/\ell\left(\sC\right) \right)\\
& \simeq & \left( \Et^R_{\bA^1,\Spt}\mid_{\Sm} j^*\LL_{\bA^1}\underline{K}/\ell\left(\sC\right) \right)\\
& \simeq & \left( \Et^R_{\bA^1,\Spt}\mid_{\Sm} \LL_{\bA^1}j^*\underline{K}\left(\sC\right) \right)/\ell,
\end{eqnarray*}
where the work of the previous paragraphs goes into the second equivalence and the third equivalence follows from the commutativity of the following diagram of left adjoints
\[
\xymatrix{
\Shv_{\et,\Spt}(\Sm_k) \ar[rr]^{j_!} \ar[dr]_{\Pi^{S,\et}_{\i,\Spt}\mid_{\Sm}} & & \Shv_{\et,\Spt}(\Sch^{\ft}_k) \ar[dl]^{\Pi^{S,\et}_{\i,\Spt}}\\
& \Spt\left(\Pro\left(\Shv\left(S_{\et}\right)\right) \right) & 
}
\]
which follows from its unstable version and Proposition~\ref{prop:stab-et-real}.
%

\end{proof}

Lastly, we prove an analogue of Blanc's conjecture for semi-topological $K$-theory, see \cite[Theorem 3.3]{antieau-heller}.

\begin{theorem} \label{thm:agreement} Let $k$ be a field.  Let $e: \Shv\left(k_{\et}\right) \rightarrow \Spc$ be the terminal geometric morphism which induces a stabilized adjunction
\[
\partial e^*: \Spt \rightleftarrows \Spt \left( \Shv\left(k_{\et} \right) \right): \partial e_*.
\]
There is a canonical map 
\[
\partial e^*K\left(\sC \right) \rightarrow K_{k}^{\DF}\left(\sC\right),
\] such that when $n$ is invertible in $k$ we have a canonical equivalence
\[
\partial e^*K\left(\sC \right)/n \simeq K_{k}^{\DF}\left(\sC\right)/n.
\]
In particular if $k$ is a separably closed field, we have a canonical equivalence
\[
K\left( \sC \right)/n \simeq K_{k}^{\DF}\left(\sC\right)/n.
\]
\end{theorem}

\begin{proof} 

 Without loss of generality we may assume that $n$ is a prime. Consider the geometric morphism $e: \Shv(k_{\et}) \rightarrow \Spc$. We have the stabilized functor $\partial e^*: \Spt \rightarrow \Spt \left( \Shv(k_{\et})  \right)$. We also have the geometric morphism $e': \Shv_{\Nis}\left(\Sch^{\ft}_k\right) \rightarrow \Spc$ and the stabilized functor $\partial e'^*: \Spt \rightarrow \Shv_{\Nis,\Spt}\left(\Sch^{\ft}_k \right)$. There is a map $c':\partial e'^* K\left( \sC\right) \rightarrow \underline{K}\left( \sC \right)$ in $\Shv_{\Nis,\Spt}\left(\Sch^{\ft}_k\right)$ which is the adjoint of the identity map in $\Spt$: $K\left( \sC \right) \rightarrow \partial e_* K\left( \sC \right) \simeq K\left(k \otimes_k \sC \right) \simeq K\left( \sC \right)$; We apply  $\Pi^{R,\et}_{\i,\Spt}$ to $c'$ obtain the map
 \[
 c:=  \Pi^{R,\et}_{\i,\Spt} \LL_{\et}c': \Pi^{R,\et}_{\i,\Spt}\LL_{\et}\partial e'^* K\left( \sC\right) \simeq \partial e^*K\left( \sC \right) \rightarrow K_{k}^{\DF}(\sC).
 \] In other words, $c$ is a map from the constant sheaf of spectra with value $K\left( \sC \right)$ to $K_{k}^{\DF}\left(\sC\right)/n$.

 After Proposition~\ref{prop:sm-res}, it suffices to prove the claim for $K$-theory restricted to smooth schemes, i.e., we may replace instances of $K$ with $j^*K$ in the notation of the proof of Proposition~\ref{prop:sm-res}. Furthermore since we are proving a claim mod $n$, we are free to $\bA^1$-localize after Proposition~\ref{lem:basics-df}.4 so that we may work with $\LL_{\bA^1}\underline{K}\left(\sC\right)/n$ instead. The upshot is that we have the full machinery of motivic homotopy theory in this situation.  The theorem thus boils down to the following claim

%

\begin{itemize}
\item the induced map $c/n:  \LL_{\et}\LL_{\bA^1}\partial e'^*K\left( \sC \right)/n \rightarrow \LL_{\et}\LL_{\bA^1}\underline{K}\left( \sC \right)/n$ is an equivalence in $\Shv_{\et,\Spt}\left( \Sm_k \right)$.
\end{itemize} 

By \'etale descent we may replace $k$ with its separable closure. To this end, it suffices to prove that for any Henselian local $k$-algebra $R$ with maximal ideal $\mathfrak{m}$ (in fact strictly Henselian is actually enough) which is essentially smooth over $k$ (remember that we have restricted $K$ to smooth $k$-schemes!) the canonical map
\[
\LL_{\bA^1}K\left( \sC \right)/n \simeq \LL_{\bA^1}K\left(\Perf_{R/\mathfrak{m}} \otimes \sC \right)/n  \rightarrow \LL_{\bA^1}K\left( \Perf_R \otimes_{k} \sC \right)/n
\]
 is an equivalence.  Now, since $\LL_{\bA^1}\underline{K}\left(\sC\right)$ is representable by a motivic spectrum by Proposition~\ref{prop:representability}, this follows immediately from from rigidity for presheaves representable by a motivic spectrum; more precisely we apply \cite[Corollary 1.3]{an-dru} to get an equivalence $\underline{K}\left( \sC \right)(R/\mathfrak{m}) \simeq \underline{K}\left(\sC\right)(R)$. Indeed, we can use \emph{loc.cit} because we have restricted ourselves to presheaves on $\Sm_k$.
 

\end{proof}

\bibliography{profinite}
\bibliographystyle{hplain}

\end{document}